\documentclass[11pt, a4paper]{amsart}
\usepackage[final]{changes}
\usepackage{hyperref}
\usepackage{tabularx,booktabs}
\usepackage{caption}
\usepackage{amsmath}
\usepackage{amsfonts}
\usepackage{amscd}
\usepackage{amsthm}
\usepackage{amssymb,mathtools}
\usepackage{latexsym}

\usepackage{enumerate}
\usepackage{multirow}

\usepackage{pdflscape}

\usepackage[left=2.9cm, right=2.9cm, top=1in, bottom=1in]{geometry}

\usepackage{tikz, cases}
\usetikzlibrary{shapes.geometric, arrows}
\numberwithin{equation}{section}
\usepackage[linesnumbered,ruled,vlined]{algorithm2e}

\DeclareMathOperator*{\argmin}{\ensuremath{arg\,min}}

\DeclareMathOperator*{\median}{\ensuremath{median}}

\DeclareMathOperator*{\sgn}{\ensuremath{Sgn}}

\DeclareMathOperator*{\vol}{\ensuremath{vol}}

\def\wbar{\accentset{{\cc@style\underline{\mskip8mu}}}}

\makeatother

\renewcommand{\vec}[1]{\mbox{\boldmath \small $#1$}}

\newcommand{\power}{\ensuremath{\mathcal{P}}}
\newcommand{\anti}{\ensuremath{\mathrm{anti}}}
\DeclareMathOperator*{\opt}{\ensuremath{opt}}
\DeclareMathOperator*{\tc}{\ensuremath{TC}}
\DeclareMathOperator*{\bc}{\ensuremath{BC}}

\newcommand{\R}{\ensuremath{\mathbb{R}}}

\newcommand{\gen}{\mathrm{genus}}

\newcommand{\norm}[1]{\Vert #1\Vert_\infty}
\newcommand{\abs}[1]{\vert #1\vert}
\newcommand{\set}[1]{\{ #1\}}
\renewcommand{\vec}[1]{\mbox{\boldmath$#1$}}
\newcommand{\normone}[1]{\Vert #1\Vert_1}
\def\sign{\operatorname{sign}}
\newcommand{\vK}{\vec{K}}
\newcommand{\vI}{\vec{I}}
\newcommand{\Span}{\ensuremath{\mathrm{span}}}
\theoremstyle{plain}
\newtheorem{theorem}{Theorem}

\newtheorem{definition}{Definition}
\newtheorem{lemma}{Lemma}
\newtheorem{remark}{Remark}

\newtheorem{cor}{Corollary}
\newtheorem{proposition}{Proposition}

\allowdisplaybreaks

\title[fundamental graph cut] {Equivalent spectral theory for fundamental graph cut problems}

\author{Sihong Shao}
\email{sihong@math.pku.edu.cn}
\address{CAPT, LMAM and School of Mathematical Sciences,  Peking University, Beijing 100871, China}

\author{Chuan Yang}
\email{chuanyang@fzu.edu.cn}
\address{School of Mathematics and Statistics, Fuzhou University, Fuzhou 350108, China,
and School of Mathematical Sciences,  Peking University, Beijing 100871, China}

\author{Dong Zhang}
\email{dongzhang@math.pku.edu.cn}
\address{School of Mathematical Sciences,  Peking University, Beijing 100871, China}

\author{Weixi Zhang}
\email{zjwzrazwx@gmail.com}

\address{School of Mathematical Sciences,  Peking University, Beijing 100871, China}

\begin{document}





\begin{abstract}
We introduce and develop  equivalent spectral graph theory for several fundamental graph cut problems including maxcut,   mincut, Cheeger cut, anti-Cheeger cut, dual Cheeger problem and their useful variants. A specified strategy for achieving an equivalent eigenproblem is proposed for a general graph cut problem  via  the set-pair Lov\'asz extension and the Dinkelbach scheme. For a class of 2-cut and 3-cut problems, we reveal the intrinsic difference-of-submodularity for the fractional formulations and show that their set-pair Lov\'asz extensions yield equivalent difference-of-convex structures. Building on the Dinkelbach scheme, we finally establish a unified research roadmap for nonlinear spectral theory that provides a one-to-one correspondence between certain eigenpairs and the optimal graph cut problems. The finer structure of the eigenvectors, the Courant nodal domain theorem and the graphic feature of eigenvalues are studied systematically in the setting of these new nonlinear eigenproblems.  

\vspace{0.2cm}

 
\end{abstract}

\maketitle

\section{Introduction}

Many problems on graph cuts, such as Cheeger cut, maxcut,  and balanced minimum cut, are quite important and popular in both pure and applied mathematics. Most of those graph cut problems have rich  spectral properties, based on adjacency matrix or Laplacian matrix on  graphs, and turns to be the foundation of spectral graph theory. However, the linear spectral properties do not capture the graph cuts accurately. 
For example, sometimes the two nodal domains of the second  eigenvector of Laplacian matrix provide a cut that is very different from the Cheeger cut.  An instructive  example where the nodal domain approximation of the Cheeger cut fails is provided  by 
Guattery and Miller \cite{Guattery/Miller}.  

A recent line of research indicates that nonlinear spectral graph theory can solve this problem. 
As a natural discretization of the eigenvalue problem regarding the  1-Laplace operator, the graph 1-Laplacian has been introduced first by Hein and B\"uhler \cite{HeinBuhler2010} for the purpose of application to spectral clustering and then by Chang \cite{refC-JCT} for the aim of studying the relatively isoperimetric problem from the variational point of view. 
The Cheeger cut and dual Cheeger problem based on graph 1-Laplacian 
had been studied systematically in both theory and application \cite{refCSZ-JCM,refCSZ-Adv,refC-JCT}. 
The spectral theory of 1-Laplacian type operator provides tools to analyze Cheeger type constants on graphs, offering insights into the combinatorial optimization through analytic methods. 
It would be interesting and useful to implement and accomplish  similar ideas in other important combinatorial optimizations. 

Once a combinatorial optimization problem is given, several  ``{soficity}'' questions arise: 
\vspace{0.2cm}

{\it Is there an equivalent eigenproblem for such combinatorial optimization problem?} 

{\it What is the spectral property of such equivalent eigenproblem?} 

{\it What are the advantages over linear spectral theory?}
\vspace{0.2cm}

In fact,  all the above questions are still open 
on most of graph optimization problems. 
 Even some of fundamental graph cut problems, including maxcut, dual Cheeger cut, anti Cheeger cut, 
and their modified versions, are not well understood from the spectral perspective. To remedy this, we shall provide equivalent spectral theory for these graph cut problems, and we would focus on the min-max eigenvalues, which closely relate to the higher order versions of these fundamental graph cut problems. We will compare  these spectral theories in detail, and highlight some different but useful results. 
Among a plenty of graph cut problems with some prefixed properties, a handful of 2-cut problems should be  basic, for instance, 
\begin{enumerate}
\item {a cut with minimum number of edges (i.e., the mincut problem);}
\item a ``SMALL" cut that separates the graph into two ``BIG" subgraphs (i.e., the Cheeger cut problem)
;
\item a cut with maximum number of edges (i.e., the maxcut problem);
\item a ``BIG" cut that separates the graph into two ``SMALL" subgraphs (i.e., the anti-Cheeger  problem)
.
\end{enumerate}
\begin{figure}
    \centering
    \includegraphics[width=0.5\linewidth]{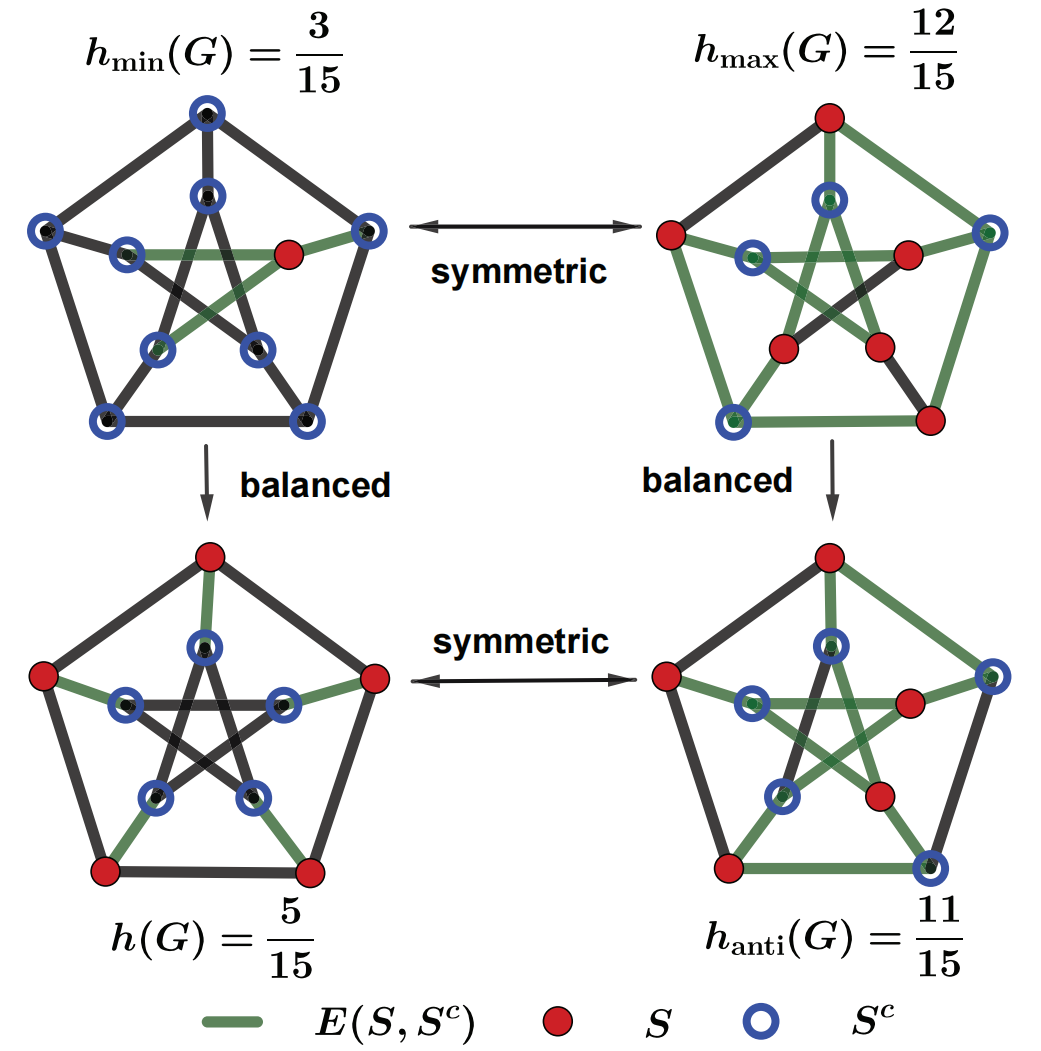}
    \caption{A bipartition
 $(S,S^c)$ of the vertex set $V$ is displayed in red bullets ($S$) and blue circles ($S^c$).    For the Cheeger cut or the anti-Cheeger cut, the remaining 
 uncut edges (in black)  are divided equally in  $S$ and $S^c$ (balanced and judicious!). For  the mincut or the maxcut,  the remaining uncut edges are all contained in one part (biased!). }
    \label{fig:f4-cut}
\end{figure}
By the description  in Fig.~\ref{fig:f4-cut}, the Cheeger cut and the anti-Cheeger cut can be regarded as  ``balanced'' versions of the mincut and the maxcut, respectively. 
Since the standard mincut problem can be solved in polynomial time, we won't pay much attention to it, and instead we will systematically study another balanced variant of the maxcut problem (called the  dual Cheeger problem) in this paper. 
Before stating the related spectral graph theory, let us recall the fundamental graph cut problems that 
 will be systematically studied in this paper:
\begin{itemize}
\item[] \textbf{Cheeger cut}. The standard \emph{Cheeger constant} of an undirected graph $G=(V,E)$ is defined as 
	\begin{equation}
\label{prob:Cheeger}	h(G)=\min\limits_{S\subseteq V,S\not\in\{\emptyset,V\}} \frac{|E( S,S^c)|}{\min\{\vol(S),\vol(S^c)\}}
	\end{equation}
where $V=\{1,\cdots,n\}$ is the vertex set, $\vol(S)$ is the sum of vertex weights of vertices in $S$, and $|E(S,S^c)|$ counts the total edge weights of edges acrossing $S$ and $S^c$. While any solution $S$ of the combinatorial minimization problem \eqref{prob:Cheeger} gives a \emph{Cheeger cut} $(S,S^c)$.
\item[]\textbf{Maxcut}. The very famous \emph{maxcut problem} is 
 a combinatorial optimization defined by 
\begin{equation}\label{prob:maxcut}
h_{\max}(G)= \max\limits_{S\subseteq V,S\not\in\{\emptyset,V\}}\frac{2|E(S,S^c)| }{\vol(V)}
\end{equation}
and a cut $(S,S^c)$ realizing \eqref{prob:maxcut} is called a maxcut. We note here that the {\sl mincut problem} is simply defined by using $\min$ instead of $\max$ in \eqref{prob:maxcut}.

\item[]\textbf{Dual Cheeger problem}.  
The \emph{dual Cheeger constant} \cite{Trevisan,Bauer-Jost} is simply defined by 
 \begin{equation}
h^+(G)=\max\limits_{V_1\cap V_2=\varnothing,V_1\cup V_2\ne \varnothing} \frac{2|E(V_1,V_2)|}{\vol(V_1\cup V_2)},\label{eq:dual-Cheeger-constant}	    
\end{equation}
A useful variant of the dual Cheeger problem \eqref{eq:dual-Cheeger-constant} is first presented in \cite{Soto} and then in \cite{SYZ},
\begin{equation}
\widehat{h}^+(G)=\max\limits_{V_1\cap V_2=\varnothing,V_1\cup V_2\ne \varnothing} \frac{2|E(V_1,V_2)|+|E\left(V_1\cup V_2,(V_1\cup V_2)^c\right)|}{\vol(V_1\cup V_2)+|E\left(V_1\cup V_2,(V_1\cup V_2)^c\right)|}.\label{eq:modified-dual-Cheeger-constant}     
 \end{equation}  

\item[]
\textbf{Anti-Cheeger cut}. As a corresponding judicious version of the maxcut problem,  the
 \emph{anti-Cheeger cut} 
\begin{equation} \label{eq:antiCheeger}
h_{\anti}(G)=\max_{S\subseteq V,S\not\in\{\emptyset,V\}}\frac{|E( S,S^c)|}{\max\{\vol(S),\vol(S^c)\}}.
\end{equation}
has more balanced structure \cite{Xu16}, 
which may fix
 the biasness of the maxcut problem \cite{SY21}. 
\end{itemize}

Fig.~\ref{fig:f4-cut} shows an example for these fundamental graph cut problems on the Petersen graph. All these graph cut problems listed  above are NP-hard \cite{Karp}, and thus it would be good to investigate fast spectral methods for these problems. 

Theories of spectra on graphs are arguably one of the most important directions in graph theory in the last decade. They link graph theory to spectral theory, and through this connection introduce new tools to graph theory. 
Let's start with 
the spectrum of graph Laplacian, which provides good lower and upper bounds for the Cheeger constant. For instance, let $\lambda_k(\Delta_2)$ denote the $k$-th  smallest eigenvalue of the normalized graph Laplacian $\Delta_2$, where $k=1,2,\cdots,n$. Then there holds the famous Cheeger inequality \cite{Alon86,Alon-Milman,Dodziuk84,Fiedler} 
\begin{equation}\label{eq:Cheeger-inequality}
\frac{h^2(G)}{2}\le \lambda_2(\Delta_2)\le 2h(G).
\end{equation}
We remark that the proof of the left-hand-side inequality (Cheeger's part) above is far from obvious, while the proof of
 the right-hand-side inequality (Buser's part)  is surprisingly   simple. 
Whilst the above Cheeger inequality gives a two-sided sharp bound for the Cheeger constant by the easily computable second Laplacian eigenvalue, neither the eigenvalue nor the eigenvector solve the Cheeger cut problem on general graphs. 

With regard to  different extensions of the Cheeger inequality, the dual version
\begin{equation}\label{eq:dual-Cheeger-inequality}
\frac{\big(1-h^+(G)\big)^2}{2}\le 2-\lambda_n(\Delta_2)\le 2\big(1-h^+(G)\big)
\end{equation}
established by Trevisan \cite{Trevisan}, and independently, by Bauer and Jost \cite{Bauer-Jost}, 
seems quite fundamental. 

Regarding the Laplacian spectrum bound for maxcut, Delorme and Poljak \cite{DelormePoljak93} give an inequality 
$$ h_{\max}(G)\le \frac n4 \lambda_n(L+U) $$
where $\lambda_n(L+U)$ is the largest eigenvalue of the matrix $L+U$, $L$ is the \emph{unnormalized} Laplacian of $G$ and $U$ is any given diagonal matrix with zero trace. It is easy to slightly modify Delorme-Poljak's inequality  by using the normalized Laplacian $\Delta_2$:  
$$ h_{\max}(G)\le \frac{\vol(V)}{4} \lambda_n(\Delta_2+D^{-1}U) $$
where $D$ is the diagonal matrix of degrees. In particular, $h_{\max}(G)\le  \frac14\vol(V) \lambda_n(\Delta_2)$.

These results from spectral graph theory have proven to be extremely useful to derive eigenvalue bounds for certain graph cut problems and  graph parameters. 
However, the classical spectral graph theory based on graph matrices (e.g., adjacency matrix and Laplace matrix) never provide the ``correct"\footnote{We shall give a footnote here to explain what   ``correct" means in our context. Linear spectral theory is not ``correct" in that the eigenvalues do not exactly represent the cut value,  the eigenvectors need rounding to get back to the original problem,  and the method of rounding the eigenvectors is not easy, and moreover, the resulting cut does not necessarily still have the original inequality relation. In contrast,  our equivalent nonlinear  spectral theory is ``correct" in that the eigenvectors are completely equivalent to the solution of the graph cut problem, and thus do not require  any rounding. }  information of cuts in general sense. 
Fortunately, the $\ell^1$-type  nonlinear spectral graph theory partially solves the problem, in which a fundamental fact is: 

\vspace{0.2cm}

{\it All the above graph cut problems can be equivalently transformed to certain eigenproblems of 1-Laplacian type.} 

\vspace{0.2cm}

We shall explain the above slogan in detail. 
It has been proved that  the Cheeger constant $h(G)$ 
equals the second eigenvalue of the set-valued eigenproblem
\textcolor{black}{\begin{equation}\label{eq:1-lap}
\vec 0\in\Delta_1\vec x-\mu \vec d\circ\sgn(\vec x)    \end{equation}} 
which is called `1-Laplacian' eigenequation in \cite{HeinBuhler2010,refC-JCT}. And the Cheeger cut is equivalent to the second eigenvector of \eqref{eq:1-lap}, which will be discussed in Sections  \ref{sec:Cheeger}. Thus, if we work on 1-Laplacian instead of the normalized Laplacian, the Cheeger inequality \eqref{eq:Cheeger-inequality} becomes an equality, and the disadvantage that the Laplacian eigenvectors do not provide Cheeger cut automatically disappears.

Similar situation  occurs  for 
the dual Cheeger problem. 
We prove that the solution of the dual Cheeger problem is equivalent to 
the first eigenpair of the so-called `signless 1-Laplacian' eigenequation  \cite{SYZ}:
\textcolor{black}{
\begin{equation}\label{eq:signless-1}
\vec 0\in\Delta_1^+\vec x-\mu\vec d\circ\sgn(\vec x).   \end{equation}} 
We continue the work in \cite{SYZ} to provide  estimates for higher order eigenvalues. For example, we establish a higher order  Cheeger inequality, and provide a  relation between the Laplacian spectrum and the signless 1-Laplacian eigenvalue \eqref{eq:signless-1} (see Theorems \ref{thm:k-way-dual-Cheeger} and \ref{thm:2Lap-s1Lap}). 
A remarkable result we would like to highlight is that the multiplicity of the largest eigenvalue of $\Delta_1^+$ is at least the independence number $\alpha$, and at most the edge covering number (see Theorem 
\ref{thm:multi-of-1}), which has no analogous on the traditional linear spectrum. In particular, if the graph is bipartite, then the eigenvalue multiplicity equals $\alpha$. 
Another noteworthy result concerns the case of equalities on higher order Cheeger inequality:  If the graph is a tree, then the $k$-way dual Cheeger constant plus  the $k$-th min-max eigenvalue of $\Delta_1^+$ equals 1,  which has no analogous result in linear spectral theory.  
This equality is a consequence of the sharp nodal count theory established in \cite{DPT23}. 
We also point out that since the property of $\Delta_1^+$ is quite different from $\Delta_1$, making the Courant-type nodal domain theorem no longer valid for $\Delta_1^+$, we employ an alternative natural modification of nodal domains to ensure the validity of Courant's theorem. 

One major contribution of this work is the comprehensive study of the equivalent eigenproblem for maxcut
\textcolor{black}{
\begin{equation}\label{eq:1-lap-maxcut}
\vec 0\in\Delta_1\vec x-\mu \vol(V) \partial\|\vec x\|_\infty    \end{equation}} 
in which any eigenvector corresponding to the largest eigenvalue of the eigenproblem \eqref{eq:1-lap-maxcut} directly produces a  cut of maximum size (i.e., a solution of the maxcut problem). 
We would like to highlight that by introducing new versions of nodal domain, we actually  strengthen the  Courant-type nodal domain theorem on \eqref{eq:1-lap-maxcut}, see Section \ref{sec:maxcut}.  

Another main contribution of this work is a new equivalent spectral theory for the Cheeger constant, 
which is closely related to, but not the same as, the graph 1-Laplacian. 
We prove that the eigenvalues of the 1-Laplacian eigenproblem \eqref{eq:1-lap} are always the eigenvalues of the  eigenproblem
\textcolor{black}{ 
\begin{equation}\label{eq:new-Cheeger-equi}
\vec 0\in \vol(V)\partial\|\vec x\|_\infty-\Delta_1^+\vec x-\mu\partial N(\vec x)
\end{equation}
}
and the second eigenvalue of \eqref{eq:new-Cheeger-equi} equals the Cheeger constant, see Section \ref{sec:new-Cheeger}. Furthermore, the nonzeros of the second eigenvector of \eqref{eq:new-Cheeger-equi} provide a vertex cover, which does not hold for the graph  1-Laplacian \eqref{eq:1-lap}. 

Similar to the investigation of  equivalent eigenproblems for maxcut, we also establish a spectral theory for anti-Cheeger problem in  Section \ref{sec:anti-Cheeger}. 
Formally, we study the property of  eigenpairs for the eigenproblem 
\textcolor{black}{
\begin{equation}\label{eq:anti-spec}
\vec 0\in 
\Delta_1\vec x -2\mu\vol(V)\partial\|\vec x\|_\infty+\mu\partial N(\vec x).
\end{equation}}



The main framework behind these equivalent eigenproblems of fundamental graph cuts is a unified research roadmap that first express 2-cut and 3-cut problems in discrete fractional difference-of-submodular formulations, and subsequently transform them into an equivalent continuous fractional difference-of-convex optimization problem via the set-pair Lov\'asz extension, and then construct the nonlinear eigenproblem by extracting the subproblem of the Dinkelbach 2-step iterative scheme, which preserves the same global optimum. 
Finally, we establish that the corresponding eigenpairs coincide with the optimal graph cuts in certain cases.


The paper is organized as follows. In the rest of this section, we collects notations used in the following sections.  
Section~\ref{sec:new-perspective} establishes the theoretical foundations of our main framework. Section~\ref{sec:Cheeger} states our main results on two equivalent eigenproblems for the Cheeger cut. Section~\ref{sec:dual-Cheeger} studies the equivalent signless 1-Laplacian spectral theory for dual Cheeger problems and their variants. Section~\ref{sec:maxcut} presents the first equivalent spectral theory for maxcut. Section~\ref{sec:anti-Cheeger}  proposes a detailed equivalent eigenproblem for the anti-Cheeger cut.


\subsection{Notations}
\label{sec:notation}
We list some notations and settings adopted in the following sections.
\begin{itemize}
	\item For the vertex set $V=\{1,2,\ldots,n\}$ of the given graph $G=(V,E)$, let
	\begin{align}
		\mathcal{P}(V)&=\{A:A\subset V\},\\
		\bc(V)&=\mathcal{P}(V)\backslash\{\emptyset,V\},\\
		\mathcal{P}_2(V)&=\{(A,B):A,B\subset V\text{ with }A\cap B=\emptyset\},\\
		\tc(V)&=\mathcal{P}_2(V)\backslash\{(\emptyset,\emptyset),(V,\emptyset),(\emptyset, V)\}.
	\end{align}
	\item Given the graph $G=(V,E)$, for any $\vec x\in\mathbb{R}^n$ and $t\in\mathbb{R}$, the vertex set $V$ can be divided into three classes:
	\begin{align}
		V_t^+(\vec x)&=\{i\in V: x_i>t\},\label{subset:positive1}\\
		V_t^-(\vec x)&=\{i\in V: x_i<t\},\label{subset:negative1}\\
		V_t^0(\vec x)&=\{i\in V: x_i=t\}.\label{subset:small1}
	\end{align}
	The standard sign function $\sign(t)$ and the related set valued mapping $\sgn(t)$ are defined respectively as
	\begin{equation}
		\sign(t)= \begin{cases}
			1 & \text{if } t>0,\\
			0 & \text{if }t=0,\\
			-1 & \text{if }t<0,
		\end{cases}
		\,\text{ and }\,
		 \mathrm{Sgn}(t)= \begin{cases}
			\{1\} & \text{if } t>0,\\
			[-1,1] & \text{if }t=0,\\
			\{-1\} & \text{if }t<0,
		\end{cases}
	\end{equation}
	and they can act on vectors in a coordinate-wise manner, i.e.,
	\begin{equation}
		\sign(\vec x)=\left(\sign(x_1),\ldots,\sign(x_n)\right),\,\, \sgn(\vec x)=\left(\sgn(x_1),\ldots,\sgn(x_n)\right).
	\end{equation}
	\item The subgradient of any convex function $f:\R^n\to\R$ is defined as 
	\begin{equation}
		\partial f(\vec x)=\{\vec s\in\R^n:f(\vec y)-f(\vec x)\ge \langle \vec s,\vec y-\vec x\rangle,\,\forall\vec y\in\R^n\}.
	\end{equation}
	\item For any $\vec x\in\mathbb{R}^n$, we denote
	\begin{align}
		I(\vec x)&=\sum_{\{i,j\}\in E}w_{ij}|x_i-x_j|,\,N(\vec x)=\min_{t\in\R}\sum_{i\in V}d_i|x_i-t|,\,I^+(\vec x)=\sum_{\{i,j\}\in E}w_{ij}|x_i+x_j|,\label{form:iniplus}\\
		\|\vec x\|&=\sum_{i\in V}d_i|x_i|,\,\|\vec x\|_1=\sum_{i\in V}|x_i|,\,\|\vec x\|_{\infty}=\max\{|x_1|,|x_2|,\ldots,|x_n|\}.\label{form:norm}
	\end{align}
	We also use the notations:
	\begin{equation}
		\Delta_{1}\vec x=\partial I(\vec x),\quad\Delta_{1}^+\vec x,=\partial I^+(\vec x),\quad \vec d\circ\sgn(\vec x)=\partial\|\vec x\|,
	\end{equation}
	 where $\vec d = (d_1,\ldots,d_n)$ is the degree vector, with $d_i$ denoting the degree of vertex $i$, and “$\circ$’’ denotes the Hadamard (elementwise) product. Here we also refer to $\Delta_1$ and $\Delta_1^+$ as the graph 1-Laplacian and the signless 1-Laplacian, respectively.
	\item We define some compact sets as follows:
	\begin{align}
		X&=\{\vec x\in\mathbb{R}^n:\|\vec x\|=1\},\\
		X_{\infty}&=\{\vec x\in\mathbb{R}^n:\|\vec x\|_{\infty}=1\},\label{set:infty}\\
		X_1&=\{\vec x\in\mathbb{R}^n:\|\vec x\|_1=1\},\\
	    \Omega_1&=\{\vec x\in X_1:\max_i x_i +\min_i x_i = 0\},\label{set:omega1}\\
	    \Omega_2&=\{\vec x\in X_1: \min_i |x_i|=0\}.\label{set:omega2}
	\end{align}
	\item We write $\vec 1=(1,1,\ldots,1)$ for the all-ones vector in $\mathbb{R}^n$ and 
	\begin{equation}
		\label{eq:one-complement}
		\vec 1^\bot=\{\vec x\in\mathbb{R}^n:\,\langle\vec x,\vec 1\rangle=0\}
	\end{equation}
	for its orthogonal complement.
	\item For any nonempty subsets $A,B\subset V$, we define indicative vectors $\vec 1_A,\,\vec 1_{A,B}\in\mathbb{R}^n$:
	\begin{equation}
		\left(\mathbf{1}_A\right)_i= \begin{cases}1, & i \in A, \\ 0, & i \notin A,\end{cases}\qquad \vec 1_{A,B}=\vec 1_A-\vec 1_B,
	\end{equation}
	thus we have $\sign(\vec x)=\vec 1_{V_0^+(\vec x),V_0^-(\vec x)}$. 
	\item Given a graph $G=(V,E)$, for a vector $\vec x =(x_1,x_2,\cdots,x_n)\in \mathbb{R}^n\setminus \{\vec 0\}$, one divides $V$ into three groups
	\begin{eqnarray}
		&D_+(\vec x)& =\set{i\in V\big| x_i = \norm{\vec x}},\quad  D_-(\vec x)=\set{i\in V\big| -x_i = \norm{\vec x}},\label{eq:dplusminus}\\
		&D_0(\vec x)&=\set{i\in V\big| \abs{x_i}<\norm{\vec x}}.\label{eq:dzero}
	\end{eqnarray}
	\item We introduce a partial order $\vec x \preceq \vec y$ in a way that $\vec x,\vec y\ne \vec 0$ satisfies:
	\begin{eqnarray}
		&D_\pm(\vec x)\subseteq D_\pm(\vec y),\label{cond1:plusminus}\\
		&x_i< x_j\Rightarrow y_i\leq y_j;~x_i= x_j\Rightarrow y_i= y_j,~\forall i,j\in\{1,2,\cdots,n\}.\label{cond2:order}
	\end{eqnarray}
\end{itemize}

\section{A new perspective of nonlinear eigenproblems}
\label{sec:new-perspective}
In this section, we present a ``spectral" line of research for the graph cut problems, as shown in Fig.~\ref{fig:som-eigen}, after which we add technical details one by one in clockwise order. We unify the fundamental graph cut problems into a general class of 2-cut and 3-cut problems with fractional formulations. For any such problem $q^*$, we show that both the numerator and denominator are difference-of-submodular functions. Through the set-pair Lov\'asz extension, $q^*$ can be equivalently expressed into the continuous formulations, where the strict submodularity implies that both components are difference-of-convex functions. Furthermore, by employing a two-step Dinkelbach iterative scheme, we achieve a complete defractionalization of the original problem and transform the optimal solution of $q^*$ into the global optimum of the scheme. Finally, by isolating the first subproblem of the  Dinkelbach scheme, we construct the corresponding nonlinear eigenproblem and establish equivalence between the eigenpairs and the optimal solutions of $q^*$ in certain cases.

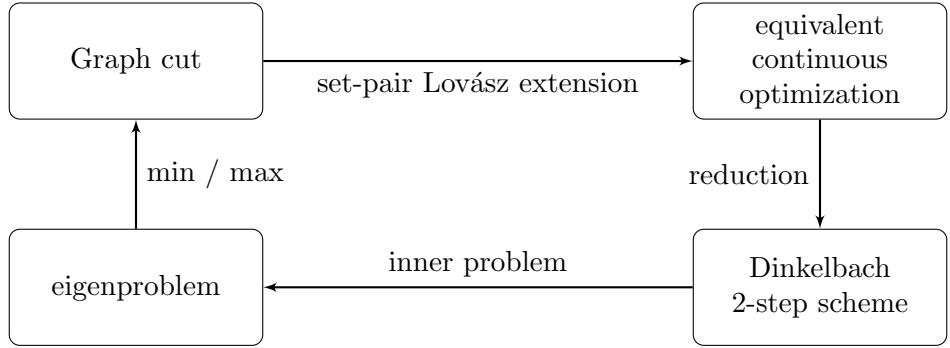
\begin{figure}[htbp]
	\centering
	\tikzstyle{pt} = [point]
	\tikzstyle{decision} = [diamond, draw, text width=4em, aspect=1.2, text badly centered, node distance=3cm, inner sep=0pt]
	\tikzstyle{decision2} = [diamond, draw, text width=6em, aspect=1.2, text badly centered, node distance=3cm, inner sep=0pt]
	\tikzstyle{block} = [rectangle, draw, 
	text width=8em, text centered, rounded corners, minimum height=4em]
	\tikzstyle{io} = [rectangle, draw, 
	text width=4em, text centered, rounded corners, minimum height=2em]
	\tikzstyle{line} = [draw, thick, -latex']
	\tikzstyle{cloud} = [draw, ellipse,fill=red!20, node distance=2.5cm,
	minimum height=2em]
	\tikzstyle{every node}=[scale=1.0]
	\begin{tikzpicture}
		
		\node [block] (graphcut) {Graph cut};
		
		\node [block, right of=graphcut, node distance=9cm] (som) {equivalent continuous optimization};
		
		\node [block, below of=som, node distance=3cm] (iter) { Dinkelbach's 2-step iterative scheme};
		
		\node [block, left of=iter, node distance=9cm] (eigen) {eigenproblem};
		
		\path [line] (graphcut)-- node[below] {set-pair Lov\'asz extension} (som);
		
			\path [line] (iter)-- node[above] {inner problem} (eigen);
			
			\path [line] (som)--node [anchor=east] {reduction } (iter);
			
			\path [line] (eigen)--node [anchor=west] {min / max} (graphcut);
		
	\end{tikzpicture}
	\caption{\small \textcolor{black}{A research roadmap for the construction of an equivalent continuous spectral theory for the graph cut problem.}} \label{fig:som-eigen}
\end{figure}

\subsection{\textcolor{black}{Equivalent continuous optimization for graph cut}}

 The 2-cut and 3-cut problems involve optimizing set functions or set-pair functions, and can be formulated as follows:
 
 \textcolor{black}{
 \begin{itemize}
 	\item 
 	\textbf{Reformulations of typical graph 2-cuts}.\;  
 	Assume that $F,G: \mathcal{P}(V)\rightarrow[0,+\infty)$ are two set functions with $G(A)>0$ whenever $A\in\bc(V)$,  $F(\emptyset)=G(\emptyset)=0$. Then numerous graph 2-cut problems can be formulated as	
 	\begin{subequations} \label{prob:2cut}
 		\begin{numcases}{}
 			\opt_{A\in\mathcal{P}(V)\backslash\{\emptyset\}}\frac{F(A)}{G(A)},\quad& if $G(V)>0$,\label{2cut-constant} \\
 			\opt_{A\in\bc(V)}\frac{F(A)}{G(A)},\quad& if $G(V)=0$, \label{2cut-nonconstant}
 		\end{numcases}
 	\end{subequations}
 	where $\opt\in\{\min,\max\}$.
 \end{itemize}
}


Among the four fundamental graph cut problems considered in this paper, maxcut \eqref{prob:maxcut} and anti-Cheeger cut \eqref{eq:antiCheeger} conform to \eqref{2cut-constant}, whereas mincut and Cheeger cut \eqref{prob:Cheeger} conform to \eqref{2cut-nonconstant}. Since set functions are not sufficient to describe graph 3-cut problems, we introduce set-pair functions to formalize the corresponding definition:
\textcolor{black}{\begin{itemize}
	\item \textbf{Reformulations of typical graph 3-cuts}.\;
	Assume that $f,g: \mathcal{P}_2(V)\rightarrow[0,+\infty)$ are two set-pair functions with $g(A,B)>0$ whenever $(A,B)\in\tc(V)$, $f(\emptyset,\emptyset)=g(\emptyset,\emptyset)=0$. Then numerous graph 3-cut problems can be formulated as	
	\begin{subequations} \label{prob:3cut}
		\begin{numcases}{}
			\opt_{(A,B)\in\mathcal{P}_2(V)\backslash\{(\emptyset,\emptyset)\}}\frac{f(A,B)}{g(A,B)},\quad& if $g(\emptyset,V),\,g(V,\emptyset)>0$,\label{3cut-constant} \\
			\opt_{(A,B)\in\tc(V)}\frac{f(A,B)}{g(A,B)},\quad& if $g(\emptyset,V)=g(V,\emptyset)=0$, \label{3cut-nonconstant}
		\end{numcases}
	\end{subequations}
	where $\opt\in\{\min,\max\}$.
\end{itemize}}

As established in \cite{Lovasz1983,BuhlerRangapuramSetzerHein2013}, the classic Lov\'asz extension is an effective tool for lifting set functions from a discrete domain $\mathcal{P}(V)$ to a continuous space $\mathbb{R}_+^n$, which facilitates continuous formulations of 2-cut problems through half-space representations. To enabling continuous formulations of 3-cut problems, the set-pair Lov\'asz extension was introduced in \cite{CSZZ-Lovasz18} to extended set-pair functions from a discrete domain $\mathcal{P}_2(V)$ to a continuous space $\mathbb{R}^n$. Notably, this extension also subsumes the 2-cut case (see Lemma 3.1 in \cite{CSZZ-Lovasz18}), and is therefore adopted throughout this paper as our primary analytical framework.

In this section, we establish the following two theorems:
\begin{theorem}
	\label{thm:dc-setpair}
	Every set-pair function $f: \mathcal{P}_2(V)\rightarrow[0,+\infty)$ is the difference of two strictly submodular functions, and its corresponding set-pair Lov\'asz extension $f^L$ is a DC (difference-of-convex) function.
\end{theorem}

\begin{theorem}
	\label{thm:dc-conti}
	Any  2-cut (see \eqref{prob:2cut}) or 3-cut (see \eqref{prob:3cut}) problem $q^* $ admits an equivalent continuous formulation 
		\begin{equation}
		\label{eq:som-lovasz}
		q^*=\mathop{\text{opt}}\limits_{\vec x\in\mathbb{R}^n\backslash\{\vec 0\}\text{ or } \text{nonconstant }\vec x\in\mathbb{R}^n} Q(\vec x),\quad Q(\vec x)=\frac{f_1^L(\vec x)-f_2^L(\vec x)}{g_1^L(\vec x)-g_2^L(\vec x)},
	\end{equation}  
	where $f_1^L$, $f_2^L$, $g_1^L$, $g_2^L$ are all convex one-homogeneous functions and arise as set-pair Lov\'asz extensions of certain strictly submodular functions.
\end{theorem}

These results show that the considered graph cut problems admit equivalent continuous formulations, where both the numerator and denominator of the objective function are set-pair Lov\'asz extensions that expressible in difference-of-convex (DC) form. Moreover, they highlight the intrinsic connections among set-pair functions, submodularity, and convexity. Before presenting the detailed proofs, we first introduce the key definitions and some properties for the set-pair Lov\'asz extension.



\textcolor{black}{
\begin{definition}[set-pair Lov\'asz extension  \cite{CSZZ-Lovasz18}]
	\label{def:setpair-lovasz}
	Given a set-pair function  $f: \mathcal{P}_2(V)\rightarrow[0,+\infty)$, its set-pair Lov{\'a}sz extension $f^L$ is defined by 
	\begin{equation}
		\label{eq:sigma-lovasz}
		f^L(\vec x)=\int_0^{\|\vec x\|_{\infty}} f\left(V_t^{+}(\vec x), V_t^{-}(\vec x)\right) d t.
	\end{equation}
\end{definition}}

\begin{proposition}	[equivalent continuous formulation for 2-cut \cite{CSZZ-Lovasz18,BuhlerRangapuramSetzerHein2013}]
	\label{thm:bc-lovasz}
	For any graph 2-cut problem specified in \eqref{prob:2cut}, let $f,g: \mathcal{P}_2(V)\rightarrow[0,+\infty)$ be two set-pair functions satisfying $f(A,B)=F(A)+F(B)$, $g(A,B)=G(A)+G(B)$, and denote their set-pair Lov\'asz extensions $f^L$, $g^L$, respectively. Then the equivalent continuous formulation of the graph 2-cut is given by
	\begin{subequations} \label{prob:2cut-conti}
		\begin{numcases}{}
			\opt_{A\in\mathcal{P}(V)\backslash\{\emptyset\}}\frac{F(A)}{G(A)}=\opt_{\vec x\in\mathbb{R}^n\backslash\{\vec 0\}}\frac{f^L(\vec x)}{g^L(\vec x)},\quad& if $G(V)>0$,\label{2cut-constant-conti} \\
			\opt_{A\in\bc(V)}\frac{F(A)}{G(A)}=\opt_{\text{nonconstant }\vec x\in\mathbb{R}^n}\frac{f^L(\vec x)}{g^L(\vec x)},\quad& if $G(V)=0$, \label{2cut-nonconstant-conti}
		\end{numcases}
	\end{subequations}
	where $\opt\in\{\min,\max\}$.
\end{proposition}

{\begin{proposition}[equivalent continuous formulation for 3-cut \cite{CSZZ-Lovasz18}]
	\label{thm:tc-lovasz}
	For any graph 3-cut problem specified in \eqref{prob:3cut}, let $f^L$ and $g^L$ be the set-pair Lov\'asz extensions of $f$ and $g$, respectively. Then the equivalent continuous formulation of the graph 3-cut is given by
		\begin{subequations} \label{prob:3cut-conti}
		\begin{numcases}{}
			\opt_{(A,B)\in\mathcal{P}_2(V)\backslash\{(\emptyset,\emptyset)\}}\frac{f(A,B)}{g(A,B)}=\opt_{\vec x\in\mathbb{R}^n\backslash\{\vec 0\}}\frac{f^L(\vec x)}{g^L(\vec x)},\,\,& if $g(\emptyset,V),\,g(V,\emptyset)>0$,\label{3cut-constant-conti} \\
			\opt_{(A,B)\in\tc(V)}\frac{f(A,B)}{g(A,B)}=\opt_{\text{nonconstant }\vec x\in\mathbb{R}^n}\frac{f^L(\vec x)}{g^L(\vec x)},\,\,& if $g(\emptyset,V)=g(V,\emptyset)=0$, \label{3cut-nonconstant-conti}
		\end{numcases}
	\end{subequations}
	where $\opt\in\{\min,\max\}$.
\end{proposition}

\textcolor{black}{\begin{definition}[set-pair function and strictly submodularity \cite{CSZZ-Lovasz18}]
		\label{def:set-submod}
	A set-pair function $f: \mathcal{P}_2(V) \rightarrow[0,+\infty)$ is said to be strictly submodular if the inequality
	\begin{equation}
		\label{leq:submodular}
		f(A_1, B_1)+f(A_2, B_2) \geq f((A_1 \cup A_2) \backslash(B_1 \cup B_2),(B_1 \cup B_2) \backslash(A_1 \cup A_2))+f(A_1 \cap A_2, B_1 \cap B_2)
	\end{equation}
	holds for any $(A_1, B_1),(A_2, B_2) \in \mathcal{P}_2(V)$.
\end{definition}}

\begin{proposition}[set-pair Lov\'asz extension and convexity \cite{CSZZ-Lovasz18}]
		\label{prop:convex}
	Let $f: \mathcal{P}_2(V) \rightarrow[0,+\infty)$ be a set-pair function satisfying $f(\varnothing, \varnothing)=0$. Then $f^L$ is convex if and only if $f$ is strictly submodular.
\end{proposition}

\begin{proposition}[\cite{SY24}]
	\label{prop:subgrad-3cut}
	For any $\vec x\in\mathbb{R}^n$, $\partial f^L(\vec x)\subseteq\partial f^L(\sign(\vec x))$ holds for every convex set-pair Lov\'asz extension $f^L$. Moreover, the one-homogeneous convexity of $f^L$ implies $f^L(\vec x)=\langle\vec x,\vec v\rangle$, $\forall\,\vec x\in\mathbb{R}^n$, $\forall\,\vec v\in\partial f^L(\vec x)$.
\end{proposition}

\begin{proposition} [\cite{CSZZ-Lovasz18}]
	\label{prop:linear}
	For any functions $f,g: \mathcal{P}_2(V) \rightarrow[0,+\infty)$ and their set-pair Lov\'asz extensions $f^L(\vec x)$ and $g^L(\vec x)$, we have
	\begin{enumerate}
		\item $f^L(\vec x)$ is one-homogeneous.
		\item $(f+\lambda g)^L=f^L+\lambda g^L$, $\forall \lambda \geq 0$.
		\item If $f$ is derived from a set function $F:\mathcal{P}(V) \rightarrow[0,+\infty)$, i.e., $f(A,B)=F(A)+F(B)$, then its set-pair Lov\'asz extension $f^L$ satisfies $f^L(\vec x+\alpha\vec 1)=f^L(\vec x)+\alpha F(V)$, $\forall\,\alpha\in\mathbb{R}$.
	\end{enumerate}
\end{proposition}

\textcolor{black}{
\begin{definition}
  The gap function $\mathcal{G}:\mathcal{P}_2(V)\times \mathcal{P}_2(V)\rightarrow\mathbb{R}$ for any given set-pair function $f$ is defined as
  $$\begin{aligned}
  	\mathcal{G}_f((A_1,B_1),(A_2,B_2))&=f(A_1,B_1)+f(A_2,B_2) -f(A_1 \cap A_2, B_1 \cap B_2)\\
  	&- f((A_1 \cup A_2) \backslash(B_1 \cup B_2),(B_1 \cup B_2) \backslash(A_1 \cup A_2)).
  \end{aligned}$$
\end{definition}
}

\textcolor{black}{
\begin{lemma}
	\label{lem:h-example}
	The set-pair function 
	\begin{equation}
		\label{eq:h-example}
		h(A_1,B_1):=|A_1||A_1^c|+|B_1||B_1^c|+\sqrt{|A_1|+|B_1|}
	\end{equation}
	is strictly submodular, and the corresponding gap function $\mathcal{G}_h((A_1,B_1),(A_2,B_2))$ equals 0 if and only if $(A_1, B_1) \subseteq(A_2,B_2)$ or $(A_1,B_1) \supseteq(A_2,B_2)$. Moreover, 
	\begin{equation}
		\mathcal{G}_h((A_1,B_1),(A_2,B_2))\geq\frac{1}{8\sqrt{2}|V|^{3/2}}
	\end{equation} 
	is satisfied whenever $(A_1,B_1) \nsubseteq (A_2,B_2)$ and $(A_1,B_1) \nsupseteq (A_2,B_2)$.
\end{lemma}}

\textcolor{black}{
\begin{proof}
	The necessity is obvious, and we only need to prove the sufficiency. Let 
	\begin{align*}
	h_1(A_1,B_1) &:= |A_1||A_1^c|+|B_1||B_1^c|, \\
	h_2(A_1,B_1) &:= \sqrt{|A_1|+|B_1|},
	\end{align*}
	and we have $h=h_1+h_2$, where $h_2$ is indeed strictly submodular \cite{CSZZ-Lovasz18}. Now we prove that $h_1$ satisfies \eqref{leq:submodular}. A direct calculation shows that for all $(A_1,B_1),(A_2,B_2)\in\mathcal{P}_2(V)$,
	$$\begin{aligned}
	  \mathcal{G}_{h_1}((A_1,B_1),(A_2,B_2))=\,&2|A_1\backslash A_2||A_2\backslash A_1|+2|B_1\backslash B_2||B_2\backslash B_1|\\
	  &+2|(A_1\cup A_2)\cap(B_1\cup B_2)||(A_1\cup B_1\cup A_2\cup B_2)^c|\geq 0.
	\end{aligned}$$
  The necessary condition of $\mathcal{G}_{h_1}((A_1,B_1),(A_2,B_2))=0$ is that at least one of the following four cases holds: $(A_1,B_1) \subseteq(A_2,B_2)$, $(A_1,B_1) \supseteq(A_2,B_2)$, $(A_1, B_2) \subseteq(A_2, B_1)$ and $(A_1, B_2) \supseteq(A_2, B_1)$. Therefore, for any $(A_1,B_1),(A_2,B_2)\in\mathcal{P}_2(V)$ not belonging to the four cases above, $$	\mathcal{G}_h((A_1,B_1),(A_2,B_2))\geq\mathcal{G}_{h_1}((A_1,B_1),(A_2,B_2))\geq 2.$$ The remaining task is to prove  $\mathcal{G}_h((A_1,B_1),(A_2,B_2))\geq\frac{1}{8\sqrt{2}|V|^{3/2}}$ under the additional condition $(A_1, B_2) \subset(A_2, B_1)$ or $(A_1, B_2) \supset(A_2, B_1)$. Without loss of generality, we only need to consider the first case with $|A_2|-|A_1|\geq 1$ and $|B_1|-|B_2|\geq 1$, that is,
  $$\begin{aligned}
  	\mathcal{G}_h((A_1,B_1),(A_2,B_2))&\geq \mathcal{G}_{h_2}((A_1,B_1),(A_2,B_2))\\
  	&\geq \sqrt{|A_1|+|B_1|}+\sqrt{|A_2|+|B_2|}-\sqrt{|B_1|+|A_2|}-\sqrt{|A_1|+|B_2|}\\
  	&\geq \frac{\sqrt{(|A_1|+|B_1|)(|A_2|+|B_2|)}-\sqrt{(|B_1|+|A_2|)(|A_1|+|B_2|)}}{2\sqrt{2|V|}}\\
  	&\geq \frac{(|A_2|-|A_1|)(|B_1|-|B_2|)}{8\sqrt{2}|V|^{3/2}}\geq \frac{1}{8\sqrt{2}|V|^{3/2}}.
  \end{aligned}$$
\end{proof}}

\textcolor{black}{
\begin{proof}[Proof of Theorem~\ref{thm:dc-setpair}]
According to Definition~\ref{def:set-submod}, $f$ is strictly submodular if and only if $\mathcal{G}_f((A_1,B_1),(A_2,B_2))\geq 0$ for all $(A_1,B_1),(A_2,B_2) \in \mathcal{P}_2(V)$. It is obvious that $$\mathcal{G}_f((A_1,B_1),(A_2,B_2))= 0$$ if $(A_1,B_1) \subseteq(A_2,B_2)$ or $(A_1,B_1) \supseteq(A_2,B_2)$. For all the remaining cases, we first rewrite $f$ as $f=(f+\lambda h)-(\lambda h)$ with $h$ being \eqref{eq:h-example}, thus $\mathcal{G}_f\geq 0$ under $\lambda\geq 8\sqrt{2}|V|^{3/2}(\max_{(A,B)}f(A,B)-\min_{(A,B)}f(A,B))$. Thus by Lemma~\ref{lem:h-example}, both $f+\lambda h$ and $\lambda h$ are strictly submodular functions. This implies that $f$ can be expressed as a difference of two strictly submodular functions.
Consequently, invoking Propositions \ref{prop:convex} and \ref{prop:linear}, we conclude that both $(f+\lambda h)^L$ and $(\lambda h)^L$ are one-homogeneous convex functions, ensuring that $f^L$ is a DC function. 
\end{proof}
}

\begin{proof}[Proof of Theorem~\ref{thm:dc-conti}]
	This proof is straightforward, relying on Theorem~\ref{thm:dc-setpair}, Propositions~\ref{thm:bc-lovasz} and \ref{thm:tc-lovasz}.
\end{proof}

\subsection{Equivalent 2-step iterative scheme}

Once the equivalent continuous formulation \eqref{eq:som-lovasz} of a graph cut problem is obtained, we are ready to establish a connection from  such continuous formulation to the Dinkelbach iterative scheme \cite{dinkelbach1967nonlinear}. 
%
Note that the feasible domain in \eqref{eq:som-lovasz} is not compact. \textcolor{black}{Hence, we restrict it to a compact set $\Omega$ with \eqref{set:omega1} and \eqref{set:omega2} defined by  
\begin{equation}
	\label{eq:dinkel-omega}
	\Omega=
 \left\{
	\begin{aligned}
		&	X_1,&\text{for } \vec x\in\mathbb{R}^n\backslash\{\vec 0\},\\
		&\Omega_1,&\text{for } \text{nonconstant }\vec x\in\mathbb{R}^n \text{ and 2-cut},\\
		&\Omega_1\cup\Omega_2,&\text{for } \text{nonconstant }\vec x\in\mathbb{R}^n \text{ and 3-cut},
	\end{aligned}
	\right.
\end{equation}}
Therefore, the corresponding graph cut problem~\eqref{eq:som-lovasz} maintains another equivalent continuous formulation:
$$q^*=\mathop{\text{opt}}\limits_{\vec x\in\Omega}Q(\vec x),\quad \text{ where }Q(\vec x)=\frac{f_1^L(\vec x)-f_2^L(\vec x)}{g_1^L(\vec x)-g_2^L(\vec x)},\quad \opt\in\{\min,\max\}.$$

\textcolor{black}{
\begin{theorem}[Dinkelbach's two-step iterative algorithm]
	\label{thm:dinkelbach} 
	For any initial point $\vec x^0\in\Omega$, the solution produced by the following two-step iteration 
	\begin{subequations}
		\label{dinkel-iter0}
		\begin{numcases}{}
			\vec x^{k+1}={\mathop{\mathrm{argopt}}\limits_{\vec x\in\Omega} \{f_1^L(\vec x)+\tau^k g_2^L(\vec x)-(f_2^L(\vec x)+\tau^k g_1^L(\vec x))\}},\label{dinkel-iter0-1}\\
			\tau^{k+1}=Q(\vec x^{k+1}), \label{dinkel-iter0-2}
		\end{numcases}
	\end{subequations}
converges to  $q^*$, where $ \mathop{\mathrm{argopt}}$ refers to $ \mathop{\mathrm{argmax}}$ or $ \mathop{\mathrm{argmin}}$.
\end{theorem}
}
\textcolor{black}{
\begin{proof}
We only need to consider the case $\text{opt}=\min$, as the case $\text{opt}=\max$ is similar.  Setting   $\tau_{\min}=\min_{\vec x\in\Omega} Q(\vec x)$, by assumptions, we have $\tau_{\min}=q^*$. 
	The vector $\vec x^{k+1}$ defined in \eqref{dinkel-iter0-1} satisfies
	\begin{equation*}
		\begin{aligned}
			0&=f_1^L(\vec x^k)+\tau^k g_2^L(\vec x^k)-(f_2^L(\vec x^k)+\tau^k g_1^L(\vec x^k))\\
			&\ge f_1^L(\vec x^{k+1})+\tau^k g_2^L(\vec x^{k+1})-(f_2^L(\vec x^{k+1})+\tau^k g_1^L(\vec x^{k+1})),\\
		\end{aligned}
	\end{equation*}
	which means 
	\begin{equation*}
		\tau^k\ge \tau^{k+1}\ge \tau_{\min}, \,\,\, \forall\, k\in \mathbb{N}^+.
	\end{equation*}
Therefore, we have
	\[
	\exists\, \tau^* \in [\tau_{\min},\tau^0]\,\, \text { s.t. }\, \lim\limits_{k\to+\infty}\tau^k=\tau^*,
	\]
According to \eqref{dinkel-iter0-2}, we have  $\tau^*\geq \tau_{\min}$, and it suffices to prove $\tau_{\min}\ge \tau^*$. Define the function
	\begin{equation*}
		l(\tau)=\min\limits_{\vec x\in\Omega}\left(f_1^L(\vec x)+\tau g_2^L(\vec x)-(f_2^L(\vec x)+\tau g_1^L(\vec x))\right).
	\end{equation*}
By the compactness of $\Omega$,  $\tau\mapsto f(\tau)$ is continuous on $\mathbb{R}$. Notice that 
	\begin{align*}
		l(\tau^k)&=f_1^L(\vec x^{k+1})+\tau^k g_2^L(\vec x^{k+1})-(f_2^L(\vec x^{k+1})+\tau^k g_1^L(\vec x^{k+1}))\\
		&=(\tau^{k+1}-\tau^k)(g_1^L(\vec x^{k+1})-g_2^L(\vec x^{k+1})) \rightarrow 0\quad \text{as}\quad k\to +\infty.
	\end{align*}
Then we have
	\begin{equation*}
		l(\tau^*)=\lim\limits_{k\to+\infty}l(\tau^k)=0,
	\end{equation*}
and thus
	\begin{equation*}
		f_1^L(\vec x)+\tau^* g_2^L(\vec x)-(f_2^L(\vec x)+\tau^* g_1^L(\vec x))\ge 0, \,\,\forall\,\vec x\in\Omega,
	\end{equation*}
which implies $\forall\, \vec x\in \Omega$, $Q(\vec x) \geq \tau^*$, i.e., $\tau_{\min}\ge \tau^*$. The proof is completed. 
\end{proof}}

\subsection{Equivalent nonlinear eigenproblem}
\textcolor{black}{In this section, we present the definition of the nonlinear eigenproblem, establish the equality between the eigenvalue and the optimal objective cut value, and explain how to obtain the optimal objective cut solution from the associated eigenvector. To this end, we first introduce some notations. 
}
Extracting the subproblem~\eqref{dinkel-iter0-1} of Dinkelbach's two-step iterative algorithm, the corresponding eigenproblem for the graph cut problem can be formulated as follows.

\begin{definition}[nonlinear eigenproblem]
	\label{def:eigen-all}
	We say that $(\mu,\vec x)\in\mathbb{R}\times \mathbb{R}^n\backslash\{\vec 0\}$  is an eigenpair associated to $Q(\vec x)$ if 
\begin{equation}
	  \label{eq:eigen-all}
	  \vec 0\in \partial f_1^L(\vec x)+\mu \partial g_2^L(\vec x)-(\partial f_2^L(\vec x)+\mu \partial g_1^L(\vec x)),
	\end{equation}	
	where $f_1^L$, $f_2^L$, $g_1^L$, $g_2^L$ are the same as those introduced in Theorem~\ref{thm:dc-conti}. We also simply call $(\mu,\vec x)$ an eigenpair of \eqref{eq:eigen-all}.
\end{definition}
\begin{theorem}
	\label{thm:eigen-value}
	For any eigenpair $(\mu,\vec x)\in\mathbb{R}\times \mathbb{R}^n\backslash\{\vec 0\}$ of the nonlinear eigenproblem~\eqref{eq:eigen-all}, we have
	$$\mu=\frac{f_1^L(\vec x)-f_2^L(\vec x)}{g_1^L(\vec x)-g_2^L(\vec x)}.$$
\end{theorem}
\begin{proof}
If $(\mu,\vec x)\in\mathbb{R}\times \mathbb{R}^n\backslash\{\vec 0\}$ is an eigenpair, there exists $\vec v_1\in\partial f_1^L(\vec x)$, $\vec v_2\in\partial f_2^L(\vec x)$, $\vec u_1\in\partial g_1^L(\vec x)$ and $\vec u_2\in\partial g_2^L(\vec x)$ satisfying 
	$$\vec 0=\vec v_1+\mu\vec u_2-(\vec v_2+\mu\vec u_1).$$
	According to Proposition~\ref{prop:subgrad-3cut}, we have
	$$
	\begin{aligned}
	 0&=\langle\vec v_1+\mu\vec u_2-(\vec v_2+\mu\vec u_1),\vec x\rangle\\
	 \Leftrightarrow 0&=f_1^L(\vec x)+\mu g_2^L(\vec x)-(f_2^L(\vec x)+\mu g_1^L(\vec x))\\
	 \Leftrightarrow \mu&= \frac{f_1^L(\vec x)-f_2^L(\vec x)}{g_1^L(\vec x)-g_2^L(\vec x)}.
	\end{aligned}
	$$
\end{proof}

\begin{theorem}[eigenpairs for 2-cut]
	\label{thm:eigen-2cut}
	All eigenvalues of the nonlinear eigenproblem \eqref{eq:eigen-all} are arranged in ascending order $\mu_1\leq\mu_2\leq\cdots\leq \mu_{\max}$. The relationship between  the 2-cut optima \eqref{eq:som-lovasz} and the eigenvalues is
	\begin{equation}
		\label{eq:eigenval}
		q^*=\left\{\begin{aligned}
			&\mu_1,&\text{if }q^*=\min\limits_{\vec x\in\mathbb{R}^n\backslash\{\vec 0\}} Q(\vec x),\\
			&\mu_2,&\text{if }q^*=\min\limits_{\text{nonconstant }\vec x\in\mathbb{R}^n} Q(\vec x),\\
			&\mu_{\max},&\text{if }q^*=\max\limits_{\vec x\in\mathbb{R}^n\backslash\{\vec 0\}} Q(\vec x).
		\end{aligned}
		\right.
	\end{equation}
	Moreover, for any eigenvector $\vec x$ associated with a given eigenvalue, \textcolor{black}{$\vec 1_{D_+(\vec x),(D_+(\vec x))^c}$ or $\vec 1_{(D_-(\vec x))^c,D_-(\vec x)}$} yields the optimal solution of the 2-cut problem.
\end{theorem}

\begin{theorem}[eigenpairs for 3-cut]\label{thm:eigen-3cut}
	All eigenvalues of the nonlinear eigenproblem \eqref{eq:eigen-all} are arranged in ascending order $\mu_1\leq\mu_2\leq\cdots\leq \mu_{\max}$. The relationship between  the 3-cut optima \eqref{eq:som-lovasz} and the eigenvalues is \eqref{eq:eigenval}.
	Moreover, for any eigenvector $\vec x$ associated with a given eigenvalue, $\sign(\vec x)$ yields the optimal solution of the 3-cut problem.
\end{theorem}

Before giving the proofs of Theorems~\ref{thm:eigen-2cut} and \ref{thm:eigen-3cut}, we need the following lemma. 

\begin{lemma}
	\label{lem:eigen-cut}
	Given any eigenpair $(\mu,\vec x)\in\mathbb{R}\times \mathbb{R}^n\backslash\{\vec 0\}$ of the eigenproblem~\eqref{eq:eigen-all}, $(\mu,\sign(\vec x))$ is also an eigenpair. Moreover, if \eqref{eq:eigen-all} is derived from a 2-cut problem, \textcolor{black}{$(\mu,\vec 1_{D_+(\vec x),(D_+(\vec x))^c})$ or $(\mu,\vec 1_{(D_-(\vec x))^c,D_-(\vec x)})$} is also an eigenpair.
\end{lemma}
\begin{proof}
	Let $\vec y=\sign(\vec x)$, thus we have $\partial f_1^L(\vec x)\subseteq \partial f_1^L(\vec y)$, $\partial f_2^L(\vec x)\subseteq \partial f_2^L(\vec y)$, $\partial g_1^L(\vec x)\subseteq \partial g_1^L(\vec y)$ and $\partial g_2^L(\vec x)\subseteq \partial g_2^L(\vec y)$  through Proposition~\ref{prop:subgrad-3cut}, which leads to 
	$$\vec 0\in \partial f_1^L(\vec y)+\mu \partial g_2^L(\vec y)-(\partial f_2^L(\vec y)+\mu \partial g_1^L(\vec y)).$$
	For the 2-cut case, select any $\vec v\in\partial f_1^L(\vec x)$ and $\alpha>0$. According to Propositions~\ref{prop:subgrad-3cut} and \ref{prop:linear}, we have
	$$\langle\vec v,\vec x+\alpha\vec 1\rangle+\langle\vec v,\vec x-\alpha\vec 1\rangle\leq f_1^L(\vec x+\alpha\vec 1)+f_1^L(\vec x-\alpha\vec 1)=2f_1^L(\vec x)=2\langle\vec v,\vec x\rangle,$$
	thus 
	$$\partial f_1^L(\vec x)\subseteq \partial f_1^L(\vec x+\alpha\vec 1),\quad \partial f_1^L(\vec x)\subseteq \partial f_1^L(\vec x-\alpha\vec 1).$$
	 Since $\vec x\in \mathbb{R}^n\backslash\{\vec 0 \}$ guarantees that at least one of \textcolor{black}{$D_+(\vec x)\neq\emptyset$ and $D_-(\vec x)\neq\emptyset$} holds, we assume, without loss of generality, that \textcolor{black}{$D_+(\vec x)\neq\emptyset$} is satisfied. Thus, a suitable $\alpha\in(0,\|\vec x\|_{\infty})$ can be chosen so that all components of $\vec y=\vec x-\alpha\vec 1$ are nonzero, and each component satisfies
	\textcolor{black}{$$\left\{
	\begin{aligned}
		y_i>0,\;\;\; &\text{if }i\in D_+(\vec x),\\
		y_i<0,\;\;\; &\text{if }i\in (D_+(\vec x))^c.
	\end{aligned}
	\right.
	$$}
	Now we have $\partial f_1^L(\vec x)\subseteq \partial f_1^L(\vec y)\subseteq\partial f_1^L(\vec z)$, $\partial f_2^L(\vec x)\subseteq \partial f_2^L(\vec y)\subseteq\partial f_2^L(\vec z)$, $\partial g_1^L(\vec x)\subseteq \partial g_1^L(\vec y)\subseteq\partial g_1^L(\vec z)$ and $\partial g_2^L(\vec x)\subseteq \partial g_2^L(\vec y)\subseteq\partial g_2^L(\vec z)$, where \textcolor{black}{$\vec z=\sign(\vec y)=\vec 1_{D_+(\vec x),(D_+(\vec x))^c}$}. Therefore, $(\mu,\vec z)$ is an eigenpair.
\end{proof}

\begin{proof}[Proof of Theorem~\ref{thm:eigen-2cut}]
 The proof is given in the following three cases.
	\begin{enumerate}
		\item 	Let  
		\begin{equation}
			\label{eq:prob1}
			q^*=\min_{A\in\mathcal{P}(V)\backslash\{\emptyset\}}\frac{F(A)}{G(A)}=\min\limits_{\vec x\in\mathbb{R}^n\backslash\{\vec 0\}} Q(\vec x),\quad Q(\vec x)=\frac{f_1^L(\vec x)-f_2^L(\vec x)}{g_1^L(\vec x)-g_2^L(\vec x)}
		\end{equation}
		 denote the 2-cut problem, and we first prove $\mu_1\geq q^*$. According to Lemma~\ref{lem:eigen-cut} and Theorem~\ref{thm:eigen-value}, there exists an eigenvector $\vec x\in\{-1,1\}^n$ corresponding to $\mu_1$, thus there exists $S\in\mathcal{P}(V)\backslash\{\emptyset\}$ satisfying 
		$$\mu_1=\frac{f_1^L(\vec x)-f_2^L(\vec x)}{g_1^L(\vec x)-g_2^L(\vec x)}=\frac{F(S)}{G(S)}\geq \min_{A\in\mathcal{P}(V)\backslash\{\emptyset\}}\frac{F(A)}{G(A)}=q^*.$$
		Now we prove $\mu_1\leq q^*$. Select $S\in\mathcal{P}(V)\backslash\{\emptyset\}$ satisfying \eqref{eq:prob1}, and then we have $q^*=Q(\vec x^*)$ with $\vec x^*=\vec 1_{S,S^c}$. For any $\vec x\in\mathbb{R}^n$, $\vec v_2\in\partial f_2^L(\vec x^*)$ and $\vec u_1\in\partial g_1^L(\vec x^*)$, let
		$$L(\vec x)=f_1^L(\vec x)+q^* g_2^L(\vec x)-\langle\vec v_2,\vec x\rangle-q^*\langle\vec u_1,\vec x\rangle.$$
		We have $L(\vec x^*)=L(\vec 0)=0$, and for any $\vec x\neq\vec 0$, we have
		$$L(\vec x)\geq f_1^L(\vec x)+q^* g_2^L(\vec x)-(f_2^L(\vec x)+q^* g_1^L(\vec x))=(g_1^L(\vec x)-g_2^L(\vec x))(Q(\vec x)-q^*)\geq 0.$$
		Based on the convexity of $L(\vec x)$, we have
		$$\vec 0\in\partial f_1^L(\vec x^*)+q^* \partial g_2^L(\vec x^*)-(\partial f_2^L(\vec x^*)+q^* \partial g_1^L(\vec x^*))\Leftrightarrow\vec x^*\in\argmin_{\vec x\in\mathbb{R}^n}L(\vec x),$$
		which means $(q^*,\vec x^*)$ is an eigenpair and $\mu_1\leq q^*$. As for obtaining optimal binary cut from $(\mu_1,\vec x)$, Lemma~\ref{lem:eigen-cut} gives the direct answer.
		\item 	Let  
		\begin{equation}
			\label{eq:prob2}
			q^*=\min_{A\in\bc(V)}\frac{F(A)}{G(A)}=\min\limits_{\text{nonconstant } \vec x} Q(\vec x),\quad Q(\vec x)=\frac{f_1^L(\vec x)-f_2^L(\vec x)}{g_1^L(\vec x)-g_2^L(\vec x)}
			\end{equation}
			denote the 2-cut problem, and we first prove $\mu_1\geq q^*$. Since $F(V)=G(V)=0$, we have $\mu_1=0$ and $$f_1^L(\vec 1)-f_2^L(\vec 1)=f_1^L(-\vec 1)-f_2^L(-\vec 1)=g_1^L(\vec 1)-g_2^L(\vec 1)=g_1^L(-\vec 1)-g_2^L(-\vec 1)=0.$$ According to Lemma~\ref{lem:eigen-cut} and Theorem~\ref{thm:eigen-value}, there exists a nonconstant eigenvector $\vec x\in\{-1,1\}^n\backslash\{\vec 1,-\vec 1\}$ corresponding to $\mu_2$, thus there exists $S\in\bc(V)$ satisfying 
			$$\mu_2=\frac{f_1^L(\vec x)-f_2^L(\vec x)}{g_1^L(\vec x)-g_2^L(\vec x)}=\frac{F(S)}{G(S)}\geq \min_{A\in\bc(V)}\frac{F(A)}{G(A)}=q^*.$$
			Now we prove $\mu_2\leq q^*$. Select $S\in\bc(V)$ satisfying \eqref{eq:prob2}, and then we have $q^*=Q(\vec x^*)$ with $\vec x^*=\vec 1_{S,S^c}$. For any $\vec x\in\mathbb{R}^n$, $\vec v_2\in\partial f_2^L(\vec x^*)$ and $\vec u_1\in\partial g_1^L(\vec x^*)$, let
			$$L(\vec x)=f_1^L(\vec x)+q^* g_2^L(\vec x)-\langle\vec v_2,\vec x\rangle-q^*\langle\vec u_1,\vec x\rangle.$$
			We have $L(\vec x^*)=L(\vec 0)=L(t\vec 1)=0$, $\forall t\in\mathbb{R}$, and for any nonconstant $\vec x\in\mathbb{R}^n$, we have
			$$L(\vec x)\geq f_1^L(\vec x)+q^* g_2^L(\vec x)-(f_2^L(\vec x)+q^* g_1^L(\vec x))=(g_1^L(\vec x)-g_2^L(\vec x))(Q(\vec x)-q^*)\geq 0.$$
			Based on the convexity of $L(\vec x)$, we have
			$$\vec 0\in\partial f_1^L(\vec x^*)+q^* \partial g_2^L(\vec x^*)-(\partial f_2^L(\vec x^*)+q^* \partial g_1^L(\vec x^*))\Leftrightarrow\vec x^*\in\argmin_{\vec x\in\mathbb{R}^n}L(\vec x),$$
			which means $(q^*,\vec x^*)$ is an eigenpair and $\mu_2\leq q^*$. Regarding the extraction of the optimal binary cut from 
			$(\mu_2,\vec x)$, since $\vec x$ is nonconstant, it follows that either \textcolor{black}{$D_+(\vec x)\notin\{\emptyset,V\}$ holds or $D_-(\vec x)\notin\{\emptyset,V\}$ holds. Consequently, either $\vec 1_{D_+(\vec x),(D_+(\vec x))^c}$ or $\vec 1_{(D_-(\vec x))^c,D_-(\vec x)}$} is nonconstant.
			\item 	Let  
			\begin{equation}
				\label{eq:prob3}
				q^*=\max_{A\in\mathcal{P}(V)\backslash\{\emptyset\}}\frac{F(A)}{G(A)}=\max\limits_{\vec x\in\mathbb{R}^n\backslash\{\vec 0\}} Q(\vec x),\quad Q(\vec x)=\frac{f_1^L(\vec x)-f_2^L(\vec x)}{g_1^L(\vec x)-g_2^L(\vec x)}
			\end{equation}
			denote the 2-cut problem, and we first prove $\mu_{\max}\leq q^*$. According to Lemma~\ref{lem:eigen-cut} and Theorem~\ref{thm:eigen-value}, there exists an eigenvector $\vec x\in\{-1,1\}^n$ corresponding to $\mu_{\max}$, thus there exists $S\in\mathcal{P}(V)\backslash\{\emptyset\}$ satisfying 
			$$\mu_{\max}=\frac{f_1^L(\vec x)-f_2^L(\vec x)}{g_1^L(\vec x)-g_2^L(\vec x)}=\frac{F(S)}{G(S)}\leq \max_{A\in\mathcal{P}(V)\backslash\{\emptyset\}}\frac{F(A)}{G(A)}=q^*.$$
			Now we prove $\mu_{\max}\geq q^*$. Select $S\in\mathcal{P}(V)\backslash\{\emptyset\}$ satisfying \eqref{eq:prob3}, and then we have $q^*=Q(\vec x^*)$ with $\vec x^*=\vec 1_{S,S^c}$. For any $\vec x\in\mathbb{R}^n$, $\vec v_1\in\partial f_1^L(\vec x^*)$ and $\vec u_2\in\partial g_2^L(\vec x^*)$, let
			$$L(\vec x)=f_2^L(\vec x)+q^* g_1^L(\vec x)-\langle\vec v_1,\vec x\rangle-q^*\langle\vec u_2,\vec x\rangle.$$
			We have $L(\vec x^*)=L(\vec 0)=0$, and for any $\vec x\neq\vec 0$, we have
			$$L(\vec x)\geq f_2^L(\vec x)+q^* g_1^L(\vec x)-(f_1^L(\vec x)+q^* g_2^L(\vec x))=(g_1^L(\vec x)-g_2^L(\vec x))(q^*-Q(\vec x))\geq 0.$$
			Based on the convexity of $L(\vec x)$, we have
			$$\vec 0\in\partial f_1^L(\vec x^*)+q^* \partial g_2^L(\vec x^*)-(\partial f_2^L(\vec x^*)+q^* \partial g_1^L(\vec x^*))\Leftrightarrow\vec x^*\in\argmin_{\vec x\in\mathbb{R}^n}L(\vec x),$$
			which means $(q^*,\vec x^*)$ is an eigenpair and $\mu_{\max}\geq q^*$. As for obtaining optimal binary cut from $(\mu_{\max},\vec x)$, Lemma~\ref{lem:eigen-cut} gives the direct answer.
	\end{enumerate}
\end{proof}
The proof of Theorem \ref{thm:eigen-3cut} is similar, and thus we omit it. 
}

\begin{proposition}
	\label{prop:another-eigenvector}
{If $(\mu,\vec x)$ is an eigenpair for graph 2-cut problems, then $\forall\,\alpha\in\R$, $(\mu,\vec x+\alpha\vec1)$ or $(\mu,\vec x-\alpha\vec1)$ is also an eigenpair.}
\end{proposition}

\begin{proof}
By the proof of Lemma \ref{lem:eigen-cut}, we have 
	$$\partial f_1^L(\vec x)\subseteq \partial f_1^L(\vec x+\alpha\vec 1),\quad \partial f_1^L(\vec x)\subseteq \partial f_1^L(\vec x-\alpha\vec 1)$$
which also holds for $f_2^L$, $g_1^L$ and $g_2^L$.
Therefore, it follows from 
$$\vec 0\in \partial f_1^L(\vec x)+\mu \partial g_2^L(\vec x)-(\partial f_2^L(\vec x)+\mu \partial g_1^L(\vec x))$$
that 
$$\vec 0\in \partial f_1^L(\vec x+\alpha\vec 1)+\mu \partial g_2^L(\vec x+\alpha\vec 1)-(\partial f_2^L(\vec x+\alpha\vec 1)+\mu \partial g_1^L(\vec x+\alpha\vec 1))$$
and
$$\vec 0\in \partial f_1^L(\vec x-\alpha\vec 1)+\mu \partial g_2^L(\vec x-\alpha\vec 1)-(\partial f_2^L(\vec x-\alpha\vec 1)+\mu \partial g_1^L(\vec x-\alpha\vec 1)).$$
\end{proof}

\begin{theorem}\label{th:tx+(1-t)y}
If $(\mu,\vec x)$ is an eigenpair, and $\vec y\in\R^n$ satisfies 
$\partial f_1^L(\vec x)\subseteq \partial f_1^L(\vec y)$, $\partial f_2^L(\vec x)\subseteq \partial f_2^L(\vec y)$, $\partial g_1^L(\vec x)\subseteq \partial g_1^L(\vec y)$ and $\partial g_2^L(\vec x)\subseteq \partial g_2^L(\vec y)$, then $(\mu,t\vec x+(1-t)\vec y)$ is an eigenpair for any $t\in[0,1]$.
\end{theorem}

\begin{proof}
It is clear that \textcolor{black}{$(\mu,\vec y)$} is an eigenpair. For the remaining part, we first 
prove $\partial f_1^L(\vec x)\subseteq \partial f_1^L(t\vec x+(1-t)\vec y)$.

For any $\vec v\in \partial f_1^L(\vec x)$, there holds $\vec v\in \partial f_1^L(\vec y)$. Then, \begin{align*}
\langle\vec v,t\vec x+(1-t)\vec y\rangle&\le f_1^L(t\vec x+(1-t)\vec y)
\\&\le tf_1^L(\vec x)+(1-t)f_1^L(\vec y)
\\&=t\langle\vec v,\vec x\rangle+(1-t)\langle\vec v,\vec y\rangle=\langle\vec v,t\vec x+(1-t)\vec y\rangle
\end{align*}
which implies that $\vec v\in \partial f_1^L(t\vec x+(1-t)\vec y)$. This proves that $\partial f_1^L(\vec x)\subseteq \partial f_1^L(t\vec x+(1-t)\vec y)$. 
Similar property holds for $f_2^L$, $g_1^L$ and $g_2^L$. Then, $(\mu,t\vec x+(1-t)\vec y)$ is also an eigenpair.
\end{proof}

{For graph 2-cut and 3-cut problems considered in this paper—namely, Cheeger cut, maxcut, anti-Cheeger cut and dual Cheeger problem—Table~\ref{tab:cut} lists their equivalent continuous formulations, their derived eigenproblems, and the correspondence between the eigenpairs and the objective graph cut optima.} 


{\setlength{\tabcolsep}{6pt}
	\renewcommand{\arraystretch}{1.7}
	\begin{landscape}
\begin{table}[]
	\caption{Graph 2-cut and 3-cut problems studied in this paper with their equivalent continuous formulations and eigenproblems.}
	\label{tab:cut}
	\begin{tabular}{ccccc}
		\hline
		\multirow{2}{*}{graph   cut}      & \multirow{2}{*}{continuous formulation} & \multirow{2}{*}{eigenproblem} & \multicolumn{2}{c}{eigenpair}              \\ \cline{4-5} 
		&                                  &                               & optimal value   & optimal solution \\ \hline
		\multirow{6}{*}{Cheeger cut}      & \multirow{3}{*}{$h(G)=\min\limits_{\text{nonconstant~}\vec x\in\mathbb{R}^n}\frac{I(\vec x)}{N(\vec x)}$}              & \multirow{3}{*}{$\vec0\in\partial I(\vec x)-\mu\partial N(\vec x)$}           & \multicolumn{2}{c}{$(\mu_2,\vec x)$}                     \\ \cline{4-5} 
		&                                  &                               & \multirow{2}{*}{$\mu_2=h(G)$} & $(D_+(\vec x),(D_+(\vec x))^c)$  or                  \\
		&                                  &                               &                     & $((D_-(\vec x))^c,D_-(\vec x))$                   \\ \cline{2-5} 
		& \multirow{3}{*}{$h(G)=\min\limits_{\text{nonconstant~}\vec x\in\mathbb{R}^n}\frac{\vol(V)\|\vec x\|_{\infty}-I^+(\vec x)}{N(\vec x)}$}              & \multirow{3}{*}{$\vec 0\in \vol(V)\partial\|\vec x\|_\infty-\partial I^+(\vec x)-\mu\partial N(\vec x)$}           & \multicolumn{2}{c}{$(\mu_2,\vec x)$}                     \\ \cline{4-5} 
		&                                  &                               & \multirow{2}{*}{$\mu_2=h(G)$} & $(D_+(\vec x),(D_+(\vec x))^c)$  or                   \\
		&                                  &                               &                     & $((D_-(\vec x))^c,D_-(\vec x))$                   \\ \hline
		\multirow{3}{*}{maxcut}           & \multirow{3}{*}{$h_{\max}(G)=\max\limits_{\vec x\in\mathbb{R}^n\backslash\{0\}^n}\frac{I(\vec x)}{\vol(V)\|\vec x\|_\infty}$}              & \multirow{3}{*}{$\vec 0\in 
		\partial I(\vec x) -\mu\vol(V)\partial\|\vec x\|_\infty$}           & \multicolumn{2}{c}{$(\mu_{\max},\vec x)$}                     \\ \cline{4-5} 
		&                                  &                               & \multirow{2}{*}{$\mu_{\max}=h_{\max}(G)$} & $(D_+(\vec x),(D_+(\vec x))^c)$  or                   \\
		&                                  &                               &                     & $((D_-(\vec x))^c,D_-(\vec x))$                   \\ \hline
		\multirow{3}{*}{anti-Cheeger cut} & \multirow{3}{*}{$h_{\anti}(G)=\max\limits_{\vec x\in\mathbb{R}^n\backslash\{0\}^n}\frac{I(\vec x)}{2\vol(V)\|\vec x\|_\infty - N(\vec x)}$}              & \multirow{3}{*}{$\vec 0\in 
			\partial I(\vec x) -2\mu\vol(V)\partial\|\vec x\|_\infty+\mu\partial N(\vec x)$}           & \multicolumn{2}{c}{$(\mu_{\max},\vec x)$}                     \\ \cline{4-5} 
		&                                  &                               & \multirow{2}{*}{$\mu_{\max}=h_{\anti}(G)$} & $(D_+(\vec x),(D_+(\vec x))^c)$  or                  \\
		&                                  &                               &                     & $((D_-(\vec x))^c,D_-(\vec x))$                  \\ \hline
		\multirow{3}{*}{dual Cheeger}     & \multirow{3}{*}{$1-h^+(G)=\min\limits_{\vec x\in\R^n\backslash\{\vec 0\}}\frac{I^+(\vec x)}{\|\vec x\|}$}              & \multirow{3}{*}{$\vec 0\in\partial I^+(\vec x)- \mu  \partial\|\vec x\|$}           & \multicolumn{2}{c}{$(\mu_1,\vec x)$}                     \\ \cline{4-5} 
		&                                  &                                & \multirow{2}{*}{$\mu_1=1-h^+(G)$} & \multirow{2}{*}{$(D_+(\vec x),D_-(\vec x))$}  \\
		&                                  &                               &                     &           \\ \hline
	\end{tabular}
\end{table}
\end{landscape}
}
\section{Equivalent eigenproblems for Cheeger cut}
\label{sec:Cheeger}
\subsection{Graph 1-Laplacian: the first equivalent eigenproblem for Cheeger cut} 
To facilitate the reader's quick understanding of our  results, in this subsection, we briefly review the graph 1-Laplacian,
and more details can be referred to \cite{HeinBuhler2010,refCSZ-JCM,refC-JCT}.

The graph 1-Laplacian $\Delta_1$, defined as the subgradient of the total variation functional can be written in coordinate form:
$$(\Delta_1\vec x)_i\coloneqq\left\{\left.\sum_{j\in V:\{i,j\}\in E}w_{ij}z_{ij}\right|\begin{array}{l}z_{ij}\in \mathrm{Sgn}(x_i-x_j),\\z_{ij}=-z_{ji}
\end{array} \right\},\;i\in V,$$
\textcolor{black}{and the \emph{eigenvalue problem} for $\Delta_1$  consists in finding  $\mu\in\R$ and $\vec x\ne \vec 0$ such that
$$\Delta_1\vec x\bigcap \mu  \vec d\circ\sgn(\vec x)\ne\emptyset,$$
which is essentially  equivalent to the eigenproblem \eqref{eq:1-lap}.}  {\color{black} Due to the high nonlinearity of the 1-Laplacian eigenvalue problem, we do not know a priori how many
eigenvalues exist \cite{refCSZ-JCM,refC-JCT}. Nevertheless, it is always possible to select $n$ of them as representatives of the whole
spectrum \cite{refC-JCT,refCSZ-Adv}. An approach in selecting typical eigenvalues is via the min-max principle. }

\textcolor{black}{A nonlinear analog to the Courant-Fischer-Weyl min-max principle, originated in the Lusternik-Schnirelman theory \cite{LS34}, is a useful tool for finding certain critical values of a functional \cite{refC-JCT}.}   
\textcolor{black}{Precisely, by applying the Lusternik-Schnirelman theory \cite{refCDM,refDM,refIS,refK,refPa} \textcolor{black}{to the energy functional of $\Delta_1$}, we can define a sequence of \emph{min-max} eigenvalues for $\Delta_1$ as follows (see \eqref{form:iniplus}):
\begin{equation}\label{eq:1-Lap-minmax}
{\color{black}\mu_k(\Delta_1)}\coloneqq  \inf_{\mathrm{genus}(S)\ge k}\sup\limits_{\vec x\in S} \frac{\sum_{\{i,j\}\in E}w_{ij}|x_i-x_j|  }{\sum_{i\in V}d_i|x_i|}= \inf_{\mathrm{genus}(S)\ge k}\sup\limits_{\vec x\in S} \frac{I(\vec x)}{\|\vec x\|},    
\end{equation}}
for $k\geq 1$,
where $\mathrm{genus}(S)$ denotes the 
Krasnoselskii genus of the centrally symmetric compact  subset $S$ in $C(V)\cong
\R^V$. {\color{black}We point out that $\mu_k(\Delta_1)$ is not the $k$-th smallest eigenvalue of $\Delta_1$, but rather the $k$-th smallest min-max eigenvalue of $\Delta_1$. And it should be noted that $\mu_1(\Delta_1)$ and $\mu_2(\Delta_1)$ coincide exactly with the smallest eigenvalue and the second smallest eigenvalue of $\Delta_1$, respectively. } {\color{black}For an eigenvector $\vec x=(x_1, x_2, \cdots, x_n)$, 
a nodal domain of $\vec x$ is defined as a maximal connected subset of  $V_0^+(\vec x)$ (or $V_0^-(\vec x)$ respectively). The nodal domains have two important properties \cite{refCSZ-Adv}: the 1st one is that any Cheeger set can be identified with any  nodal domain of any eigenvector corresponding to $\mu_2(\Delta_1)$; the 2nd one is the Courant-type nodal domain theorem which states that the number of nodal domains of any eigenvector corresponding to $\mu_k(\Delta_1)$ is at most $k+r-1$ with $r$ being the multiplicity of $\mu_k(\Delta_1)$. }

\textcolor{black}{The main motivation of the study of 1-Laplacian is the following identity \cite{HeinBuhler2010,refCSZ-JCM,refC-JCT} (see \eqref{form:iniplus} and the second row of Table~\ref{tab:cut}):
\begin{equation}
h(G)={\color{black}\mu_2(\Delta_1)}=\min\limits_{\text{nonconstant~}\vec x\in\mathbb{R}^n}\frac{I(\vec x)}{N(\vec x)}.
\label{cheeger-contin1} 
\end{equation}}
Applying Definition~\ref{def:eigen-all} into the equivalent formulation~\eqref{cheeger-contin1} leads to \textcolor{black}{
\begin{equation}\label{eq:1-lap-new}
\vec0\in\Delta_1\vec x-\mu\partial N(\vec x),
\end{equation}}
which will be proved to be a reformulation of 1-Laplacian eigenproblem~\ref{eq:1-lap}.

\textcolor{black}{
\begin{proposition}\label{pro:equi-1-lap}
The relationships between the two eigenproblems with respect to graph 1-Lapalcian $\Delta_{1}$ and Cheeger cut are as follows:
\begin{enumerate}[(a)]
\item If $(\mu,\vec x)$ is an eigenpair of $\Delta_1$ \eqref{eq:1-lap}, then for any $c\in\R$, $(\mu,\vec x-c\vec1)$ is an eigenpair of \eqref{eq:1-lap-new}. 
\item Conversely, if $(\mu,\vec x)$ is an eigenpair of \eqref{eq:1-lap-new}, then for any $c\in \mathrm{median}(\vec x)$, $(\mu,\vec x-c\vec1)$ is an eigenpair of $\Delta_1$ \eqref{eq:1-lap}.
\end{enumerate}
\end{proposition}}
\textcolor{black}{
\begin{proof}
According to Proposition~\ref{prop:another-eigenvector}, we know that if  $(\mu,\vec x)$ is an eigenpair of \eqref{eq:1-lap-new}, then $(\mu,\vec x-c\vec1)$ is also an eigenpair of \eqref{eq:1-lap-new} for any $c\in\R$. Since the graph 1-Laplacian $\Delta_1$ and $N(\vec x)$ are translation invariants along the vector $\vec1$, i.e., $$I(\vec x-c\vec 1)=I(\vec x),\quad\Delta_1(\vec x-c\vec1)=\Delta_1\vec x,\quad N(\vec x-c\vec 1)=N(\vec x),\quad  \partial N(\vec x-c\vec 1)=\partial N(\vec x)$$ for any $\vec x\in\R^n$ and any $c\in\R$. And it is easy to calculate that
$$\sum_{i=1}^n(\Delta_1\vec x)_i=0\Rightarrow\Delta_1\vec x\subset\vec 1^\bot,\quad\partial N(\vec x)=\vec d\circ\sgn(\vec x-c\vec 1)\cap \vec1^\bot,\,\forall\, c\in\median(\vec x).$$
Now we prove (a). If $(\mu,\vec x)$ is an eigenpair of $\Delta_1$ with $\mu\ne0$, then by Theorems 2.6 and 2.9 in \cite{refCSZ-JCM}, it is known that $0\in \median(\vec x)$, which is equivalent to $N(\vec x)=\|\vec x\|$. So for the nontrivial eigenvector $\vec x$, we have
$$\vec0\in \Delta_1\vec x-\mu\partial\|\vec x\|\cap \vec1^\bot=\Delta_1\vec x-\mu\vec d\circ\sgn(\vec x)\cap \vec1^\bot= \Delta_1\vec x-\mu\partial N(\vec x),$$
meaning that  $(\mu,\vec x)$ is an eigenpair of \eqref{eq:1-lap-new}, and thus 
for any $c\in\R$, $(\mu,\vec x-c\vec1)$ is an eigenpair of \eqref{eq:1-lap-new}. Finally, we prove (b). Suppose that $(\mu,\vec x)$ is an eigenpair of \eqref{eq:1-lap-new} and take  $c\in\mathrm{median}(\vec x) $, then 
\begin{align*}
	\vec0&\in \Delta_1(\vec x)-\mu\partial N(\vec x)\subseteq\Delta_1(\vec x-c\vec1)-\mu\partial N(\vec x-c\vec1)\\
	&=\Delta_1(\vec x-c\vec1)-\mu\vec d\circ\sgn(\vec x-c\vec1)\cap \vec1^\bot\subseteq \Delta_1(\vec x-c\vec1)-\mu\vec d\circ\sgn(\vec x-c\vec1)    
\end{align*}
indicating that $(\mu,\vec x-c\vec 1)$ is an eigenpair of $\Delta_1$.
\end{proof}
}

Proposition \ref{pro:equi-1-lap} shows that  Eq.~\eqref{eq:1-lap-new} is actually a new formulation of the graph 1-Laplacian eigenproblem.

\subsection{A new equivalent eigenproblem for Cheeger cut}
\label{sec:new-Cheeger}

\textcolor{black}{Recently, a new equivalent formulation of Cheeger constant is presented \cite{SY24} (see \eqref{form:iniplus} and \eqref{form:norm}):
\begin{equation}
h(G)=\min\limits_{\text{nonconstant~}\vec x\in\mathbb{R}^n}\frac{\vol(V)\|\vec x\|_{\infty}-I^+(\vec x)}{N(\vec x)}\label{cheeger-contin2},
\end{equation}
which leads to the following new equivalent eigenproblem for Cheeger cut (see the third row of Table~\ref{tab:cut}):
\begin{equation*}
0\in \vol(V)\partial\|\vec x\|_\infty-\partial I^+(\vec x)-\mu\partial N(\vec x)
\end{equation*}}
and by introducing the signless 1-Laplacian $\Delta^+_1:=\partial I^+(\cdot)$, we get the eigenproblem \eqref{eq:new-Cheeger-equi}. 
It would be convenient to reformulate \eqref{eq:new-Cheeger-equi} in coordinate form: $\exists\,z_{ij}\in \sgn(x_i+x_j) \mbox{ s.t. }z_{ij}=z_{ji}$ and  $\exists\,v_i\in d_i\sgn(x_i-c_x)$, $c_x\in\median(\vec x)$ such that
 \begin{equation}
 	\label{eigen-cheeger}
 	\left\{
 	\begin{aligned}
 		&\sum_{i\in V} v_i=0,\\
 		&\sum\limits_{j:\{i,j\}\in E} z_{ij} +\mu v_i= 0, &i\in D_0,\\
 		&\sum\limits_{j:\{i,j\}\in E} z_{ij} +\mu v_i\in  \vol(V)\sign(x_i)\cdot [0,1],& i\in D_\pm,\\
 		&\sum_{i\in V} \big|\sum\limits_{j:\{i,j\}\in E} z_{ij}+\mu v_i \big|=\vol(V).&
 	\end{aligned}
 	\right.
 \end{equation}
The structure of  eigenvectors of \eqref{eigen-cheeger} is quite complicate. We provide some properties in this section. 


  \begin{theorem}\label{theorem:ch2}
Given a graph $G=(V,E)$, the following properties hold:
 	\begin{enumerate}[(a)]
 		\item\label{th:c-ch} For any nonempty proper subset $A\subseteq V$, taking $\vec x=\vec 1_{A,A^c}$, then 
 		$\mu=\frac{\abs{\partial A}}{\min\{\vol(A),\vol(A^c)\}}$, and  $(\mu,\vec x)$ is an eigenpair.
 		\item\label{tpD-ch} If  $(\mu,\vec x)$  is an eigenpair, then $\frac{\vol(V)\|\vec x\|_{\infty}-I^+(\vec x)}{N(\vec x)}=\mu$.
 		\textcolor{black}{In addition, $(\mu,\vec 1_{V_0^+,V_0^-})$ is also an  eigenpair, where the definition of $V_0^+$ and  $V_0^-$ is shown in \eqref{subset:positive1} and \eqref{subset:negative1} with $t=0$.}
 	\end{enumerate}
 \end{theorem}
 \textcolor{black}{
\begin{proof}
	(a) and (b) follow directly from Lemma~\ref{lem:eigen-cut} and Theorem \ref{thm:eigen-value}.
\end{proof}}

\begin{definition}[vertex cover]
A subset $A\subseteq V$ is called a \emph{vertex cover} if every edge of $G$ has at least one of member of $A$  as an endpoint. 
\end{definition}

\begin{theorem}
Given a connected graph $G=(V,E)$, suppose that $\mu$ is the smallest nonzero eigenvalue, and $\vec x$ is the corresponding eigenvector, then $V_0^+(\vec x)\cup V_0^-(\vec x)$ gives a vertex cover of $G$.
 \end{theorem}

 \begin{proof}
Let $\vec y=\vec 1_{V_0^+,V_0^-}$. By Theorem~\ref{theorem:ch2}~\ref{tpD-ch}, $(\mu,\vec y)$ is also an eigenpair. For $c_{y}\in\median(\vec y)$, we discuss it below:
 	\begin{enumerate}
 		\item If $c_{y}=1$, then $\vol(V_0^+)\geq \frac{1}{2}\vol(V)$, and  $$\mu=\frac{2|E(V_0^+,V_0^-)|+|E(V_0^+,V_0^0)|+|E(V_0^-,V_0^0)|+2|E(V_0^0,V_0^0)|}{2\vol(V_0^-)+\vol(V_0^0)}.$$
 	By Theorem~\ref{theorem:ch2}~\ref{th:c-ch},	 $(\mu_1,\vec 1_{V_0^+,(V_0^+)^c})$ is an eigenpair, where
 		$$\mu_1=\frac{|E(V_0^+,V_0^-)|+|E(V_0^+,V_0^0)|}{\vol(V_0^-)+\vol(V_0^0)}.$$
 Similarly, $(\mu_2,\vec 1_{V_0^-,(V_0^-)^c})$,  is an eigenpair, where 	$$\mu_1=\frac{|E(V_0^+,V_0^-)|+|E(V_0^-,V_0^0)|}{\vol(V_0^-)}.$$
 		By the condition, 
 		$$\mu\leq\min\{\mu_1,\mu_2\}\leq \frac{2|E(V_0^+,V_0^-)|+|E(V_0^+,V_0^0)|+|E(V_0^-,V_0^0)|}{2\vol(V_0^-)+\vol(V_0^0)}\leq \mu,$$
 		and then we get $|E(V_0^0,V_0^0)|=0$.
 		
 		\item If $c_{y}=-1$, then $\vol(V_0^-)\geq \frac{1}{2}\vol(V)$, and  $$\mu=\frac{2|E(V_0^+,V_0^-)|+|E(V_0^+,V_0^0)|+|E(V_0^-,V_0^0)|+2|E(V_0^0,V_0^0)|}{2\vol(V_0^+)+\vol(V_0^0)}.$$
 		Note that $(\mu_1,\vec 1_{V_0^+,(V_0^+)^c})$ is an eigenpair with
 		$$\mu_1=\frac{|E(V_0^+,V_0^-)|+|E(V_0^+,V_0^0)|}{\vol(V_0^+)},$$
 		Analogously, $(\mu_2,\vec 1_{V_0^-,(V_0^-)^c})$ is an eigenpair with $$\mu_1=\frac{|E(V_0^+,V_0^-)|+|E(V_0^-,V_0^0)|}{\vol(V_0^+)+\vol(V_0^0)}.$$
 		By assumption, 
 		$$\mu\leq\min\{\mu_1,\mu_2\}\leq \frac{2|E(V_0^+,V_0^-)|+|E(V_0^+,V_0^0)|+|E(V_0^-,V_0^0)|}{2\vol(V_0^+)+\vol(V_0^0)}\leq \mu,$$
 		which implies $|E(V_0^0,V_0^0)|=0$.
 		
 		\item If $c_{y}=0$, then $\vol(V_0^+),\,\vol(V_0^-)\leq \frac{1}{2}\vol(V)$ and $$\mu=\frac{2|E(V_0^+,V_0^-)|+|E(V_0^+,V_0^0)|+|E(V_0^-,V_0^0)|+2|E(V_0^0,V_0^0)|}{\vol(V_0^+)+\vol(V_0^-)}.$$
 		Clearly, $(\mu_1,\vec 1_{V_0^+,(V_0^+)^c})$ is an eigenpair with $$\mu_1=\frac{|E(V_0^+,V_0^-)|+|E(V_0^+,V_0^0)|}{\vol(V_0^+)},$$
 		and $(\mu_2,\vec 1_{V_0^-,(V_0^-)^c})$ is an eigenpair with 
 		$$\mu_1=\frac{|E(V_0^+,V_0^-)|+|E(V_0^-,V_0^0)|}{\vol(V_0^-)}.$$
 		The condition implies  
 		$$\mu\leq\min\{\mu_1,\mu_2\}\leq \frac{2|E(V_0^+,V_0^-)|+|E(V_0^+,V_0^0)|+|E(V_0^-,V_0^0)|}{\vol(V_0^+)+\vol(V_0^-)}\leq \mu,$$
 		and thus $|E(V_0^0,V_0^0)|=0$.
 	\end{enumerate}
This derives that $V_0^+(\vec x)\cup V_0^-(\vec x)$ is a vertex cover.
 \end{proof}

The theorem below shows a connection between the spectra of graph 1-Laplacian and the eigenproblem \eqref{eq:new-Cheeger-equi}.

\textcolor{black}{
\begin{theorem}\label{eq:Cheeger-1-in-2}
The spectrum of $\Delta_1$ \eqref{eq:1-lap} is included in the spectrum of the  eigenproblem \eqref{eq:new-Cheeger-equi}.  In addition, the first nontrivial eigenvalue of either $\Delta_1$ \eqref{eq:1-lap} or  
\eqref{eq:new-Cheeger-equi} equals $h(G)$.
\end{theorem}}

\begin{proof}
This is a direct consequence of Lemma~\ref{lem:eigen-cut}, Proposition~\ref{pro:equi-1-lap} and Theorems \ref{thm:eigen-value} and \ref{thm:eigen-2cut}.
\end{proof}

Combining Theorem \ref{theorem:ch2} \ref{tpD-ch} and Theorem \ref{eq:Cheeger-1-in-2}, it is easy so see
\begin{proposition}
For any eigenvector $\vec x$ corresponding to the second eigenvalue of \eqref{eq:new-Cheeger-equi}, $(V_0^+(\vec x),(V_0^+(\vec x))^c)$ is a Cheeger cut of $G$.
\end{proposition}

\textcolor{black}{
\begin{theorem}
Assume that  $(\mu,\vec x)$ is an eigenpair of $\Delta_1$ \eqref{eq:1-lap}, if $\{j\in V:x_j=\max_i x_i\text{ or }x_j=\min_ix_i\}$ is a vertex cover of $G$, then $\vec x-(\frac{\max_i x_i+\min_ix_i}{2})\vec 1$ is an eigenvector corresponding to the eigenvalue $\mu$ of the eigenproblem \eqref{eq:new-Cheeger-equi}.   
\end{theorem}}
\textcolor{black}{
\begin{proof}
Let  $A=\{j\in V:x_j=\max_i x_i\}$,  $B=\{j\in V: x_j=\min_ix_i\}$ and let $c=\frac{\max_i x_i+\min_ix_i}{2}$. The assumption says  that $A\cup B$ gives a vertex cover of the graph. We shall prove that $(\mu,\vec x-c\vec1)$ is an eigenpair of \eqref{eq:new-Cheeger-equi}.
First, by Proposition \ref{pro:equi-1-lap}, $(\lambda,\vec x-c\vec1)$ is an eigenpair of \eqref{eq:1-lap-new}. 
The next key step is to show  
$$\vol(V)\|\vec x-c\vec1\|_\infty-I^+(\vec x-c\vec1)=I(\vec x-c\vec1).$$
In fact, without loss of generality, we only need to consider the case of  $c=0$ (otherwise we just change variable $\hat{\vec x}=\vec x-c\vec1$), and in this case, $A\cup B=\{j\in V:|x_j|=\|\vec x\|_\infty\}$ is a vertex cover. Thus, for any edge $\{i,j\}\in E$, either $|x_j|=\|\vec x\|_\infty$ or $|x_i|=\|\vec x\|_\infty$, meaning that $\max\{|x_i|,|x_j|\}=\|\vec x\|_\infty$. In consequence, 
\begin{align*}
I^+(\vec x)+I(\vec x)&=\sum_{\{i,j\}\in E}w_{ij}\big(|x_i+x_j|+|x_i-x_j|\big)
\\&=2\sum_{\{i,j\}\in E}w_{ij}\max\{|x_i|,|x_j|\}
\\&=2\sum_{\{i,j\}\in E}w_{ij}\|\vec x\|_\infty=\vol(V)\|\vec x\|_\infty.
\end{align*}
It is easy to see that for any $\vec y\in \R^n$, $\vol(V)\|\vec y \|_\infty\ge I^+(\vec y )+ I(\vec y )$. Thus, for any $\vec v\in \partial I(\vec x )$ and $\vec v_+\in \partial I^+(\vec x )$, 
\begin{align*}
\langle\vec v +\vec v_+,\vec y-\vec x\rangle& = \langle\vec v  ,\vec y-\vec x\rangle+\langle \vec v_+,\vec y-\vec x\rangle
\\&\le   (I(\vec y )-I(\vec x ))+(I^+(\vec y )-I^+(\vec x ) )
\\&=(I(\vec y )+I^+(\vec y ))-(I(\vec x )+I^+(\vec x ))
\\&\le \vol(V)\|\vec y \|_\infty-\vol(V)\|\vec x \|_\infty,\;\forall \vec y\in\R^n,
\end{align*} 
meaning that $\vec v +\vec v_+\in \partial \big(\vol(V)\|\vec x \|_\infty\big)=\vol(V) \partial\|\vec x \|_\infty$. In consequence, $\vec v\in \vol(V)\partial\|\vec x \|_\infty-\vec v_+\subseteq \vol(V) \partial\|\vec x \|_\infty-\partial I^+(\vec x )$. By the arbitrariness of $\vec v$, we obtain 
$$\Delta_1\vec x=\partial I (\vec x ) \subseteq \vol(V) \partial\|\vec x \|_\infty-\partial I^+(\vec x )=\vol(V) \partial\|\vec x \|_\infty-\Delta_1^+\vec x.$$
Since $(\mu,\vec x)$ is an eigenpair of \eqref{eq:1-lap-new},
$$\vec0\in \Delta_1\vec x-\mu\partial N(\vec x)\subseteq \vol(V) \partial\|\vec x \|_\infty-\Delta_1^+\vec x-\mu\partial N(\vec x).$$
This derives that $(\mu,\vec x)$ is also an eigenpair of the eigenproblem \eqref{eq:new-Cheeger-equi}.  
\end{proof}
}

\section{Higher order spectral theory for dual Cheeger problems
}
\label{sec:dual-Cheeger}

The signless 1-Laplacian $\Delta_1^+$, introduced in \cite{SYZ},  can be written in coordinate form:
$$(\Delta_1^+\vec x)_i\coloneqq\left\{\left.\sum_{j\in V:\{i,j\}\in E}w_{ij}z_{ij}\right|\begin{array}{l}z_{ij}\in \mathrm{Sgn}(x_i+x_j),\\z_{ij}=z_{ji}
\end{array} \right\},\;i\in V.$$
The \emph{eigenvalue problem} for $\Delta_1^+$ is to find  $\lambda\in\R$ and $\vec x\ne \vec 0$ such that
$$\Delta_1^+\vec x\bigcap {\color{black}\mu \vec d\circ}\mathrm{Sgn}(\vec x)\ne\emptyset.$$
This is essentially  equivalent to the eigenproblem \eqref{eq:signless-1}. 
For convenience, we will simply set $w_{ij}=1$ {\color{black}for $\{i,j\}\in E$}. 
One of the main results in \cite{SYZ} gives 
\begin{equation}
1-h^+(G) = \mu_1(\Delta_1^+) =\min\limits_{\vec x\in\R^n\backslash\{\vec 0\}}\frac{I^+(\vec x)}{\|\vec x\|}, \label{eq:dc-discrete} 
\end{equation}
which is one of the advantages of signless 1-Laplacian. {\color{black}Parallel to \eqref{eq:1-Lap-minmax}, the authors \cite{SYZ} applied the Lusternik-Schnirelman theory to define a sequence of min-max eigenvalues \begin{equation}\label{eq:1-Lap+minmax}
\mu_k(\Delta_1^+)\coloneqq   \inf_{\mathrm{genus}(S)\ge k}\sup\limits_{\vec x\in S} \frac{I^+(\vec x)}{\|\vec x\|}, 
\end{equation} for $k\geq 1$. Again, we point out that $\mu_k(\Delta_1^+)$ is not the $k$-th smallest eigenvalue of $\Delta_1^+$, but rather the $k$-th smallest \textbf{min-max} eigenvalue of $\Delta_1^+$. And it should be noted that $\mu_1(\Delta_1^+)$ coincides with the smallest eigenvalue of $\Delta_1^+$. }

Based upon the spectral theory for signless 1-Laplacians in  \cite{SYZ}, we shall further study higher order eigenpairs of 1-Laplacians. 

\vspace{0.17cm}

\vskip 0.3cm
\begin{definition}\label{def:nd}
The {\color{black}support domains of a vector $\vec x=(x_1, x_2, \ldots, x_n)$} 
are defined to be the maximal connected components of the {\color{black}support set $V(\vec x):=\{i\in V:x_i\ne 0\}$}.
\end{definition}


To see the difference between {\color{black}the nodal domains and the support domains}, let us consider a connected bipartite graph. In this case, {\color{black}$\mu_1(\Delta_1^+)=0$}, according to the usual definition for positive and negative nodal domains, 
any eigenvector {\color{black}corresponding to $\mu_1(\Delta_1^+)$} has $n$ nodal domains, while it has only $1$ {\color{black}support domain}. 
The reason, we prefer to use the {\color{black}support domain} for signless 1-Laplacian rather than the {\color{black}nodal domain}, can be seen from the following  Courant-type nodal domain theorem for $\Delta^+_1 $.

\begin{proposition}\label{pro:ternarynu}
Suppose $(\mu^+,\vec x)$ is an eigenpair of the signless graph
1-Laplacian and $D_1,\ldots,D_k$ are {\color{black}support} domains of $\vec x$. Let
$\vec x^i$ and $\hat{\vec x}^i$ be defined as
\[
x^i_j=\begin{cases}
\frac{x_j}{\sum_{j\in D_i(\vec x)}d_j|x_j|}, & \text{ if } \;\ j\in D_i(\vec x), \\
0, & \text{ if } \;\ j\not\in D_i(\vec x),\\
\end{cases}\]
and
\[\hat{x}^i_j=\begin{cases}
\frac{1}{\sum_{j\in D_i(\vec x)}d_j}, & \text{ if } \;\ j\in D_i(\vec x) \text{ and } x_j>0,\\
\frac{-1}{\sum_{j\in D_i(\vec x)}d_j},& \text{ if } \;\ j\in D_i(\vec x) \text{ and } x_j<0,\\
0, & \text{ if } \;\ j\not\in D_i(\vec x),\\
\end{cases}
\]
for $j=1,2,\cdots,n$ and $i=1,2,\cdots, k$. Then both $(\mu^+,\vec x^i)$ and $(\mu^+,\hat{\vec x}^i)$ are eigenpairs, too.
\end{proposition}
\begin{proof}
It can be directly verified that $\sgn(\hat{x}^i_j)\supset \sgn(x^i_j)\supset \sgn(x_j)$ and $\sgn(\hat{x}^i_{j'}+\hat{x}^i_j)\supset  \sgn(x^i_{j'}+x^i_j) \supset  \sgn(x_{j'}+x_j)$, $j'\sim j$, $j=1,2,\cdots,n$, $i=1,2,\cdots,k$. Then, following the proof of Lemma 2.2 in \cite{SYZ}, 
we complete the proof.
\end{proof}

This property means that the sign-preserved characteristic function on any {\color{black}support}  domain of an eigenvector of $\Delta_1^+$ is again an eigenvector with the same eigenvalue.  
The {\color{black}Courant-type} nodal domain theorem for the signless $1$-Laplacian reads
\begin{theorem}\label{th:nodal-domain}
Let $\vec x^k$ be an eigenvector with eigenvalue ${\color{black}\mu_k(\Delta_1^+)}$ and multiplicity $r$,  and let $S(\vec x^k)$ be the number of {\color{black}support} domains of $\vec x^k$. Then we have
$$1\le  S(\vec x^k)\le k+r-1.$$
\end{theorem}
\begin{proof}
Suppose the contrary, that there exists $\vec x^k=(x_1, x_2, \cdots, x_n)$ such that $S(\vec x^k)\ge k+r$. Let $D_1(\vec x^k)$, $\cdots$, $D_{k+r}(\vec x^k)$ be the {\color{black}support} domains of $\vec x^k$. Let $\vec y^i=(y^i_1, y^i_2, \cdots, y^i_n),$ where
\[
y^i_j=\begin{cases}
\frac{1}{\sum_{j\in D_i(\vec x^k)}d_j}, & \text{ if } \;\ j\in D_i(\vec x^k) \text{ and } x_j>0,\\
\frac{-1}{\sum_{j\in D_i(\vec x^k)}d_j},& \text{ if } \;\ j\in D_i(\vec x^k) \text{ and } x_j<0,\\
0, & \text{ if } \;\ j\not\in D_i(\vec x^k) \text{ or } x_j=0,\\
\end{cases}
\]
for $i=1,2,\cdots, k+r$, $j=1,2,\cdots,n$. By the construction, we have:

(1) The {\color{black}support} domain of $\vec y^i$ is the $i$-th {\color{black}support} domain of
$\vec x^k$, i.e., $D(\vec y^i)=D_i(\vec x^k)$,

(2) $D(\vec y^i)\cap D(\vec y^j)=\varnothing,$

(3) By Proposition \ref{pro:ternarynu}, $\vec y^1,\cdots,\vec y^{k+r}$ are all ternary eigenvectors with the same eigenvalue ${\color{black}\mu_k(\Delta_1^+)}$.

Now $\forall\, \vec x=\sum\limits_{i=1}^{k+r}a_i\vec y^i\in X, \forall\, v\in V, \,\exists$ unique $j$ such that  $x_v=a_jy^j_v$. Hence, $|x_v|=\sum_{j=1}^{k+r}|a_j||y^j_v|$. Since $\vec x\in X$, $\vec y^j\in X$, $j=1,\cdots,k+r$, we have
$$
1=\sum_{v\in V}d_v|x_v|=\sum_{v\in V}d_v\sum_{j=1}^{k+r}|a_j||y^j_v|
=\sum_{j=1}^{k+r}|a_j|\sum_{v\in V}d_v|y^j_v|=\sum_{j=1}^{k+r}|a_j|.
$$
Therefore, for any $\vec x\in \mathrm{span}(\vec y^1,\cdots,\vec y^{k+r})\cap X$, we have
\begin{align*}
I^+(\vec x)&=\sum_{u\sim v} |x_u+x_v|
\le \sum_{u\sim v} \sum_{i=1}^{k+r}|a_i||y^i_u+y^i_v|
\\&= \sum_{i=1}^{k+r}|a_i|\sum_{u\sim v}|y^i_u+y^i_v|
= \sum_{i=1}^{k+r}|a_i| I^+(\vec y^i)\le \max\limits_{i=1,2,\cdots,k+r}I^+(\vec y^i).
\end{align*}
Note that $\vec y^1,\cdots,\vec y^{k+r}$ are non-zero orthogonal vectors, so $\mathrm{span}(\vec y^1,\cdots,\vec y^{k+r})$ is a $k+r$ dimensional linear space. It follows that $\mathrm{span}(\vec y^1,\cdots,\vec y^{k+r})\cap X$ is a symmetric set which is {\color{black}homeomorphism to the unit sphere $\mathbb{S}^{k+r-1}$ of dimension $k+r-1$}. Obviously, ${\color{black}\gen}(\mathrm{span}(\vec y^1,\cdots,\vec y^{k+r})\cap X)=k+r$. Therefore, we derive that
\begin{align*}
{\color{black}\mu_{k+r}(\Delta_1^+)}&=\inf\limits_{{\color{black}\gen}(A)\ge k+r}\sup\limits_{\vec x\in A} I^+(\vec x)
\\&\le \sup\limits_{\vec x\in \mathrm{span}(\vec y^1,\cdots,\vec y^{k+r})\cap X} I^+(\vec x)
\\&=\max\limits_{i=1,\cdots,k+r} I^+(\vec y^i)
={\color{black}\mu_k(\Delta_1^+)}.
\end{align*}
It contradicts with ${\color{black}\mu_k(\Delta_1^+)<\mu_{k+r}(\Delta_1^+)}$.
\end{proof}

In \cite{SYZ}, the authors initiate the study of $\widehat{\Delta}_1^+$-eigenproblem:
\begin{equation}\label{eq:hat-Delta_1}
\vec0\in\Delta_1^+\vec x-{\color{black}\mu\cdot}(\Delta_1^+\vec x+\Delta_1\vec x)=(1-{\color{black}\mu})\Delta_1^+\vec x-{\color{black}\mu}\Delta_1\vec x
\end{equation}
and prove that  
\begin{align} 
	1-\widehat{h}^+(G) &=\min\limits_{\vec x\in\R^n\backslash\{\vec 0\}}\frac{I^+(\vec x)}{I^+(\vec x)+I(\vec x)}, \label{eq:mdc-discrete}
	\end{align}
 which yields that the first eigenvalue of $\widehat{\Delta}_1^+$ 
 equals $1-\widehat{h}^+(G) $. 
 
{\color{black}Again, parallel to \eqref{eq:1-Lap-minmax} and \eqref{eq:1-Lap+minmax}, we define the min-max eigenvalues of $\widehat{\Delta}_1^+$ as follows
\[{\color{black}\mu_k}(\widehat{\Delta}_1^+):=\inf_{{\color{black}\gen}(T)\ge k} \max_{\vec x\in T\subseteq X} \frac{I^+(\vec  x)}{I^+(\vec x)+I(\vec x)},\;\; k\ge 1.\]
}
We can also derive a Courant-type nodal domain inequality for $\widehat{\Delta}_1^+$ which is an analog for Theorem \ref{th:nodal-domain}. 
Precisely, we have:

\begin{proposition}\label{pro:nodal-domain-hat1Lap}
Let $\vec x^k$ be an eigenvector corresponding to 
{\color{black}$\mu_k(\widehat{\Delta}_1^+)$}  whose multiplicity is $r$. Then the number of {\color{black}support} domains of $\vec x^k$ does not exceed $k+r-1$.
\end{proposition}

Also, it is not difficult to verify 
\begin{proposition}
 For any $k=1,2,\cdots,n$,  $${\color{black}\mu_k}(\widehat{\Delta}_1^+)\le {\color{black}\mu_k(\Delta_1^+)}\le 2{\color{black}\mu_k}(\widehat{\Delta}_1^+).$$ 
\end{proposition}

The $k$-way Cheeger constant, defined by Miclo \cite{Miclo08}, 
generalizes the Cheeger constant of a graph. In their seminal work, Lee, Gharan and Trevisan \cite{LGT12} proved the $k$-way Cheeger inequality, which not only solves an open problem proposed by Miclo, but also brings  the powerful tool of random metric partitions developed originally in theoretical computer science. 
As a dual analog,  the $k$-way dual Cheeger constant  introduced in \cite{Liu15}, 
\begin{equation}\label{eq:k-waydualCheeger}
h^+_k:=\max\limits_{\text{ disjoint } (V_1,V_2),\ldots,(V_{2k-1},V_{2k})}\min\limits_{1\le i\le k}\frac{2|E(V_{2i-1},V_{2i})|}{\vol(V_{2i-1}\cup V_{2i})}
\end{equation}
is a direct generalization of the dual Cheeger constant. And similarly, there is a $k$-way dual Cheeger inequality which is a `dual' version of Lee-Gharan-Trevisan's Theorem \cite{Liu15}. Below, we show another $k$-way dual Cheeger inequality regarding the spectrum of the singless 1-Laplacian.
 \begin{theorem}\label{thm:k-way-dual-Cheeger}
{\color{black}Suppose that $\vec x$ is an eigenvector corresponding to $\mu_k(\Delta_1^+)$}. 
Then
$$  1-h^+_{\textsl{S}(\vec x)}\le {\color{black}\mu_k(\Delta_1^+)}\le 1-h^+_k,\,\,\,\,\,\,\,\forall\, k. $$
\end{theorem}

\begin{proof}
Given $\{(V_{2i-1},V_{2i})\}_{i=1}^k$, we define the  $k$-dimensional linear subspace $E_k=\mathrm{span}\{\hat{\vec 1}_{V_{2i-1},V_{2i}}\,|\,1\le i\le k\}$. Then, 
${\color{black}\gen}(E_k\cap X)=k$, and hence
\begin{align*}
{\color{black}\mu_k(\Delta_1^+)}&=\inf_{{\color{black}\gen}(A)\ge k} \max_{\vec x\in A} I^+(\vec x)\\
&\le \min_{\{(V_{2i-1},V_{2i})\}_{i=1}^k}\max_{\vec x\in E_k\cap X} I^+(\vec x)\\
&=\min_{\{(V_{2i-1},V_{2i})\}_{i=1}^k}\max_{\sum^k_{i=1}|a_i|=1}I^+(\sum^k_{i=1}a_i\hat{\vec 1}_{V_{2i-1},V_{2i}})\\
&\le \min_{\{(V_{2i-1},V_{2i})\}_{i=1}^k}\max_{\sum^k_{i=1}|a_i|=1}\sum^k_{i=1}|a_i|I^+(\hat{\vec 1}_{V_{2i-1},V_{2i}})\\
&\le \min_{\{(V_{2i-1},V_{2i})\}_{i=1}^k}\max_{1\le i \le k}I^+(\hat{\vec 1}_{V_{2i-1},V_{2i}})
=1-h_k.
\end{align*}
Let $D_1,\cdots,D_m$ be the {\color{black}support} domains of an eigenvector $\vec x$ corresponding to ${\color{black}\mu_k(\Delta_1^+)}$. Suppose $D_i=D_i^+\cup D_i^-$ where $D_i^\pm=\{j\in D_i:\pm x_j>0\}$, $i=1,\cdots,m$. Then $S(\vec x)=m$, and it  can be verified that 
$$1-\frac{2|E(D_i^+,D_i^-)|}{\vol(D_i^+\cup D_i^-)}=I^+(\hat{\vec 1}_{D_i^+,D_i^-})={\color{black}\mu_k(\Delta_1^+)}.$$
Consequently, 
$$1-h_m\le \max\limits_{i=1,\cdots,m}\left(1-\frac{2|E(D_i^+,D_i^-)|}{\vol(D_i^+\cup D_i^-)}\right)={\color{black}\mu_k(\Delta_1^+)}.$$
\end{proof}

\begin{itemize}
    \item Theorem \ref{thm:k-way-dual-Cheeger} shows an easy proof of the Courant-type nodal domain  theorem. Indeed, if $S(\vec x)\ge k+r$, then ${\color{black}\mu_{k+r}(\Delta_1^+)}\le 1-h_{k+r}^+\le 1- h^+_{\textsl{S}(\vec x)}\le {\color{black}\mu_k(\Delta_1^+)}$, which leads to a contradiction with the assumption that ${\color{black}\mu_k(\Delta_1^+)}<{\color{black}\mu_{k+r}(\Delta_1^+)}$.
    
    \item Taking $k=1$ in Theorem \ref{thm:k-way-dual-Cheeger},  by $S(\vec x)\ge 1$, we have  $1-h_1^+={\color{black}\mu_1(\Delta_1^+)}$, which is nothing but the fundamental identity relating the first eigenvalue and the dual Cheeger constant. Moreover, if $S(\vec x)=m$, then $h_1^+=\cdots=h_m^+$ and ${\color{black}\mu_1(\Delta_1^+)=\cdots=\mu_m(\Delta_1^+)}$.  
    \item Theorem \ref{thm:k-way-dual-Cheeger} is a dual version of the $k$-way Cheeger inequality (Theorem 8 in \cite{refCSZ-Adv}).
\end{itemize}

 We present a higher order inequality linking the min-max eigenvalues $\{{\color{black}\mu_k(\Delta_1^+)}\}_{k=1}^n$ of $\Delta_1^+$, and the eigenvalues $\{{\color{black}\lambda_k(\Delta_2)}\}_{k=1}^n$ of the normalized Laplacian $\Delta_2$: 
 \begin{theorem}\label{thm:2Lap-s1Lap}{\color{black}
$$  \frac{\mu_k^2(\Delta_1^+)}{2}\le 2-\lambda_{n-k+1}(\Delta_2)\le 2\mu_k(\Delta_1^+),\,\,\,\,\,\,\,\forall\, k. $$}
\end{theorem}

\begin{proof}
{\color{black}Suppose that $(\mu_k,\vec y)$ is an eigenpair of  $\Delta_1^+$. Let $\vec x=(x_1,\cdots,x_n)$ be defined via $x_i:=\sqrt{|y_i|}\mathrm{sgn}(y_i)$, $\forall i=1,\cdots,n$.} 
The key inequalities for proving Theorem \ref{thm:2Lap-s1Lap} are
\[ 2-\frac{\sum_{\{i,j\}\in E}\left(x_i-x_j\right)^2}{\sum_{i\in V}d_ix_i^2}\le 2\frac{\sum_{\{i,j\}\in E}\left|x_i^{2}\mathrm{sign}(x_i)+x_j^2\mathrm{sign}(x_j)\right|}{\sum_{i\in V}d_ix_i^2}\]
and \[\frac12\left(\frac{\sum_{\{i,j\}\in E}\left|x_i^{2}\mathrm{sign}(x_i)+x_j^2\mathrm{sign}(x_j)\right|}{\sum_{i\in V}d_ix_i^2}\right)^2\le2-\frac{\sum_{\{i,j\}\in E}\left(x_i-x_j\right)^2}{\sum_{i\in V}d_ix_i^2}\]
where $\mathrm{sign}(x_i)$ indicates the usual signature of $x_i$.  
Their proofs are elementary, and the rest proof is the same as that of Theorem 1.1 in \cite{Zhang23+}, and thus we omit the details.
\end{proof}

\begin{remark}\label{rem:dual-Cheeger-inequ}
Combining Theorems \ref{thm:k-way-dual-Cheeger} and \ref{thm:2Lap-s1Lap}, we immediately obtain
\[\frac{(1-h^+_{\textsl{S}(\vec x)})^2}{2}\le 2-\lambda_{n-k+1}(\Delta_2)\le 2(1-h^+_k)\]
which is actually an analog of 
the main theorem in \cite{Liu15}. Here, $\vec x$ can be taken as any eigenvector corresponding to ${\color{black}\mu_k(\Delta_1^+)}$.
\end{remark}

The next theorem establishes the equality case for Theorem \ref{thm:k-way-dual-Cheeger}  on special graphs. 
\begin{theorem}
Let $G$ be a forest graph. Then 
$${\color{black}\mu_k(\Delta_1^+)}=1-h_k^+,\;k=1,\cdots,n.$$
\end{theorem}

\begin{proof}
First, since every forest graph is bipartite, it follows from Proposition 3.1 in \cite{Liu15} that $h_k^+=1-h_k$. 
On the other hand, by Theorem 2.7 in \cite{SYZ}, we further have ${\color{black}\mu_k(\Delta_1)=\mu_k(\Delta_1^+)}$. Finally, by the results in \cite{DPT23}, we know ${\color{black}\mu_k(\Delta_1)}=h_k$ on forests. 

All the three identities hold on forests, and therefore
$${\color{black}\mu_k(\Delta_1^+)}={\color{black}\mu_k(\Delta_1)}=h_k=1-h_k^+.$$
\end{proof}

Thus, for forest, the higher order dual Cheeger inequality presented in Remark \ref{rem:dual-Cheeger-inequ} can be improved as:
\begin{proposition}\label{pro:higher-dual-C}
For any $k=1,\cdots,n$, the dual Cheeger inequality 
\[\frac{(1-h^+_{k})^2}{2}\le 2-\lambda_{n-k+1}(\Delta_2)\le 2(1-h^+_k)\]
    holds on a forest graph.
\end{proposition}
We remark here that, on forest graphs, Proposition \ref{pro:higher-dual-C}  improves Liu's higher order dual Cheeger inequality \cite{Liu15}, where in his original version on general graphs, the factor appearing in  the lower bound is a much smaller number that depends on $k$;  while in our version on forest graphs, the factor is simply $1/2$ which is independent of $k$. 
Besides, it is somehow a reformulation of Miclo's inequality $h_k^2/2\le \lambda_k(\Delta_2)$ on trees \cite{Miclo08,DJM12}. 

\vspace{0.3cm}

Regarding the Laplacian spectrum or adjacency spectrum, we also would like to review Cvetkovi\'{c}'s inertia bound for independence number (recall that a subset of $V$ is called independent if its vertices are mutually non-adjacent; and the  independence number is the cardinality of the largest independent set):
Let $$\mathcal{M}=\{ M=(m_{ij})_{1\leq i,j\leq n} \text{ is a symmetric matrix s.t. $m_{ij}=0$ if $i\nsim j$}\}.$$ 
    Then the independence number $\alpha(G)$  satisfies
    $$\alpha(G)\leq \min_{M\in \mathcal{M}}\min\{n-n^+(M),n-n^-(M)\}$$
where $n^+(M)\left(\text{resp., } n^-(M) \right)$ denotes the number of positive (resp., negative) eigenvalues of the matrix $M$. 
The normalized  Laplacian version of Cvetkovi\'{c}'s inertia bound can be written as
$$\alpha(G)\le \min\{\#\{k:\lambda_k(\Delta_2)\le 1\},\#\{k:\lambda_k(\Delta_2)\ge 1\}\}.$$
However,  both adjacency spectrum and Laplacian spectrum cannot give upper bounds of the independence number by means of  the eigenvalue multiplicity.  While, 1-Laplacian  spectrum can provide a Cvetkovi\'{c}-type bound by the eigenvalue multiplicity \cite{Zhang23+}.

For normalized signless 1-Laplacian, we have a special estimate for the multiplicity of the maximum eigenvalue by means of independence number and edge covering number.
\begin{theorem}\label{thm:multi-of-1}
Let $\alpha$ be the  independence number of a graph $G$. Then,
\begin{enumerate}
\item The multiplicity of the eigenvalue $1$ of the normalized signless 1-Laplacian  $\Delta_1^+$ is at least $\alpha(G)$, and at most the edge covering number. 

\item If $G$ is further assumed to be a bipartite connected graph, then the multiplicity of the eigenvalue $1$ for $\Delta_1^+$ equals $\alpha(G)$.
\end{enumerate}    
\end{theorem}

\begin{proof}\begin{enumerate}
    \item We simply write $\alpha:=\alpha(G)$. Let $t$ denote the multiplicity of the eigenvalue $1$ for $\Delta_1^+$. To prove that $t\ge \alpha$, we need to show that ${\color{black}\mu_k(\Delta_1^+)}=1$ for all $k\ge n-\alpha+1$.

Note that in this case, $n-k+1\le \alpha$, then there exist at least $n-k+1$ non-adjacent vertices in $G$. We denote them by $v_1,\ldots,v_{n-k+1}$. Let $\mathbf{1}_v\in \R^n$ be the indicator function of $v$, that is, $(\mathbf{1}_v)_v=1$ and $(\mathbf{1}_v)_i=0$ for all $i\ne v$. Let also
 $$X_{n-k+1}\coloneqq \mathrm{span}( \mathbf{1}_{v_1},\ldots,\mathbf{1}_{v_{n-k+1}}).$$ 
 
Note that, for each $\vec x$ in $X_{n-k+1}$, there exists  $(l_1,\ldots,l_{N-k+1})\ne \mathbf{0}$ such that $$\vec x=l_1\mathbf{1}_{v_1}+\ldots+l_{n-k+1}\mathbf{1}_{v_{n-k+1}}.$$ 
Therefore,  the intersection property of the  genus 
implies 
that $$A\cap X_{n-k+1}=A\cap \mathrm{span}( \mathbf{1}_{v_1},\ldots,\mathbf{1}_{v_{N-k+1}})\ne\emptyset$$ for any $A$ with $\mathrm{genus}(A)\ge k$. It is easy to see that $I^+(\vec x)=\|\vec x\|$, $\forall\vec x\in X_{n+1-k}$.  Hence, $$\sup_{\vec x\in A}\frac{I^+(\vec x)}{\|\vec x\|}\ge
\inf_{\vec x\in X_{n+1-k}}\frac{I^+(\vec x)}{\|\vec x\|}=1,$$  implying that ${\color{black}\mu_k(\Delta_1^+)}=1$. This shows that $t\geq \alpha$.\newline


Let $\{v_1,v_2\}$, $\cdots$, $\{v_{2\beta-1},v_{2\beta}\}$ be $\beta$ edges that realize the maximum matching number. Consider the linear subspace 
$$X=\mathrm{span}( \mathbf{1}_{v_1}-\vec 1_{v_2},\cdots,\mathbf{1}_{v_{2\beta-1}}-\vec 1_{v_{2\beta}}).$$
Then $\dim X=\beta$. And every nonzero vector $\vec x\in X$ is an eigenvector  corresponding to the largest eigenvalue. 
Note that for any centrally symmetric subset $S\subset\R^n$ with ${\color{black}\gen}(S)\ge n-\beta+1$, $S\cap (X\setminus0)\ne\varnothing$. 
Hence, 
\begin{align*}
\mu_{\beta}(\Delta_1^+)&=
\inf\limits_{ S:{\color{black}\gen}(S)\ge \beta}\sup\limits_{\vec y\in S}\frac{I^+(\vec y)}{\|\vec y\|} \le \sup\limits_{\vec x\in X\setminus0}\frac{I^+(\vec x)}{\|\vec x\|}
\\&\le
\max_{1\le i\le \beta} \frac{\deg(v_{2i-1})+\deg(v_{2i})-2}{\deg(v_{2i-1})+\deg(v_{2i})} <1 
\end{align*}
which implies that {\color{black}the number of times 1 appears in the sequence $\mu_1(\Delta_1^+)\le\cdots\le\mu_n(\Delta_1^+)$} is at most $n-\beta=\eta$. 
\item 
In fact, for a bipartite connected  graph $G=(V,E)$, by Hall's marriage theorem, it is easy to see that $\alpha=\eta$. Therefore, we have $t=\alpha=\eta$.   
\end{enumerate}
\end{proof}

{Some of the results on the signless 1-Laplacian can be extended to $\widehat{\Delta}_1^+$. For example,  we also have a sequence of inequalities linking the min-max eigenvalues of $\widehat{\Delta}_1^+$ and a multi-way version of $\hat{h}^+$.}
\begin{proposition}
{\color{black}Suppose that $\vec x$ is an eigenvector corresponding to $\mu_k(\widehat{\Delta}_1^+)$}. 
Then
$$  1-\hat{h}^+_{\textsl{S}(\vec x)}\le {\color{black}\mu_k}(\widehat{\Delta}_1^+)\le 1-\hat{h}^+_k,\,\,\,\,\,\,\,\forall\, k, $$
where 
$$\hat{h}^+_k:=\max\limits_{\text{ disjoint } (V_1,V_2),\ldots,(V_{2k-1},V_{2k})}\min\limits_{1\le i\le k}\frac{2|E(V_{2i-1},V_{2i})|+|E(V_{2i-1}\cup V_{2i},(V_{2i-1}\cup V_{2i})^c)|}{\vol(V_{2i-1}\cup V_{2i})+|E(V_{2i-1}\cup V_{2i},(V_{2i-1}\cup V_{2i})^c)|}$$
\end{proposition}

\section{\textcolor{black}{Spectral theory for maxcut} 
}
\label{sec:maxcut}


\textcolor{black}{
In this section, we devote to investigate the eigenvalue problem for $\Delta_1$ restricted to the boundary of the $l^\infty$-norm polytope $X_{\infty}$ \eqref{set:infty}, that is, the eigenproblem \eqref{eq:1-lap-maxcut} (see also the fourth row of Table~\ref{tab:cut}).
}


 \textcolor{black}{The vertex set of a connected component of the subgraph induced by $D_\pm(\vec x)$ \eqref{eq:dplusminus} and $D_0(\vec x)$ \eqref{eq:dzero} is called the $\pm$ {\color{black}extremal} domain and the null domain, respectively.} Accordingly, we divide $V$ into $r^++r^-+r^0$ disjoint $\pm$ {\color{black}extremal} domains and null domains
\begin{equation}
\label{eq:maxcut-nodal}
V=\Big(\bigcup_{\alpha=1}^{r^+} D_+^\alpha\Big)\bigcup \Big(\bigcup_{\beta=1}^{r^-}D_-^\beta\Big) \bigcup \Big(\bigcup_{\omega=1}^{r^0} D_0^\omega\Big),
\end{equation}
where $D_\pm^\gamma $ is a $\pm$ {\color{black}extremal} domain and $D_0^\omega$ is a null domain, and $r^\pm$, $r^0$ are the numbers of $\pm$ {\color{black}extremal} domains and the null domains, respectively. The reason, we prefer to use the {\color{black}extremal} domains and null domains for 1-Laplacian w.r.t. $l^\infty$-norm rather than the nodal domains, can be seen from the Courant-type nodal domain theorem (see Theorem \ref{thm:Courant}). 
Moreover, the vertex sets of connected components of the subgraphs induced by $ D_+(\vec x)\cup D_-(\vec x)$, denoted by  $D_1,D_2,\cdots,D_r$ with $D_i=A_i\cup B_i$, $A_i\subseteq D_+(\vec x)$ and $B_i\subseteq D_-(\vec x)$ for $i=1,2,\cdots,r$. For convenience, we let
\[
S(\vec x)=r^++r^-, \;\;\; S_0(\vec x)=r^0, \;\;\; S'(\vec x)=r.
\]

The adjacent relations connecting {\color{black}extremal} domains and  null domains can be summarized as the following equivalent statements:
\begin{enumerate}[(S1)]
\item\label{statement:S1} $D_+^\alpha$ and $D_+^\beta$ have no connection, $\forall \alpha \neq \beta$. The same is true if $D_+$ is replaced by $D_-$ or $D_0$.
\item If $j\sim i\in D_+^\alpha$ and $j\notin D_+^\alpha$, then $j\in D_-\cup D_0$. The same is true if $j\sim i\in D_-^\beta$ and $j\notin D_-^\beta$, then $j\in D_+\cup D_0$ and if $j\sim i\in D_0^\omega$ and $j\notin D_0^\omega$, then $j\in D_+\cup D_-$.
\item If $i\in D_+^\alpha$, then the \ensuremath{i}-th summation $\sum_{j\sim i} z_{ij}$ depends only on the connections inside $D_+^\alpha$ and those connection to $D_-\cup D_0$.
\end{enumerate}


Now we turn to study the    eigenvalue problem \eqref{eq:1-lap-maxcut}.
First of all, we shall compute the subdifferential \textcolor{black}{$\partial \|\vec x\|_\infty$}.


\begin{proposition}\label{prop:partial_infty}
Given  $\vec x\in \mathbb{R}^n\setminus \{0\}$, then the subgradient of  $\norm{\vec x}$ reads
\begin{equation}\label{eq:pinfty}
\partial \norm{\vec x} = \left\{\vec u\in \mathbb{R}^n\left|  \normone{\vec u}=1, u_i =
\begin{cases}
0, &    i\in D_0(\vec x)\\
\sign(x_i)\cdot [0,1],&   i\in D_\pm(\vec x)
\end{cases}
\right.
\right\}.
\end{equation}
\end{proposition}

\begin{proof}
Let $\vec x=(x_1, \cdots, x_n),  \vec y=(y_1,\cdots, y_n)$ and
$$ \delta_\pm  = \{j_\pm \in D_\pm \mid \pm y_{j_\pm}=\max_{k\in D_\pm} \{\pm y_k\} \}.
$$
\textcolor{black}{We have
$$
\frac{d \|\vec x+t\vec y\|_\infty}{dt}\Big|_{t=0}=\max\{y_j, -y_k\,|\, j\in \delta_+, k\in \delta_-\}.
$$
Thus $\vec u=(u_1, \cdots, u_n)\in \partial \|\vec x\|_\infty$ if and only if
\begin{equation}\label{eq:zz}
\max\{y_j, -y_k\,|\, j\in \delta_+, k\in \delta_-\}\ge \sum^n_{i=1}u_i y_i,\,\, \forall\, \vec
y\in \mathbb{R}^n.
\end{equation}
Thus $\vec e_i, i\in D_+$ and $-\vec e_j, j\in D_-$ are all in $\partial \|\vec x\|_\infty$.
$ \partial \|\vec x\|_\infty$ is convex, and it follows $\sum_{i\in D_+} u_i \vec e_i -\sum_{i\in D_-} u_i \vec e_i \in \partial \|\vec x\|_\infty$ with $u_i\ge 0, \sum_{i\in D_+\cup D_-} u_i = 1$.}
Here $\{\vec e_i\}$ is the Cartesian basis of $\mathbb{R}^n$.

On the other hand, taking $y_i=0, i\in D_+\cup D_-$ in \eqref{eq:zz},
we have $u_i=0, \forall\, i\in D_0$. And, setting $\vec y = -\vec e_i,\forall i\,\in D_+$ (resp. $\vec y =\vec e_j, \forall j\,\in D_-$) in \eqref{eq:zz} yields $u_i\ge 0$ (resp. $u_j\le 0$). Finally, taking $\vec y= \sum_{i\in D_+}\vec e_i -\sum_{i\in D_-} \vec e_i$ and $-\sum_{i\in D_+}\vec e_i +\sum_{i\in D_-} \vec e_i$ in \eqref{eq:zz}, we have $\|\vec u\|_1=1$.
\end{proof}

The coordinate form of the eigenvalue problem \eqref{eq:1-lap-maxcut}, or eqivalently, $$\Delta_1\vec x\cap \mu \vol(V)\partial\|\vec x\|_\infty\ne\varnothing,$$ reads as $\exists\,z_{ij}\in \sgn(x_i-x_j) \mbox{ with }z_{ij}=-z_{ji}$~such that
\begin{align}
\label{eq:1LinftyN1}&\sum\limits_{j\sim i} z_{ij} = 0, &i\in D_0,\\
\label{eq:1LinftyN2}&\sum\limits_{j\sim i} z_{ij} \in \mu \vol(V)\sign(x_i)\cdot [0,1],& i\in D_\pm,\\
\label{eq:1LinftyN3}&\sum\limits_i^n \big|\sum\limits_{j\sim i} z_{ij} \big|=\mu\vol(V).&
\end{align}

The following result indicates that the eigenvectors are very abundant.

\textcolor{black}{
\begin{lemma}\label{lemma:preceq}
If $(\mu,\vec x)$ is an eigenpair of  \eqref{eq:1-lap-maxcut}  on $G=(V,E)$ and $\vec x \preceq \vec y$, then $(\mu,t\vec x+(1-t)\vec y)$ is also an eigenpair of  \eqref{eq:1-lap-maxcut}  for any $ t\in[0,1]$. Moreover, there exists a binary vector $\hat{\vec x}$ such that $(\mu,\hat{\vec x})$ is an eigenpair.
\end{lemma}
}

\textcolor{black}{
\begin{proof}
It is straightforward to verify that from the definition of $\vec x \preceq \vec y$ (see \eqref{cond1:plusminus} and \eqref{cond2:order}) that $\partial I(\vec x)\subseteq \partial I(\vec y)$ and $\partial \|\vec x\|_{\infty}\subseteq \partial\|\vec y\|_{\infty}$ indeed hold. Therefore, the conclusion follows from Theorem \ref{th:tx+(1-t)y} and Lemma~\ref{lem:eigen-cut}.
\end{proof}}

\begin{theorem}\label{theorem:spct2}
Given a graph $G=(V,E)$, the following statements hold:
\begin{enumerate}[(a)]
\item For every eigenvalue $\mu$ of  \eqref{eq:1-lap-maxcut}, we have $0\leq \mu\leq 1$.
\item The difference of any two eigenvalues is at least $\frac{2}{\vol(V)}$. Therefore,  \eqref{eq:1-lap-maxcut}  has at most $\frac{\vol(V)}{2}+1$ distinct eigenvalues.
\item\label{th:c} For any subset $A\subseteq V$, let $\vec x=\vec 1_{A,A^c}$ and
$\mu=\frac{2\abs{\partial A}}{\vol(V)}$, then $(\mu,\vec x)$ is an eigenpair of  \eqref{eq:1-lap-maxcut} .
\item\label{tpD} If $(\mu,\vec x)$is an eigenpair of  \eqref{eq:1-lap-maxcut}, then $\frac{I(\vec x)}{\vol(V)\norm{\vec x}}=\mu$ and $\frac{\mu}{2}=\frac{\abs{\partial D_-}}{\vol(V)}=\frac{\abs{\partial D_+}}{\vol(V)}$.
Moreover, both $(\mu,\vec 1_{D_+,D_-})$ and $(\mu,\vec 1_{D_+,D_+^c})$ are eigenpairs.
\end{enumerate}
\end{theorem}

\begin{proof} ~
\begin{enumerate}[(a)]
\item \textcolor{black}{From Lemma~\ref{lemma:preceq}, for every eigenvalue $\mu$ of \eqref{eq:1-lap-maxcut}, there exists a subset $A\subseteq V$ such that}
$$0\leq \mu=\frac{2\abs{\partial A}}{\vol(V)}\leq 1.$$
\item For every eigenvalues $\mu_1\ne \mu_2$, there exists subsets $A_1,A_2\subseteq V$ such that
$$\abs{\mu_1-\mu_2}=\frac{2\abs{\abs{\partial A_1}-\abs{\partial A_2}}}{\vol(V)}\ge \frac{2}{\vol(V)}.$$
Since all eigenvalues are contained in $[0,1]$, $\Delta_{1,\infty}^-$ have at most $\frac{\vol(V)}{2}+1$ different eigenvalues.
\textcolor{black}{\item and (d) are consequences of Lemma~\ref{lem:eigen-cut} and Theorem \ref{thm:eigen-value}.}
\end{enumerate}
\end{proof}


\begin{proposition}\label{pro:sym}
Suppose $(\mu,\vec x)$ is an eigenpair of  \eqref{eq:1-lap-maxcut}, and $D_0^1,~\cdots,D_0^{r^0}$ are null domains of $\vec x$ on $G=(V,E)$, then
\begin{equation}\label{equ:sym}
\abs{E(D_+,D_0^\omega)}=\abs{E(D_-,D_0^\omega)}, \quad \forall\,1\leq \omega \leq r^0.
\end{equation}
\end{proposition}

\begin{proof}
It suffices to prove \eqref{equ:sym} for $\omega=1$. Since  $(\mu,\vec x)$ is an eigenpair of  \eqref{eq:1-lap-maxcut}, there exists $z_{ij}\in \sgn(x_i-x_j)$ with $z_{ij}=-z_{ji}$ such that
\begin{equation*}
\sum_{j\sim i} z_{ij} = 0,\quad i\in D_0.
\end{equation*}
Summing over the above equation for all $i\in D_0^1$ on both sides leads to
\begin{equation}
\sum_{i\in D_0^1}\sum_{j\sim i} z_{ij} = 0,
\end{equation}
and thus
\begin{equation}
\sum_{i\in D_0^1}\sum_{j\sim i,~j\in D_\pm} z_{ij} = 0,
\end{equation}
provided by (S\ref{statement:S1}). Hence
\begin{equation}
-\abs{E(D_+,D_0^1)}+\abs{E(D_-,D_0^1)}=0.
\end{equation}
\end{proof}


\textcolor{black}{
\begin{theorem}\label{theorem:nodal1}
Reorder all the eigenvalues of  \eqref{eq:1-lap-maxcut}  into
$\bar\mu_0>\bar\mu_1>\cdots >
\bar\mu_{\min}=0$,
then for any eigenpair $(\bar\mu_k, \vec x_k)$ with $k=0,1,\cdots$, we have
$$N_s(\vec x_k)\leq k,$$
where $N_s(\vec x_k)$ denotes the number of non-single point sets in a subgraph $G_0(\vec x_k)$ induced by $D_0(\vec x_k)$.
Especially,  $G_0(\vec x_0)$ consists of single points.
\end{theorem}}

\begin{proof}
Suppose $(\mu,\vec x)$ is an eigenpair of $\Delta_{1}|_{X^{\infty}}$. Let $f:\power(D_0)\rightarrow \mathbb{R}$ be a function such that
$$
f(S)=\frac{2(\abs{\partial (D_+\cup S)}-\abs{\partial D_+})}{\vol(V)}, \;\; \forall\, S\subseteq D_0
$$
At first we shall prove
\begin{equation}\label{eq:lemmaofeigenpari}
f(\bigcup_{i=1}^{r^0} A_i)=\sum_{i=1}^{r^0} f(A_i),
\;\; \text{for} \;\; A_i \subseteq D_0^i, \;\; i=1,2,\cdots,r^0.
\end{equation}
In fact, the fact (S\ref{statement:S1}) implies
$$
E(A_i,D_0^j\setminus A_j)=0 \mbox{ for any }  i\neq j,
$$
then
\begin{equation}\label{eq:E(A,D_0/A)}
\abs{E(A,D_0\setminus A)}=\sum_{i=1}^{r^0}\sum_{j=1}^{r^0}\abs{E(A_i,D_0^j\setminus A_j)}=\sum_{i=1}^{r^0}\abs{E(A_i,D_0^i\setminus A_i)},
\end{equation}
where $A = \bigcup_{i=1}^{r^0} A_i$. Hence, \eqref{eq:lemmaofeigenpari} can be deduced via   \eqref{eq:E(A,D_0/A)} as follows
\begin{align*}
\frac{1}{2}\vol(V)f(A)&=\abs{\partial (D_+\cup A)}-\abs{\partial D_+}\\
&=\abs{E(A,D_0\setminus A)}+\abs{E(A,D_-)}-\abs{E(A,D_+)}\\
&=\sum_{i=1}^{r^0}(\abs{E(A_i,D_0^i\setminus A_i)}+\abs{E(A_i,D_-)}-\abs{E(A_i,D_+)})\\
&=\sum_{i=1}^{r^0}(\abs{E(A_i,D_0\setminus A_i)}+\abs{E(A_i,D_-)}-\abs{E(A_i,D_+)})\\
&=\frac{1}{2}\vol(V)\sum_{i=1}^{r^0} f(A_i).
\end{align*}


By Theorem \ref{theorem:spct2}, $\vec x=\vec 1_{D_+(\vec x_k)}-\vec 1_{D_-(\vec x_k)}$ is also an eigenvector of $\mu_k$ and satisfies $D_0(\vec x_k)=D_0(\vec x)$. So there exists $z_{ij}\in \sgn(x_i-x_j)$ with $z_{ij}=-z_{ji}$ such that
\begin{equation*}
\sum_{j\sim i} z_{ij} = 0,\quad i\in D_0.
\end{equation*}

Let $D_0^1(\vec x),D_0^2(\vec x),\cdots ,D_0^{\textcolor{black}{N_s(\vec x_k)}}(\vec x)$ be non-single point sets in $G_0(\vec x_k)$. We will show that for any $\omega \in\{1,2,\cdots,\textcolor{black}{N_s(\vec x_k)}\}$, there exists $i_\omega \in D_0^\omega(\vec x)$ such that
$$f(\{i_\omega \})>0.$$
Suppose the contrary, that there exists $\omega$ such that
$$ f(\set{i})\leq 0$$
for any $~i\in D_0^\omega$, which means that
\begin{align*}
\frac{1}{2}\vol(V)f(\set{i})&=\abs{E(\set{i},D_0^\omega \setminus \set{i})}+\abs{E(\set{i},D_-)}-\abs{E(\set{i},D_+)}\\
&=\abs{E(\set{i},D_0^\omega \setminus \set{i})}-\sum_{j\sim i,~j\in D_0^\omega}z_{ij}(\vec x).
\end{align*}
Summing over the above equation for all $i\in D_0^\omega$ on both sides yields
\begin{equation*}
0\geq \sum_{i\in D_0^\omega}\frac{1}{2}\vol(V)f(\set{i}) =\sum_{i\in D_0^\omega}\abs{E(\set{i},D_0^\omega/\set{i})}\ge 0,
\end{equation*}
and thus $\abs{E(\set{i},D_0^\omega/\set{i})}=0$ for any $i\in D_0^\omega$.
That is , the volume of subgraph induced by $D_0^\omega$, $\sum_{i\in D_0^\omega}\abs{E(\{i\},D_0^\omega/\{i\})}$, vanishes, which contradicts the fact that non-single point set $D_0^\omega$ is the vertex set of a connected component of the subgraph induced by $D_0(\vec x)$. Hence there exists $i_\omega \in D_0^\omega(\vec x)$ such that $f(\{i_\omega\})>0$ for any $\omega\mbox{ with } 1\leq \omega\leq r^0$.
Accordingly, let $B_0=D_0$, $A_\omega = \{i_\omega\}$ and $B_\omega=\bigcup_{j=1}^\omega A_j\bigcup D_+$ with $\omega=1,2,\cdots,\textcolor{black}{N_s(\vec x_k)}$. By Theorem \ref{theorem:spct2}, there is a series of eigenvectors $\vec 1_{B_\omega,B_\omega^c}$ with eigenvalues $\nu_\omega$, where
$$
\nu_\omega = \frac{2\abs{\partial B_\omega}}{\vol(V)}=\bar\mu_k+f(\bigcup_{j=1}^\omega A_j)=\bar\mu_k+\sum_{j=1}^\omega f( A_j).
$$
That is, there are at least $\textcolor{black}{N_s(\vec x_k)}$ different eigenvalues such that
$$\bar\mu_k=\nu_0<\nu_1<\cdots<\nu_{\textcolor{black}{N_s(\vec x_k)}}.$$
Since $\bar\mu_k$ is the $(k+1)$-th largest eigenvalue,
we can deduce $\textcolor{black}{N_s(\vec x_k)}\leq k$.
\end{proof}

%
%


\begin{proposition}
Given a connect graph $G=(V,E)$, suppose $\mu$ is the minimal non-zero eigenvalue of  \eqref{eq:1-lap-maxcut}  and $\vec x$  the corresponding eigenvector, then the number of null domains
\begin{equation*}
S_0(\vec x)\leq 2.
\end{equation*}
\end{proposition}

\begin{proof}
Suppose the contrary: $S_0(\vec x)=k\geq 3$, and let $D_0^1,~\cdots,D_0^k$ be null domains of $\vec x$. From Proposition \ref{pro:sym}, denote
$$
a_i=\abs{E(D_+,D_0^i)}=\abs{E(D_-,D_0^i)}\geq 1,\;\; i=1,2,\ldots,k,
$$
where the condition that $G$ is connected has been used.
Without loss of generality, assume $1\leq a_1\leq \cdots \leq a_k$,
and then let $\tilde{\vec x}=\vec 1_{D_0^1,(D_0^1)^c}$.
Theorem \ref{theorem:spct2} implies that $\tilde{\vec x}$ is an eigenvector corresponding to an eigenvalue $\tilde{\mu}$. Thus we have
\begin{equation*}
0<\frac{1}{2}\vol(V)\tilde{\mu}=2a_1<\sum_{i=1}^k a_i\leq\frac{1}{2}\vol(V)\mu,
\end{equation*}
i.e., $\tilde{\mu}$ is another non-zero eigenvalue of  \eqref{eq:1-lap-maxcut}  less than $\mu$. This is a contradiction and thus $S_0(x)\leq 2.$
\end{proof}


Below we will show some very special (amazing) structures of the maximal eigenvector.

\begin{proposition}\label{pro:max}
Let $\vec x$ be an eigenvector corresponding to $\mu_{\max}$. Then we have:
\begin{enumerate}[(a)]
\item The collection of vertex sets of connected components of $G_0$
equals to that of $G_1$. Here $G_0=(D_+\bigcup D_-,E(D_+,D_-))$ and $G_1$ is the subgraph induced by $D_+\bigcup D_-$ on $G$. Moreover, $\sum_{i=1}^{S'(\vec x)} a_i \vec 1_{A_i,B_i}$ with $a_i\in\{-1,+1\}$ is a maximal eigenvector, too.
\item $\abs{\partial A_i}=\abs{\partial B_i}$ for $i=1,2,\ldots, S'(\vec x)$.
\item $|E(A_i,\{v\})|=|E(B_i,\{v\})|$ for any $i\in\{1,\ldots,S'(\vec x)\}$ and $v\in D_0$.
\end{enumerate}
\end{proposition}

\begin{proof}
The proof is given in turn below.
\begin{enumerate}[(a)]
\item Let $\tilde{D}_1,\tilde{D}_2,\ldots ,\tilde{D}_k$ be the vertex sets of connected components of $G_0$ where $\tilde{D_i}=\tilde{A_i}\bigcup \tilde{B_i}$ with $\tilde{A_i}\subseteq D_+$, $\tilde{B_i}\subseteq D_-$. It suffices to show that there is no connection between $\tilde{D}_i$ and $\tilde{D}_j$ on $G_1$ whenever $i\ne j$, which is equivalent to prove
\begin{equation}\label{eq:tp2}
\abs{E(\tilde{A}_i,\tilde{A}_j)}=\abs{E(\tilde{B}_i,\tilde{B}_j)}=0.
\end{equation}
In fact, let
$$
T=\set{\sum_{i=1}^k a_i\vec 1_{\tilde{A}_i,\tilde{B}_i}: a_i\in\{-1,+1\}}.
$$
For any $\vec y=\sum_{i=1}^k a_i\vec 1_{\tilde{A}_i,\tilde{B}_i}\in T$,
we have
\begin{align*}
I(\vec y)=\,&2\abs{E(D_+,D_-)}+\abs{E(D_+,D_0)}+\abs{E(D_-,D_0)}\\
+&\frac{1}{2}\sum_{i=1}^k\sum_{ j=1}^k\abs{a_i-a_j}(\abs{E(\tilde{A}_i,\tilde{A}_j)}+\abs{E(\tilde{B}_i,\tilde{B}_j)})\\
=\,&\vol(V)\mu_{max}+\frac{1}{2}\sum_{i=1}^k\sum_{ j=1}^k\abs{a_i-a_j}(\abs{E(\tilde{A}_i,\tilde{A}_j)}+\abs{E(\tilde{B}_i,\tilde{B}_j)})\\
\ge\,&\vol(V)\mu_{max}.
\end{align*}
On the other hand, since $\mu_{\max}$ is the maximal eigenvalue of  \eqref{eq:1-lap-maxcut},  one has
\begin{equation*}
\mu_{\max}\ge \frac{I(\vec y)}{\vol(V) \norm{\vec y}}=\frac{I(\vec y)}{\vol(V)}.
\end{equation*}
Finally, we obtain that $\vec y$ is also a maximal eigenvector, and
$$
\sum_{i=1}^k\sum_{ j=1}^k\abs{a_i-a_j}(\abs{E(\tilde{A}_i,\tilde{A}_j)}+\abs{E(\tilde{B}_i,\tilde{B}_j)})=0, \;\; \forall\, a_i\in\{-1,+1\}.
$$
It follows that
\eqref{eq:tp2} is true.

\item Suppose the contrary, that there exists $i$ such that
$$\abs{\partial A_i}\neq\abs{\partial B_i}.$$
Without loss of generality, we can suppose $\abs{\partial A_i}<\abs{\partial B_i}$ and let $S=(D_+\setminus A_i)\cup B_i$. Then
\begin{equation*}
\abs{\partial S}=\sum_{j\neq i}\abs{\partial A_j}+\abs{\partial B_i}=\abs{\partial D_+}-\abs{\partial A_i}+\abs{\partial B_i}>\abs{\partial D_+},
\end{equation*}
which means
\begin{equation*}
I(\vec 1_{S,S^c})=\frac{2\abs{\partial S}}{\vol(V)}>\frac{2\abs{\partial D_+}}{\vol(V)}=\mu_{\max},
\end{equation*}
where Theorem \ref{theorem:spct2}\ref{tpD} is applied. \textcolor{black}{Up to this point, we have shown there is an eigenvector $\vec 1_{S,S^c}$ whose eigenvalue is larger than $\mu_{\max}$, contradicting the assumption that $\mu_{\max}$ is the maximum eigenvalue of \eqref{eq:1-lap-maxcut}.}

\item We have shown that any $\vec y\in T$ is an maximal eigenvector,
and  the vertices sets of subgraph induced by $D_0(\vec y)$ are all single point sets by Theorem \ref{theorem:nodal1}. Then, Theorem \ref{pro:sym} implies
$$
\abs{E(D_+(\vec y),\set{v})}=\abs{E(D_-(\vec y),\set{v})},\;\; \forall\,v\in D_0(\vec y),
$$
which leads to
$$
\sum_{i=1}^ka_i(\abs{E(A_i,\set{v})}-\abs{E(B_i,\set{v})})=0,\,\,\forall\, a_i\{-1,+1\},
$$
and thus $\abs{E(A_i,\set{v})}=\abs{E(B_i,\set{v})}$.
\end{enumerate}
\end{proof}

\begin{proposition} \label{example:max}
Given $k\in\mathbb{N}^+$, there exists a connect graph $G=(V,E)$ such that
$$S'(\vec x)=k,$$
where $\vec x$ is an eigenvector of   \eqref{eq:1-lap-maxcut}  corresponding to the maximal eigenvalue.
\end{proposition}

\begin{proof}
Let $a_i = i, b_i=k+i,c=2k+1$ and $G=(V,E)$ with
\begin{align*}
V=\Big(\bigcup_{i=1}^k \{a_i\}\Big)\bigcup \Big(\bigcup_{i=1}^k\{b_i\}\Big)\bigcup \{c\}, \;\;
E=\bigcup_{i=1}^k\{(a_i,c),(b_i,c),(a_i,b_i)\}.
\end{align*}
It can be easily verified that
$\vec x=\sum_{i=1}^k \vec 1_{\{a_i\},\{b_i\}}$
is a maximal eigenvector, and $S'(\vec x)=k$ with $D_i = \{a_i\}\cup\{b_i\}$.
\end{proof}

\begin{theorem}
Give a connect graph $G=(V,E)$ with $n=|V|\geq 3$ and an eigenvector $\vec x$ of  \eqref{eq:1-lap-maxcut}  to the maximal eigenvalue, we have
\begin{equation}\label{eq:tp3}
S'(\vec x)\leq \frac{n-1}{2}.
\end{equation}
\end{theorem}

\begin{proof}
When $D_0=\varnothing$, it holds always that $S'(\vec x)=1\leq (n-1)/2$ for $n\geq 3$.
In the following, we may assume $D_0\neq \varnothing$. By Proposition \ref{pro:max}, it holds $D_i=A_i\cup B_i$ is not empty and
\begin{equation*}
\abs{\partial A_i}=\abs{\partial B_i}\geq 1,\;\;i=1,2,\cdots, S'(\vec x),
\end{equation*}
which means $A_i$ and $B_i$ are both non-empty. Consequently, one obtains that
\begin{equation*}
n=\sum_{i=1}^{S'(\vec x)}(\abs{A_i}+\abs{B_i})+\abs{D_0}\geq 2S'(\vec x)+1,
\end{equation*}
and thus \eqref{eq:tp3} holds. Actually, the equality can hold by Proposition \ref{example:max}.
\end{proof}

\textcolor{black}{Finally, we turn to the min-max eigenvalues. The Liusternik-Schnirelmann theory is applied to study the multiplicity of the critical points for even functions on the symmetric piecewise linear manifold $X_{\infty}$ \eqref{set:infty}.  More details about the critical point theory on $X_\infty$ can be referred to \cite{CSZZ18prepare}.}


Let $A\subseteq X_\infty$ be a symmetric set, i.e., $-A=A$. Then
\begin{equation}\label{eq:ck-infty}
{\color{black}\mu_k(\Delta_1|_{X_\infty}):} = \inf_{{\color{black}\gen}(A)\ge k}\max\limits_{\vec x\in A} \frac{I(\vec x)}{{\color{black}\vol(V)}\|\vec x\|_\infty}
\end{equation}
are critical values of $\tilde{I}(\vec x)/{\color{black}\vol(V)}$, $k=1,2,\cdots,n$, where $\tilde{I}=I\big|_{X_\infty}$.
They satisfy
$${\color{black}\mu_1(\Delta_1|_{X_\infty})\leq \mu_2(\Delta_1|_{X_\infty})\leq \cdots \leq \mu_n(\Delta_1|_{X_\infty})}.$$
Moreover, if
$$c=c_{k+1}=\cdots =c_{k+l},~0\leq k\leq k+l\leq n,$$
then ${\color{black}\gen}(\vK \cap \tilde{I}^{-1}(c))\geq l$, where $\vK$ is the critical sets of $\tilde{I}$. A critical value $c$ is said of multiplicity $l$, if ${\color{black}\gen}(\vK \cap \tilde{I}^{-1}(c))=l$.


Finally, we have the Courant-type nodal domain theorem for ${\color{black}\mu_k(\Delta_1|_{X_\infty})}$.
\begin{theorem}\label{thm:Courant}
If $\vec \phi_k$ is an eigenvector corresponding to ${\color{black}\mu_k(\Delta_1|_{X_\infty})}$, then we have
\begin{equation}\label{eq:nodal1}
S(\vec \phi _k)\leq k+r-1,
\end{equation}
and
\begin{equation}\label{eq:nodal2}
S_0(\vec \phi _k)\leq k+r-2,
\end{equation}
where $r$ is the multiplicity of ${\color{black}\mu_k(\Delta_1|_{X_\infty})}$.
\end{theorem}

\begin{proof}
First, we prove \eqref{eq:nodal1}. Suppose the contrary: $S(\vec \phi_k)= r^++r^-\geq k+r$.  Denote $a_{ij} = \abs{E(D_+^i,D_-^j)}$, $d_+^i =\abs{ E(D_+^i,D_0)}$ and $d_-^j=\abs{ E(D_-^j,D_0)}$ for $i=1,2,\cdots,r^+$ and $j=1,2,\cdots,r^-$. Then from Theorem \ref{theorem:spct2},
we have
\begin{equation*}
\sum_{i=1}^{r^+} d_+^i=\sum_{j=1}^{r^-} d_-^j :=d,\;\;
\frac{1}{2}\vol(V){\color{black}\mu_k(\Delta_1|_{X_\infty})}=\sum_{i=1}^{r^+}\sum_{j=1}^{r^-}a_{ij}+d.
\end{equation*}

Let
$$
Y=\Span\{\vec y_+^1,\cdots,\vec y_+^{r^+},\vec y_-^1,\cdots,\vec y_-^{r^-}\},\;\;
\vec y_+^i=\vec 1_{D_+^i}, \;\; \vec y_-^j=\vec 1_{D_-^j},
$$
and then the dimension of $Y$ is $r^++r^-\geq k+r$.

We will show for any
$\vec x=\sum_{i=1}^{r^+} t_+^i \vec y_+^i+\sum_{j=1}^{r^-} t_-^j \vec y_+^j \in Y\setminus\{\vec 0\}$,
it holds that
$$\frac{I(\vec x)}{\vol(V)\norm{\vec x}}\leq {\color{black}\mu_k(\Delta_1|_{X_\infty})}.$$

In fact,  we have
\textcolor{black}{
\begin{align*}
I(\vec x)&=\sum_{i=1}^{r^+}\sum_{j=1}^{r^-} a_{ij} \abs{t_+^i-t_-^j}+\sum_{i=1}^{r_+} d_+^i\abs{t_+^i}+\sum_{j=1}^{r_-} d_-^j\abs{t_-^j}\\
&\le \sum_{i=1}^{r^+}\sum_{j=1}^{r^-} 2a_{ij} \max \{\abs{t_+^i},\abs{t_-^j}\}+2d\norm{\vec x}\\
&\le 2(\sum_{i=1}^{r^+}\sum_{j=1}^{r^-}a_{ij}+d)\norm{\vec x},
\end{align*}}
which induces
\begin{equation*}
\frac{I(\vec x)}{\vol(V)\norm{\vec x}} \leq \frac{2(\sum_{i=1}^{r^+}\sum_{j=1}^{r^-}a_{ij}+d)}{\vol(V)}={\color{black}\mu_k(\Delta_1|_{X_\infty})}.
\end{equation*}

Let $T=Y\cap X_\infty$, then $T=-T$,~and ${\color{black}\gen}(T)\geq k+r$. And we have
\begin{equation*}
\tilde{I}(\vec x)/{\color{black}\vol(V)}=I(\vec x)/{\color{black}\vol(V)}\Big|_{ X_\infty}\leq {\color{black}\mu_k(\Delta_1|_{X_\infty})},\;\; \forall\, \vec x\in T,
\end{equation*}
implying that the multiplicity of ${\color{black}\mu_k(\Delta_1|_{X_\infty})}$ is not less than $r+1$,
which contradicts the condition that  the multiplicity of ${\color{black}\mu_k(\Delta_1|_{X_\infty})}$ is $r$.
Therefore, $S(\vec \phi _k)\leq k+r-1$.

Second, we will prove \eqref{eq:nodal2}. Suppose the contrary: $r^0 = S_0(\vec \phi _k)\geq k+r-1$. By Proposition \ref{pro:sym}, denote
$$
a_i=\abs{E(D_+,D_0^i)}=\abs{E(D_-,D_0^i)},\;\; i=1,2,\ldots,r^0.
$$
Note that
\begin{equation*}
\frac{1}{2}\vol(V){\color{black}\mu_k(\Delta_1|_{X_\infty})} = \sum_{i=1}^{r^0} a_i+d,
\end{equation*}
Let $Y=\Span\{\vec y_0,\vec y_1,\cdots,\vec y_{r^0}\}$, where $\vec y_0=\vec1_{D_+,D_-}$ and $ \vec y_i= \vec1_{D_0^i}$ for $i=1,2,\ldots,r^0$. The dimension of $Y$ is $r^0+1$. Next, we will prove that for any $\vec x=\sum_{i=0}^{r^0} t_i \vec y_i \in Y\setminus \{\vec 0\}$,
we have
$$\frac{I(\vec x)}{\vol(V)\norm{\vec x}}\leq {\color{black}\mu_k(\Delta_1|_{X_\infty})}.$$
In fact,
\begin{align*}
I(\vec x)&=\sum_{i=1}^{r^0} (a_i\abs{t_i-t_0}+a_i\abs{t_i+t_0})+2d\abs{t_0}\\
&= \sum_{i=1}^{r^0} 2a_i\max\{\abs{t_i},\abs{t_0}\}+2d\abs{t_0}\\
&\le (\sum_{i=1}^{r^0} 2a_i+2d)\norm{\vec x},
\end{align*}
which means
\begin{equation*}
\frac{I(\vec x)}{\vol(V)\norm{\vec x}}\leq \frac{\sum_{i=1}^{r^0} 2a_i+2d}{\vol(V)}={\color{black}\mu_k(\Delta_1|_{X_\infty})}.
\end{equation*}
Let $T=Y\cap X_\infty$, then $T=-T$,~and ${\color{black}\gen}(T)=r^0+1$, and thus
\begin{equation*}
\tilde{\vI}(\vec x)\leq \vol(V){\color{black}\mu_k(\Delta_1|_{X_\infty})},\;\;\forall\, \vec x\in T,
\end{equation*}
implying that the multiplicity of ${\color{black}\mu_k(\Delta_1|_{X_\infty})}$ is not less than $r+1$,
which contradicts the condition that  the multiplicity of ${\color{black}\mu_k(\Delta_1|_{X_\infty})}$ is $r$.
Therefore, $S_0(\vec \phi _k)\leq k+r-2$.

\end{proof}

\vskip 0.5cm

\section{An equivalent eigenproblem for anti-Cheeger cut}\label{sec:anti-Cheeger}

The main purpose of this section is to explore the spectral properties of the eigenproblem \eqref{eq:anti-spec} for  sum-of-norms  derived from equivalent  continuous formulation of the anti-Cheeger cut:
\begin{align} 
h_{\anti}(G) &=\max_{\text{nonzero~}\vec x\in\mathbb{R}^n}\frac{I(\vec x)}{2\vol(V)\|\vec x\|_\infty - N(x)}. \label{eq:antiCheeger-I(x)}
\end{align}
The eigenvalue problem \eqref{eq:anti-spec} in componentwise form reads: $(\mu,\vec x)$ is an eigenpair iff   $\exists\,z_{ij}\in \sgn(x_i-x_j) \mbox{ satisfying }z_{ij}=-z_{ji}$ and $\exists\,v_i\in d_i\sgn(x_i-c_x)$, $c_x\in\median(\vec x)$~such that
 \begin{equation}
 	\label{eigen-anticheeger}
 	\left\{
 	 \begin{aligned}
 	 	&\sum_{i\in V} v_i=0,\\
 	&\sum\limits_{j:\{i,j\}\in E} z_{ij} +\mu v_i= 0, &i\in D_0,\\
 		&\sum\limits_{j:\{i,j\}\in E} z_{ij} +\mu v_i\in 2\mu \vol(V)\sign(x_i)\cdot [0,1],& i\in D_\pm,\\
 	&\sum_{i\in V} \big|\sum\limits_{j:\{i,j\}\in E} z_{ij}+\mu v_i \big|=2\mu\vol(V).&
 	\end{aligned}
 	\right.
 \end{equation}
Similar to Lemma \ref{lemma:preceq}, we have
\textcolor{black}{
\begin{lemma}\label{lemma:preceq2}
If $(\mu,\vec x)$ is an eigenpair of \eqref{eigen-anticheeger}, and for any $\vec y$ satisfying $\vec x \preceq \vec y$, for any $ t\in[0,1]$, $(\mu,t\vec x+(1-t)\vec y)$ is also an eigenpair of  \eqref{eigen-anticheeger}. Moreover, there exists a binary vector $\hat{\vec x}$ such that $(\mu,\hat{\vec x})$ is an eigenpair.
\end{lemma}}

%
%
%

 \begin{theorem}\label{theorem:ach2}
Given a graph $G=(V,E)$, the following claims hold for the eigenproblem \eqref{eigen-anticheeger}:
 	\begin{enumerate}[(a)]
 		\item For any eigenvalue $\mu$,  $0\leq \mu\leq 1$.
 		\item For any two distinct eigenvalues, their distance is at least $\frac{2}{(\vol(V))^2}$. In consequence, there are at most $\frac{(\vol(V))^2}{2}+1$ different eigenvalues.
 		\item\label{th:c-ach} For any set $A\subseteq V$, taking $\vec x=\vec 1_{A,A^c}$, we have
 		$\mu=\frac{\abs{\partial A}}{\max\{\vol(A),\vol(A^c)\}}$, and $(\mu,\vec x)$ is an eigenpair.
 		\item\label{tpD-ach} If $(\mu,\vec x)$ is an eigenpair, then  $\frac{I(\vec x)}{2\vol(V)\norm{\vec x}-N(\vec x)}=\mu$.
 		Moreover, $(\mu,\vec 1_{D_+,D_-})$ and $(\mu,\vec 1_{D_+,D_+^c})$ are both eigenpairs.
 	\end{enumerate}
 \end{theorem}

  \begin{proof} ~
 	\begin{enumerate}[(a)]
 		\item By Lemma~\ref{lemma:preceq2}, for any eigenvalue $\mu$,  we can find a subset $A\subseteq V$ such that
 		$$0\leq \mu=\frac{\abs{\partial A}}{\max\{\vol(A),\vol(A^c)\}}\leq 1.$$
 		\item For any two eigenvalues $\mu_1\ne \mu_2$, there exist $A_1,A_2\subseteq V$ such that
 		$$\abs{\mu_1-\mu_2}=\big|\frac{\abs{\partial A_1}}{\max\{\vol(A_1),\vol(A_1^c)\}}-\frac{\abs{\partial A_2}}{\max\{\vol(A_2),\vol(A_2^c)\}}\big|.$$
 		Then for any $S\subseteq V$, we have:
 		$$
 		\begin{aligned}
 			\vol(S)&=2\abs{E(S,S)}+\abs{\partial S},\\
 			\vol(S^c)&=2\abs{E(S^c,S^c)}+\abs{\partial S},
 		\end{aligned}
 		$$
 		Hence.  $\abs{\partial S}$,  $\vol(S)$ and $\vol(S^c)$ are all odd or all even. In consequence, 
 		$$\begin{aligned}
 			&\abs{\abs{\partial A_1}\max\{\vol(A_2),\vol(A_2^c)\}-\abs{\partial A_2}\max\{\vol(A_1),\vol(A_1^c)\}}\geq 2,\\
 		   	&\max\{\vol(S),\vol(S^c)\}\leq \vol(V).
 		\end{aligned}$$
 		Since all eigenvalues lie in the range $[0,1]$, the eigenproblem has at most $\frac{(\vol(V))^2}{2}+1$ different eigenvalues. 
 		\textcolor{black}{\item and (d) are consequences of Lemma~\ref{lem:eigen-cut} and Theorem \ref{thm:eigen-value}.}
 	\end{enumerate}
 \end{proof}
 
  \textcolor{black}{We adopt the definition of {\color{black}extremal} domains from Section~\ref{sec:maxcut} for this section, comprising the positive {\color{black}extremal} domains $D_+^1,~\cdots,D_+^{r^+}$, the negative {\color{black}extremal} domains $D_-^1,~\cdots,D_-^{r^-}$, and the null domains $D_0^1,~\cdots,D_0^{r^0}$ (see \eqref{eq:maxcut-nodal}).}
  \begin{proposition}\label{pro:sym-ach}
 	Suppose $(\mu,\vec x)$ is an eigenpair, and  $c_x\in\median(\vec x)$. Let $p_i\in [0,1]$  $\forall\,i\in D_{\pm}$, be such that the 3rd equality in \eqref{eigen-anticheeger} holds, that is,
 	\begin{equation}
 		\begin{aligned}
 			\sum\limits_{j:\{i,j\}\in E} z_{ij} +\mu v_i&= 2\mu \vol(V)p_i,& i\in D_+,\\
 			\sum\limits_{j:\{i,j\}\in E} z_{ij} +\mu v_i&= -2\mu \vol(V)p_i,& i\in D_-.
 		\end{aligned}
 	\end{equation}
 	Then, we have
 	\begin{equation}
 		\label{eq:inf-equal-p}
 		\sum\limits_{i\in D_+} p_i=\sum\limits_{i\in D_-} p_i=\frac{1}{2}.
 	\end{equation}
  \textcolor{black}{Moreover, we have the following properties:}
 	\begin{enumerate}[(a)]
 		\item If $c_x=\norm{\vec x}$, then 
 		\begin{align}
 			\mu&=\frac{-\abs{E(D_0^\omega,D_+)}+\abs{E(D_0^\omega,D_-)}}{\vol(D_0^\omega)},\quad &\forall\,\omega\in\{1,2,\ldots,r^0\},\label{eq:sym1-ach-1}\\
 			\mu&=\frac{\abs{\partial D_-^\beta}}{2\vol(V)\sum\limits_{i\in D_-^\beta}p_i-\vol(D_-^\beta)},\quad &\forall\,\beta\in\{1,2,\ldots,r^-\}.\label{eq:sym1-ach-2}
 		\end{align}
 		\item If $c_x=-\norm{\vec x}$, then 
 		\begin{align}
 			\mu&=\frac{\abs{E(D_0^\omega,D_+)}-\abs{E(D_0^\omega,D_-)}}{\vol(D_0^\omega)},\quad &\forall\,\omega\in\{1,2,\ldots,r^0\},\label{eq:sym2-ach-1}\\
 			\mu&=\frac{\abs{\partial D_+^\alpha}}{2\vol(V)\sum\limits_{i\in D_+^\alpha}p_i-\vol(D_+^\alpha)},\quad &\forall\,\alpha\in\{1,2,\ldots,r^+\}.\label{eq:sym2-ach-2}
 		\end{align} 		
 		\item If $|c_x|<\norm{\vec x}$, then \eqref{eq:sym2-ach-2} and \eqref{eq:sym1-ach-2} hold.
 	\end{enumerate}
 \end{proposition}

  \begin{proof}
 	\begin{enumerate}[(1)]
 		\item First let's prove (a). In the case of $c_x=\norm{\vec x}$, for any $i\in D_0,\,D_-$, we have $v_i=-d_i$. Then by the 2nd equality in \eqref{eigen-anticheeger}, for ant $\omega\in \{1,2,\ldots,r^0\}$, we have:
 		\begin{equation}
 			\begin{aligned}
 				\sum\limits_{i\in D_0^\omega}	\sum\limits_{j:\{i,j\}\in E} z_{ij} -\mu\sum\limits_{i\in D_0^\omega} d_i&= 0\\
 				\Leftrightarrow -\abs{E(D_0^\omega,D_+)}+\abs{E(D_0^\omega,D_-)}&=\mu\vol(D_0^\omega),
 			\end{aligned}
 		\end{equation}
 		which derives \eqref{eq:sym1-ach-1}. Here we used the property that each \textcolor{black}{null domain} is a connected component of  $D_0$. Considering the 3rd equality in \eqref{eigen-anticheeger}, for any $\beta\in \{1,2,\ldots,r^-\}$, the following equalitities hold:
 		\begin{equation}
 			\label{eq:sym-ach-proof-1-2}
 			\begin{aligned}
 				\sum\limits_{i\in D_-^\beta}	\sum\limits_{j:\{i,j\}\in E} z_{ij} -\mu\sum\limits_{i\in D_-^\beta} d_i&= -2\mu\vol(V)\sum\limits_{i\in D_-^\beta} p_i,\quad p_i\in[0,1],\\
 				\Leftrightarrow -\abs{E(D_-^\beta,D_+\cup D_0)}&=\mu\vol(D_-^\beta)-2\mu\vol(V)\sum\limits_{i\in D_-^\beta} p_i,
 			\end{aligned}
 		\end{equation}
 		and then we deduce \eqref{eq:sym1-ach-2}. Again, here we used the fact that  each negative {\color{black}extremal} domain is a connected component of  $D_-$.
 		\item Now we prove (b). In the case $c_x=-\norm{\vec x}$, for any $i\in D_0,\,D_+$, we have $v_i=d_i$, and then by the second equality in  \eqref{eigen-anticheeger}, for any $\omega\in \{1,2,\ldots,r^0\}$, we obtain
 		\begin{equation}
 			\begin{aligned}
 				\sum\limits_{i\in D_0^\omega}	\sum\limits_{j:\{i,j\}\in E} z_{ij} +\mu\sum\limits_{i\in D_0^\omega} d_i&= 0\\
 				\Leftrightarrow -\abs{E(D_0^\omega,D_+)}+\abs{E(D_0^\omega,D_-)}&=-\mu\vol(D_0^\omega).
 			\end{aligned}
 		\end{equation}
 		This yields \eqref{eq:sym2-ach-1}, where we have used the property of \textcolor{black}{null domain} $D_0$. Considering the third equality in \eqref{eigen-anticheeger}, for any $\alpha\in \{1,2,\ldots,r^+\}$, we have
 		\begin{equation}
 			\label{eq:sym-ach-proof-2-2}
 			\begin{aligned}
 				\sum\limits_{i\in D_+^\alpha}	\sum\limits_{j:\{i,j\}\in E} z_{ij} +\mu\sum\limits_{i\in D_+^\alpha} d_i&= 2\mu\vol(V)\sum\limits_{i\in D_+^\alpha} p_i,\quad p_i\in[0,1],\\
 				\Leftrightarrow \abs{E(D_+^\alpha,D_-\cup D_0)}&=-\mu\vol(D_+^\alpha)+2\mu\vol(V)\sum\limits_{i\in D_+^\alpha} p_i,
 			\end{aligned}
 		\end{equation}
 		which implies \eqref{eq:sym2-ach-2}.
 		\item We shall prove (c). In the case of $|c_x|<\norm{\vec x}$, for any $i\in D_+$, we have $v_i=d_i$; and for any $i\in D_-$, we have $v_i=-d_i$. Thus, by the third equality in \eqref{eigen-anticheeger}, we have \eqref{eq:sym-ach-proof-2-2} and \eqref{eq:sym-ach-proof-1-2}, which simply yield \eqref{eq:sym2-ach-2} and \eqref{eq:sym1-ach-2}.
 		\item Finally, we shall verify \eqref{eq:inf-equal-p}. Similar to the previous discussion, we complete the proof case by case.
 		
 		 \textcolor{black}{
 		 1. In the case of $c_x=\norm{\vec x}$, $\vol(D_+)\geq \frac{1}{2}\vol(V)$ and
 		 \begin{equation}
 		 	\label{ach:mu-binary-1}
 		 	\mu=\frac{\left|\partial D_{+}\right|}{\operatorname{vol}\left(D_{+}\right)}=\frac{\left|\partial D_{-}\right|}{\operatorname{vol}(V)-\operatorname{vol}\left(D_{-}\right)}
 		 \end{equation}
 		hold. By \eqref{eq:sym1-ach-2}, we get
 		 \begin{equation}
 		 	\label{eq:ach-proof-inf-beta}
 		 	\mu=\frac{\sum\limits_{\beta=1}^{r^-}\abs{\partial D_-^\beta}}{\sum\limits_{\beta=1}^{r^-}(2\vol(V)\sum\limits_{i\in D_-^\beta}p_i-\vol(D_-^\beta))}=\frac{\abs{\partial D_-}}{2\vol(V)\sum\limits_{i\in D_-}p_i-\vol(D_-)}.
 		 \end{equation}
 		 Combining with \eqref{ach:mu-binary-1}, we have  $\sum\limits_{i\in D_-}p_i=\frac{1}{2}$. The equality \eqref{eq:inf-equal-p} then follows from $\sum\limits_{i\in D_\pm}p_i=1$.}
 		 
 		 \textcolor{black}{
 		 2. In the case of $c_x=-\norm{\vec x}$, $\vol(D_-)\geq \frac{1}{2}\vol(V)$ and
 		 \begin{equation}
 		 	\label{ach:mu-binary-2}
 		 	\mu=\frac{\left|\partial D_{-}\right|}{\operatorname{vol}\left(D_{-}\right)}=\frac{\left|\partial D_{+}\right|}{\operatorname{vol}(V)-\operatorname{vol}\left(D_{+}\right)}
 		 \end{equation}
 		hold. By \eqref{eq:sym2-ach-2}, we have
 		\begin{equation}
 			\label{eq:ach-proof-inf-alpha}
 			\mu=\frac{\sum\limits_{\alpha=1}^{r^+}\abs{\partial D_+^\alpha}}{\sum\limits_{\alpha=1}^{r^+}(2\vol(V)\sum\limits_{i\in D_+^\alpha}p_i-\vol(D_+^\alpha))}=\frac{\abs{\partial D_+}}{2\vol(V)\sum\limits_{i\in D_+}p_i-\vol(D_+)}.
 		\end{equation}
 		 Combining with \eqref{ach:mu-binary-2}, we obtain  $\sum\limits_{i\in D_+}p_i=\frac{1}{2}$. Again, by $\sum\limits_{i\in D_\pm}p_i=1$, we have \eqref{eq:inf-equal-p}.}
 		 
 		 \textcolor{black}{
 		 3. In the case of $|c_x|<\norm{\vec x}$, we have $\vol(D_{\pm})=\frac{1}{2}\vol(V)$ and
 		 \begin{equation}
 		 	\label{ach:mu-binary-3=}
 		 	\mu=\frac{\left|\partial D_{-}\right|}{\operatorname{vol}(V)-\operatorname{vol}\left(D_{-}\right)}=\frac{\left|\partial D_{+}\right|}{\operatorname{vol}(V)-\operatorname{vol}\left(D_{+}\right)},
 		 \end{equation}
 		 thus both \eqref{eq:ach-proof-inf-alpha} and \eqref{eq:ach-proof-inf-beta} hold. Together with \eqref{ach:mu-binary-3=}, it is easy to get \eqref{eq:inf-equal-p}.}
 	\end{enumerate}
 \end{proof}

 \begin{proposition}
 	Suppose $(\mu,\vec x)$ is an eigenpair, and $c_x\in\median(\vec x)$. \textcolor{black}{For the null domains}  $D_0^1,~\cdots,D_0^{r^0}$, if $|c_x|=\norm{\vec x}$, $(\mu,\vec 1_{S_+^\iota,(S_+^\iota)^c})$ and  $(\mu,\vec 1_{S_-^\iota,(S_-^\iota)^c})$ are also eigenpair, where $\iota\subseteq\{1,2,\ldots,r^0\}$,  $S_+^\iota:=\bigcup\limits_{\omega\in\iota} D_0^\omega\bigcup D_+$ and $S_-^\iota:=\bigcup\limits_{\omega\in\iota} D_0^\omega\bigcup D_-$.
 \end{proposition}

\begin{proof}
 	We divide the proof into two cases.
 	\begin{enumerate}[(1)]
 		\item In the case of $c_x=\norm{\vec x}$, by \eqref{eq:sym1-ach-1}, we have:
 		\begin{equation}
 			\mu=\frac{\sum\limits_{\omega\in\iota}(-\abs{E(D_0^\omega,D_+)}+\abs{E(D_0^\omega,D_-)})}{\sum\limits_{\omega\in\iota}\vol(D_0^\omega)}
 		\end{equation}
 		Then it follows from  \eqref{ach:mu-binary-1} that
 		\begin{equation*}
 			\begin{aligned}
 					\mu&=\frac{|E(D_+,D_-)|+(|E(D_+,D_0)|-\sum\limits_{\omega\in\iota}\abs{E(D_0^\omega,D_+)})+\sum\limits_{\omega\in\iota}\abs{E(D_0^\omega,D_-)}}{\vol(D_+)+\sum\limits_{\omega\in\iota}\vol(D_0^\omega)}=\frac{\abs{\partial S_+^\iota}}{\vol(S_+^\iota)},\\
 						\mu&=\frac{|E(D_-,D_+)|+(|E(D_-,D_0)|-\sum\limits_{\omega\in\iota}\abs{E(D_0^\omega,D_-)})+\sum\limits_{\omega\in\iota}\abs{E(D_0^\omega,D_+)}}{\vol(V)-\vol(D_-)-\sum\limits_{\omega\in\iota}\vol(D_0^\omega)}=\frac{\abs{\partial S_-^\iota}}{\vol((S_-^\iota)^c)},
 			\end{aligned}
 		\end{equation*}
 		By  $\vol(S_+^\iota)\geq\frac{1}{2}\vol(V)$ and $\vol((S_-^\iota)^c)\geq \frac{1}{2}\vol(V)$,  $(\mu,\vec 1_{S_+^\iota,(S_+^\iota)^c})$ and $(\mu,\vec 1_{S_-^\iota,(S_-^\iota)^c})$ are also eigenpairs.
 		
 			\item In the case of $c_x=-\norm{\vec x}$, by \eqref{eq:sym2-ach-1}, the following equality holds:
 		\begin{equation}
 			\mu=\frac{\sum\limits_{\omega\in\iota}(\abs{E(D_0^\omega,D_+)}-\abs{E(D_0^\omega,D_-)})}{\sum\limits_{\omega\in\iota}\vol(D_0^\omega)}
 		\end{equation}
 		Again by \eqref{ach:mu-binary-2}, we get
 		\begin{equation*}
 			\begin{aligned}
 				\mu&=\frac{|E(D_+,D_-)|+(|E(D_+,D_0)|-\sum\limits_{\omega\in\iota}\abs{E(D_0^\omega,D_+)})+\sum\limits_{\omega\in\iota}\abs{E(D_0^\omega,D_-)}}{\vol(V)-\vol(D_+)-\sum\limits_{\omega\in\iota}\vol(D_0^\omega)}=\frac{\abs{\partial S_+^\iota}}{\vol((S_+^\iota)^c)},\\
 				\mu&=\frac{|E(D_-,D_+)|+(|E(D_-,D_0)|-\sum\limits_{\omega\in\iota}\abs{E(D_0^\omega,D_-)})+\sum\limits_{\omega\in\iota}\abs{E(D_0^\omega,D_+)}}{\vol(D_-)+\sum\limits_{\omega\in\iota}\vol(D_0^\omega)}=\frac{\abs{\partial S_-^\iota}}{\vol(S_-^\iota)},
 			\end{aligned}
 		\end{equation*}
 		By  $\vol((S_+^\iota)^c)\geq\frac{1}{2}\vol(V)$ and $\vol(S_-^\iota)\geq \frac{1}{2}\vol(V)$,  $(\mu,\vec 1_{S_+^\iota,(S_+^\iota)^c})$ and $(\mu,\vec 1_{S_-^\iota,(S_-^\iota)^c})$ are also eigenpairs.
 	\end{enumerate}
 \end{proof}

 \begin{theorem}
 	Given a connected graph $G=(V,E)$, suppose that $\mu$ is the smallest nonzero eigenvalue, and $\vec x$  is the corresponding eigenvector, and  $c_x\in\median(\vec x)$. Then we have the following estimates for nodal count:
 	\begin{enumerate}[(a)]
 		\item If $c_x=\norm{\vec x}$, then $r^0\leq 2$ and $r^-\leq 1$.
 		\item If $c_x=-\norm{\vec x}$, then $r^0\leq 2$ and $r^+\leq 1$.
 		\item If $|c_x|<\norm{\vec x}$, then $r^+\leq 1$ and $r^-\leq 1$.
 	\end{enumerate}
 \end{theorem}

\begin{proof}
 	\begin{enumerate}[(1)]
 		\item First we shall prove (a). In the case of  $c_x=\norm{\vec x}$, we have  $\vol(D_+)\geq \frac{1}{2}\vol(V)$ and \eqref{ach:mu-binary-1}, and thus
 		$$\mu=\frac{\abs{\partial D_+}+\abs{\partial D_-}}{\vol(D_+)+\vol(V)-\vol(D_-)}=\frac{\abs{\partial D_0}+2\abs{E(D_+,D_-)}}{\vol(D_0)+2\vol(D_+)}.$$
 		By (c) in Theorem \ref{theorem:ach2}, for eigenpair $(\mu_{0}^\omega,\vec 1_{D_0^\omega,(D_0^\omega)^c})$, $\forall\,\omega\in\{1,2,\ldots,r^0\}$, we derive 
 		\begin{equation}
 			\label{eq:ach-sum-gamma}
 			0<\min_\omega \mu_{0}^\omega\leq \frac{\sum\limits_{\omega=1}^{r^0}\abs{\partial D_0^\omega}}{\sum\limits_{\omega=1}^{r^0}(\vol(V)-\vol(D_0^\omega))}=\frac{\abs{\partial D_0}}{r^0\vol(V)-\vol(D_0)}.
 		\end{equation}
 		Suppose the contrary, that $r^0\geq 3$. Then by 
 		$$r^0\vol(V)-\vol(D_0)\geq 3\vol(V)-\vol(D_0)>\vol(D_0)+2\vol(D_+),$$
 		we have $\min_\omega \mu_{0}^\omega<\mu$, a contradiction. Now we check $r^-\leq 1$. 
 		
 		Considering the eigenpair  $(\mu_{-}^\beta,\vec 1_{D_-^\beta,(D_-^\beta)^c})$, $\forall\,\beta\in\{1,2,\ldots,r^-\}$, we derive 
 		\begin{equation}
 			\label{eq:ach-sum-beta}
 			0<\min_\beta \mu_{-}^\beta\leq \frac{\sum\limits_{\beta=1}^{r^-}\abs{\partial D_-^\beta}}{\sum\limits_{\beta=1}^{r^-}(\vol(V)-\vol(D_-^\beta))}=\frac{\abs{\partial D_-}}{r^-\vol(V)-\vol(D_-)}.
 		\end{equation}
 		If $r^-\geq 2$, then 
 		$$\min_\beta \mu_{-}^\beta\leq \frac{\abs{\partial D_-}}{2\vol(V)-\vol(D_-)} <\mu=\frac{\abs{\partial D_-}}{\vol(V)-\vol(D_-)},$$
 		a contradiction.
 		\item The proof of (b). If $c_x=-\norm{\vec x}$, then $\vol(D_-)\geq \frac{1}{2}\vol(V)$ and \eqref{ach:mu-binary-2} hold. Accordingly, 
 		$$\mu=\frac{\abs{\partial D_-}+\abs{\partial D_+}}{\vol(D_-)+\vol(V)-\vol(D_+)}=\frac{\abs{\partial D_0}+2\abs{E(D_+,D_-)}}{\vol(D_0)+2\vol(D_-)}.$$
 		Consider the eigenpairs $(\mu_{0}^\omega,\vec 1_{D_0^\omega,(D_0^\omega)^c})$, $\forall\,\omega\in\{1,2,\ldots,r^0\}$, which satisfy \eqref{eq:ach-sum-gamma}. We can similarly obtain  $r^0\leq 2$.
        Now we show $r^+\leq 1$. 
 		From the eigenpairs  $(\mu_{+}^\alpha,\vec 1_{D_+^\alpha,(D_+^\alpha)^c})$, $\forall\,\alpha\in\{1,2,\ldots,r^+\}$, we derive
 		\begin{equation}
 			\label{eq:ach-sum-alpha}
 			0<\min_\alpha \mu_{+}^\alpha\leq \frac{\sum\limits_{\alpha=1}^{r^+}\abs{\partial D_+^\alpha}}{\sum\limits_{\alpha=1}^{r^+}(\vol(V)-\vol(D_+^\alpha))}=\frac{\abs{\partial D_+}}{r^+\vol(V)-\vol(D_+)}.
 		\end{equation}
 		Suppose the contrary that  $r^+\geq 2$. Then 
 		$$\min_\alpha \mu_{+}^\alpha\leq \frac{\abs{\partial D_+}}{2\vol(V)-\vol(D_+)} <\mu=\frac{\abs{\partial D_+}}{\vol(V)-\vol(D_+)},$$
 		which leads to a contradiction.
 		\item The proof of (c): Since $|c_x|<\norm{\vec x}$, we have $\max\{\vol(D_+),\vol(D_-)\}\leq \frac{1}{2}\vol(V)$ and \eqref{ach:mu-binary-3=}, which further imply
         \eqref{eq:ach-sum-alpha} and \eqref{eq:ach-sum-beta}, so that we can similarly verify  $r^+\leq 1$ and $r^-\leq 1$. 
 	\end{enumerate}
 \end{proof}

  \begin{theorem}
  	\label{thm:ach-single}
 	Given a graph $G=(V,E)$ without isolated vertices, suppose $\mu$ is the maximum eigenvalue, and $\vec x$  is a corresponding eigenvector. Let $G_0=G[D_0]$ be the subgraph of $G$ induced by $D_0$. Then $G_0$ has no edge.
 \end{theorem}

 \begin{proof}
 Let $c_x\in\median(\vec x)$. We separate the proof into three cases.  Before going to the details, we use  
 	\begin{align}
 		d_i^0=|E(\{i\},D_0)|,\label{vetex:i-zero}
 	\end{align}
 	to denote the zero edges incident to $i$, \textcolor{black}{and use the notations $d_i^+=|E(\{i\},D_+)|$ and $d_i^-=|E(\{i\},D_-)|$.}
 	They  satisfy $d_i=d_i^++d_i^-+d_i^0$,  for any $i\in V$.
 	\begin{enumerate}[(1)]
 		\item In the case of $c_x=\norm{\vec x}$, for any $i\in D_0$, by the second equality in \eqref{eigen-anticheeger},
 		\begin{equation}
 			\label{eq:ach-mumax-1}
 			-d_i^0\leq \sum_{j\in D_0,\{i,j\}\in E} z_{ij}=d_i^+-d_i^-+\mu d_i\leq d_i^0.
 		\end{equation}
 		Consider the eigenpair $(\mu^\prime,\vec 1_{S_i,(S_i)^c})$, where $S_i=D_+\bigcup \{i\}$. Then
 		\begin{equation*}
 			\mu^\prime=\frac{\abs{\partial S_i}}{\vol(S_i)}=\frac{\abs{\partial D_+}-d_i^++d_i^-+d_i^0}{\vol(D_+)+d_i}\leq\mu= \frac{\abs{\partial D_+}}{\vol(D_+)},
 		\end{equation*}
 		and hence we have
 		\begin{equation*}
 			\frac{-d_i^++d_i^-+d_i^0}{d_i}\leq\mu\Rightarrow -d_i^++d_i^-+d_i^0\leq\mu d_i.
 		\end{equation*}
 		By the inequality on the right hand side of \eqref{eq:ach-mumax-1}, we get:
 		\begin{equation*}
 			-d_i^++d_i^-+d_i^0=\mu d_i\Rightarrow z_{ij}=1,\,\forall\,j\in D_0,\{i,j\}\in E.
 		\end{equation*}
 		Suppose the contrary that $G_0$ has edges, i.e., there exist $i,j\in D_0,\{i,j\}\in E$ such that $z_{ij}=z_{ji}=1$, which contradicts to $z_{ij}=-z_{ji}$.
 		
 	    \item In the case of $c_x=-\norm{\vec x}$, for any $i\in D_0$, by the 2nd equality in \eqref{eigen-anticheeger}, 
 		\begin{equation}
 			\label{eq:ach-mumax-2}
 			-d_i^0\leq \sum_{j\in D_0,\{i,j\}\in E} z_{ij}=d_i^+-d_i^--\mu d_i\leq d_i^0.
 		\end{equation}
 		Considering the eigenpair $(\mu^\prime,\vec 1_{S_i,(S_i)^c})$, with $S_i=D_-\bigcup \{i\}$, then
 		\begin{equation*}
 			\mu^\prime=\frac{\abs{\partial S_i}}{\vol(S_i)}=\frac{\abs{\partial D_-}-d_i^-+d_i^++d_i^0}{\vol(D_-)+d_i}\leq\mu= \frac{\abs{\partial D_-}}{\vol(D_-)},
 		\end{equation*}
 		and thus 
 		\begin{equation*}
 			\frac{-d_i^-+d_i^++d_i^0}{d_i}\leq\mu\Rightarrow -d_i^-+d_i^++d_i^0\leq\mu d_i.
 		\end{equation*}
 		By the inequality on the left hand side of \eqref{eq:ach-mumax-2}, we have
 		\begin{equation*}
 			-d_i^-+d_i^++d_i^0=\mu d_i\Rightarrow z_{ij}=-1,\,\forall\,j\in D_0,\{i,j\}\in E.
 		\end{equation*}
This means that if $G_0$ has edges, then there exist $i,j\in D_0,\{i,j\}\in E$ such that $z_{ij}=z_{ji}=-1$, which contradicts to $z_{ij}=-z_{ji}$.
 		\item In the case of $|c_x|<\norm{\vec x}$, we first prove $i\in D_0$ and $d_i^-\geq d_i^+$ (note that we would call such node $i$ ``positive good point"), the vector obtained by moving $i$ into  $D_+$ satisfies $\median(\vec y)=\{1\}$, where $\vec y:=\vec 1_{S_i,(S_i)^c}$, $S_i:=D_+\bigcup \{i\}$. 
   Let the objective function $f:\mathbb{R}^n\backslash\{\vec 0\}\rightarrow \mathbb{R}$ (corresponding to the eigenproblem) be defined via 
 		\begin{equation}
 		   f(\vec u)=\frac{I(\vec u)}{2\vol(V)\norm{\vec u}-N(\vec u)}.
 		\end{equation}
 		Suppose $0\in\median(\vec y)$. Then by \eqref{eq:antiCheeger-I(x)}, one has:
 		\begin{equation*}
 			\begin{aligned}
 				f(\vec y)&=\frac{2\abs{E(D_+,D_-)}+\abs{E(D_0,D_+\cup D_-)}-d_i^++d_i^-+d_i^0}{\vol(V)+\vol(D_0)-d_i}\leq f(\vec 1_{D_+,(D_+)^c})\\
 				&=\mu=\frac{2\abs{E(D_+,D_-)}+\abs{E(D_0,D_+\cup D_-)}}{\vol(V)+\vol(D_0)}.
 			\end{aligned}
 		\end{equation*}
 		Since $-d_i^++d_i^-+d_i^0\geq 0$ and $d_i>0$, there holds $f(\vec y)>\mu$, a contradiction.
 		Therefore, 
 		\begin{equation}
 			\label{eq:ach-mumax-3}
 			\begin{aligned}
 				f(\vec y)&=\frac{2\abs{E(D_+,D_-)}+\abs{E(D_0,D_+\cup D_-)}-d_i^++d_i^-+d_i^0}{2\vol(D_+)+\vol(D_0)+d_i}\leq f(\vec 1_{D_+,(D_+)^c})\\
 				&=\mu=\frac{2\abs{E(D_+,D_-)}+\abs{E(D_0,D_+\cup D_-)}}{\vol(V)+\vol(D_0)}.
 			\end{aligned}
 		\end{equation}
        If  $2\vol(D_+)+\vol(D_0)+d_i<\vol(V)+\vol(D_0)$, then $f(\vec y)>\mu$, a contradiction. So, if \eqref{eq:ach-mumax-3} holds, we have:
        \begin{equation}
        	\label{eq:ach-mumax-3-1}
        	-d_i^++d_i^-+d_i^0\leq\mu(-\vol(V)+2\vol(D_+)+d_i).
        \end{equation}
 		 By $v_i\in d_i\sgn(x_i-c_x)$ and the 1st equality in \eqref{eigen-anticheeger}, it is possible to give an upper bound for $v_i$ via
 		\begin{equation*}
 			\begin{aligned}
 			v_i&=-\sum\limits_{j\in D_0\backslash\{i\}}v_j+\vol(D_-)-\vol(D_+)\leq\min\{d_i,\vol(D_0)-d_i+\vol(D_-)-\vol(D_+)\}\\
 			&=\vol(D_0)-d_i+\vol(D_-)-\vol(D_+),
 			\end{aligned}
 		\end{equation*}
 		where the last equality is due to $1\in\median(\vec y)$. 
Then by the 2nd equality in \eqref{eigen-anticheeger}, we have:
 		\begin{equation*}
 			d_i^+-d_i^-=\sum\limits_{j\in D_0,\{i,j\}\in E} z_{ij}+\mu v_i\leq d_i^0+\mu(\vol(D_0)-d_i+\vol(D_-)-\vol(D_+)).
 		\end{equation*}
 		Combining with \eqref{eq:ach-mumax-3-1}, we get 
 		\begin{equation}
 				\label{eq:ach-mumax-3-2}
 			\mu d_i=-d_i^++d_i^-+d_i^0+\mu(\vol(V)-2\vol(D_+)),\Rightarrow z_{ij}=1,\,\forall\,j\in D_0,\{i,j\}\in E,
 		\end{equation}
 		\begin{equation}
 			\label{eq:ach-mumax-3-3}
 			v_i=\vol(D_0)-d_i+\vol(D_-)-\vol(D_+),\quad v_j=d_j,\,\forall j\in D_0\backslash\{i\}.
 		\end{equation}
 		
 		Similarly, for any $i\in D_0$ with $d_i^+> d_i^-$ (we call such $i$ ``negative good node"), the vector obtained by moving $i$ into 
 $D_-$ satisfies $\median(\vec u)=\{-1\}$, where $\vec u=\vec 1_{S_i,(S_i)^c}$, $S_i=D_-\bigcup \{i\}$. In consequence, 
 		\begin{equation}
 			\label{eq:ach-mumax-4}
 			\begin{aligned}
 				f(\vec u)&=\frac{2\abs{E(D_+,D_-)}+\abs{E(D_0,D_+\cup D_-)}-d_i^-+d_i^++d_i^0}{2\vol(D_-)+\vol(D_0)+d_i}\leq f(\vec 1_{D_-,(D_-)^c})\\
 				&=\mu=\frac{2\abs{E(D_+,D_-)}+\abs{E(D_0,D_+\cup D_-)}}{\vol(V)+\vol(D_0)}.
 			\end{aligned}
 		\end{equation}
If $2\vol(D_-)+\vol(D_0)+d_i<\vol(V)+\vol(D_0)$, we have $f(\vec u)>\mu$, a contradiction. So, if \eqref{eq:ach-mumax-4} holds, we have:
 		\begin{equation}
 			\label{eq:ach-mumax-4-1}
 			-d_i^-+d_i^++d_i^0\leq\mu(-\vol(V)+2\vol(D_-)+d_i).
 		\end{equation}
By $v_i\in d_i\sgn(x_i-c_x)$ and the first equality in \eqref{eigen-anticheeger}, we get an upper bound of $v_i$:
 		\begin{equation*}
 			\begin{aligned}
 				v_i&=-\sum\limits_{j\in D_0\backslash\{i\}}v_j+\vol(D_-)-\vol(D_+)\geq\max\{-d_i,d_i-\vol(D_0)+\vol(D_-)-\vol(D_+)\}\\
 				&=-\vol(D_0)+d_i+\vol(D_-)-\vol(D_+),
 			\end{aligned}
 		\end{equation*}
 		where the last equality is due to $-1\in\median(\vec u)$. By the 2nd equality in \eqref{eigen-anticheeger}, there holds
 		\begin{equation*}
 			d_i^+-d_i^-=\sum\limits_{j\in D_0,\{i,j\}\in E} z_{ij}+\mu v_i\geq -d_i^0+\mu(-\vol(D_0)+d_i+\vol(D_-)-\vol(D_+)).
 		\end{equation*}
 Combining it with \eqref{eq:ach-mumax-4-1}, we obtain 
 		\begin{equation}
 			\label{eq:ach-mumax-4-2}
 			\mu d_i=-d_i^-+d_i^++d_i^0+\mu(\vol(V)-2\vol(D_-)),\Rightarrow z_{ij}=-1,\,\forall\,j\in D_0,\{i,j\}\in E,
 		\end{equation}
 		\begin{equation}
 			\label{eq:ach-mumax-4-3}
 			v_i=-\vol(D_0)+d_i+\vol(D_-)-\vol(D_+),\quad v_j=-d_j,\,\forall j\in D_0\backslash\{i\}.
 		\end{equation}

Now we prove that $G_0$ has no edge. Suppose the contrary, that there exist  $i,j\in D_0$ such that $\{i,j\}\in E$. If $i$ and $j$ are ``positive good" nodes, the by \eqref{eq:ach-mumax-3-2}, $z_{ij}=z_{ji}=1$, which contradicts to $z_{ij}=-z_{ij}$. If $i$ and  $j$ are ``negative good" nodes, then by \eqref{eq:ach-mumax-4-2}, $z_{ij}=z_{ji}=-1$, which is a contradiction to $z_{ij}=-z_{ij}$. If $i$ and $j$ are ``positive good" and ``negative good" nodes, respectively. Without loss of generality, we may assume that 
 $i$ is ``positive good", and $j$ is ``negative good". Then, by \eqref{eq:ach-mumax-3-3} and \eqref{eq:ach-mumax-4-3}, we have
 		\begin{equation*}
 			\begin{aligned}
 				v_i&=\vol(D_0)-d_i+\vol(D_-)-\vol(D_+)=-d_i,\quad&\Rightarrow \vol(D_0)=\vol(D_+)-\vol(D_-),\\
 				v_j&=d_j-\vol(D_0)+\vol(D_-)-\vol(D_+)=d_j,\quad&\Rightarrow \vol(D_0)=\vol(D_-)-\vol(D_+),
 			\end{aligned}
 		\end{equation*}
 		which deduces $\vol(D_0)=0$, a contradiction.
 	\end{enumerate}
 \end{proof}

 \begin{cor}
  	\label{cor:ach-mumax-1}
Suppose $G=(V,E)$ has no isolated node, and let $\mu$ be the largest eigenvalue, with $\vec x$  being its corresponding eigenvector. Let $c_x\in\median(\vec x)$. If $|c_x|<\norm{\vec x}$, then $|D_0|\leq 1$.
 \end{cor}
 
 \begin{proof}
This property is a corollary of the 3rd case of Theorem~\ref{thm:ach-single}. Suppose the contrary: if $|c_x|<\norm{\vec x}$ and  $|D_0|\geq 2$, then by the last part of the proof of  Theorem~\ref{thm:ach-single} (3),  the nodes in $D_0$ are all ``positive good" (or all ``negative good''). If they are all ``positive good'', then \eqref{eq:ach-mumax-3-3} yields
   	\begin{equation*}
   			v_i=d_i=\vol(D_0)-d_i+\vol(D_-)-\vol(D_+),\quad\forall\, i\in D_0.
   	\end{equation*}
   	Combing the 1st equality in \eqref{eigen-anticheeger}, we derive
   	\begin{equation}
   		\label{eq:ach-mumax-small-1}
   		\vol(D_-)-\vol(D_+)=\sum\limits_{i\in D_0}v_i=\vol(D_0)\quad \Rightarrow\quad d_i=\vol(D_0),
   	\end{equation}
which implies $|D_0|=1$, a contradiction. If they are all ``negative good'', then by \eqref{eq:ach-mumax-4-3}, 
\begin{equation*}
	v_i=-d_i=-\vol(D_0)+d_i+\vol(D_-)-\vol(D_+),\quad\forall\, i\in D_0,
\end{equation*}
Combing the 1st equality in \eqref{eigen-anticheeger}, there holds
\begin{equation}
	\label{eq:ach-mumax-small-2}
	\vol(D_-)-\vol(D_+)=\sum\limits_{i\in D_0}v_i=-\vol(D_0)\quad \Rightarrow\quad d_i=\vol(D_0),
\end{equation}
meaning that $|D_0|=1$, a contradiction. The proof is then completed.
 \end{proof}

%

\vspace{0.3cm}

\textbf{Acknowledgement}. 
This research was supported by the National Key R~\&~D Program of China (No.~2022YFA1005102) and the National Natural Science Foundation of China (Nos.~ 12401443, 12325112, 12288101).

\bibliographystyle{amsplain}

\end{document}